\documentclass[10pt]{amsart}

\usepackage[utf8]{inputenc}
\usepackage[T1]{fontenc}

\usepackage{setspace}
\usepackage[margin=1.25in]{geometry}
\usepackage{parskip}
\usepackage[sc]{mathpazo}
\usepackage[T1]{eulervm}
\usepackage[scaled]{helvet}

\usepackage{amsaddr}

\usepackage{amssymb}
\usepackage{mathtools}
\usepackage{graphicx}
\usepackage[hidelinks]{hyperref}

\newcommand{\arxiv}[1]{\href{https://arxiv.org/abs/#1}{\texttt{ArXiv:#1}}}
\newcommand{\arxivmath}[1]{\href{https://arxiv.org/abs/math/#1}{\texttt{ArXiv:#1}}}
\newcommand{\arxivmaph}[1]{\href{https://arxiv.org/abs/math-ph/#1}{\texttt{ArXiv:#1}}}
\newcommand{\arxivcnma}[1]{\href{https://arxiv.org/abs/cond-mat/#1}{\texttt{ArXiv:#1}}}

\theoremstyle{plain}
\newtheorem{thm}{Theorem}

\newtheorem{lem}{Lemma}[section]
\newtheorem{prop}{Proposition}[section]

\theoremstyle{definition}
\newtheorem{defn}{Definition}[section]
\newtheorem{rem}{Remark}[section]

\numberwithin{equation}{section}

\newcommand{\pr}[1]{\mathbf{Pr}\left[#1\right]}
\newcommand{\ind}[1]{\mathbf{1}_{\{ #1 \}}}
\newcommand{\dt}[1]{\mathrm{det}\left(#1\right)}
\newcommand{\intz}[1]{\oint\limits_{\gamma_{\tau_{#1}}} d \zeta_{#1}}

\newcommand{\D}{\Delta}
\newcommand{\dr}{\nabla}
\newcommand{\vt}{\, | \,}
\newcommand{\veps}{\vec{\varepsilon}}
\newcommand{\vdel}{\vec{\delta}}
\newcommand{\eps}{\varepsilon}
\newcommand{\Z}{\mathbb{Z}}
\newcommand{\R}{\mathbb{R}}
\newcommand{\C}{\mathbb{C}}
\newcommand{\W}{\mathbb{W}}
\newcommand{\Ai}{\mathrm{Ai}}
\newcommand{\A}{\mathcal{A}}
\newcommand{\hb}{\mathcal{H}}
\newcommand{\G}{\mathcal{G}}
\newcommand{\sm}{\mathcal{S}}

\title{Multi-time distribution in discrete polynuclear growth}

\author{Kurt Johansson}
\address{
	Department of Mathematics,
	KTH Royal Institute of Technology,
	SE-100 44 Stockholm, Sweden}
\email{kurtj@kth.se}

\author{Mustazee Rahman}
\address{
	Department of Mathematics,
	KTH Royal Institute of Technology,
	SE-100 44 Stockholm, Sweden}
\email{mustazee@gmail.com}

\keywords{Airy kernel, exclusion processes, Fredholm determinant, KPZ universality,
	last-passage percolation, random growth, Tracy-Widom law.}

\subjclass[2000]{Primary 60F05, 60K35, 82C23, 82C24; Secondary 05E10, 15A15, 30E20.}

\thanks{Supported by the grant KAW 2015.0270 from the Knut and Alice Wallenberg Foundation
and grant 2015-04872 from the Swedish Science Research Council (VR)}

\date{}

\begin{document}

\begin{abstract}
\normalsize
We study the multi-time distribution in a discrete polynuclear growth model or, equivalently,
in directed last-passage percolation with geometric weights. A formula for the joint multi-time distribution
function is derived in the discrete setting. It takes the form of a multiple contour integral of a block Fredholm determinant.
The asymptotic multi-time distribution is then computed by taking the appropriate KPZ-scaling limit of this formula.
This distribution is expected to be universal for models in the Kardar-Parisi-Zhang universality class.
\end{abstract}

\maketitle

\newpage

\begin{doublespace}
\addtocontents{toc}{~\hfill \text{Page}\par}
\tableofcontents
\end{doublespace}

\newpage

\section{Introduction} \label{sec:intro}

Decorate points of $\Z^2$ with independent and identically distributed random weights $\omega(m,n)$
that are non-negative. Associated to this random environment is a growth function $\mathbold{G}$ as follows.
For every $m, n \geq 1$,
\begin{equation} \label{gmn}
\mathbold{G}(m,n) = \max \left \{ \, \mathbold{G}(m-1,n), \mathbold{G}(m,n-1) \, \right \} + \omega(m,n)
\end{equation}
with boundary conditions $\mathbold{G}(m,0) = \mathbold{G}(0,n) = 0$ for $m,n \geq 0$.
The function grows out from the corner of the first quadrant along up-right directions, so
it is a model of local random growth.

Consider weights chosen according to the geometric law: for some $0 < q < 1$,
\begin{equation*} \label{geom}
\pr{\omega(m,n) = k} = (1-q)q^{k} \quad \text{for}\;\; k \geq 0.
\end{equation*}

The subject of this article is the calculation, and then a derivation of the asymptotic value, of
the multi-point probability
\begin{equation} \label{ptime}
\pr{\mathbold{G}(m_1,n_1) < a_1, \mathbold{G}(m_2, n_2) < a_2, \ldots, \mathbold{G}(m_p,n_p) < a_p},
\end{equation} 
where $m_1 < m_2 < \cdots < m_p$ and $n_1 < n_2 < \cdots < n_p$.
In the asymptotic derivation the parameters $m,n$ and $a$ are scaled according to
Kardar-Parisi-Zhang (KPZ) scaling \cite{KPZ, Kr}. This means that for a large parameter $T$, the $m_k$s, $n_k$s
and $a_k$s are written (ignoring rounding) as
\begin{align} \label{KPZscaling}
	n_k  &= t_kT-c_1x_k(t_kT)^{\frac{2}{3}}, \\ \nonumber
	m_k &= t_k T+c_1 x_k(t_kT)^{\frac{2}{3}}, \\ \nonumber
	a_k  &= c_2t_kT+c_3\xi_k(t_kT)^{\frac{1}{3}}.
\end{align}
The $c_i$s are constants that depend on $q$ and will be specified in $\S$\ref{sec:theorems}.
They are determined from the macroscopic behaviour of $\mathbold{G}(m,n)$. The parameters above are
$0 < t_1 < t_2 < \cdots < t_p$, $x_1, x_2, \ldots, x_p \in \R$ and $\xi_1, \xi_2, \ldots, \xi_p \in \R$.
One is interested in the large $T$ limit of \eqref{ptime} with this scaling.

In Theorem \ref{thm:1} we provide the asymptotic distribution function of $\mathbold{G}$ under KPZ scaling \eqref{KPZscaling}.
Theorem \ref{thm:2} provides an expression for the distribution function \eqref{ptime}.
Theorem \ref{thm:1} is based on an asymptotical analysis of the latter. The calculations leading to Theorem \ref{thm:2},
contained in $\S$\ref{sec:discrete} and $\S$\ref{sec:fredholm}, should be more broadly applicable.

The probability \eqref{ptime} is expressed in terms of a $(p-1)$-fold contour integral of a
Fredholm determiant involving an $n_p \times n_p$ matrix with a $p \times p$ block structure.
This structure persists in the large $T$ limit, and the limiting multi-point probability is expressed by such
an integral of some Fredholm determinant over
$\hb = \underbrace{L^2(\R_{<0}) \oplus \cdots \oplus L^2(\R_{<0})}_{p-1} \oplus \, L^2(\R_{>0})$.
\smallskip 

\paragraph{\textbf{Interpretation as a growing interface and a non-equilibrium system}}
The growth model \eqref{gmn} has several interpretations. It can be seen as a randomly growing Young diagram,
or as a totally asymmetric exclusion process, or yet a directed last passage percolation model, also as
a kind of first passage percolation model (with non-positive weights), a system of queues in tandem,
and a type of random polymer at zero temperature. A natural interpretation is as a randomly growing
interface called discrete polynuclear growth, which we explain.

Rotating the first quadrant 45 degrees, define a function $h(x,t)$ by
\begin{equation*}
h(x,t)=\mathbold{G}\left(\frac{t+x+1}{2},\frac{t-x+1}{2} \right),
\end{equation*}
where $x+t$ is odd, $|x| < t$ and $h(x, 0) \equiv 0$. Extend $h(x,t)$ to $x \in \Z$ by linear interpolation.
Then \eqref{gmn} leads to the rule, see \cite{JoDPG}, that
$$h(x,t+1) = \max \, \{\, h(x-1,t), h(x,t), h(x+1,t) \, \} + \eta(x, t+1).$$
The $\eta(x,t)$ are independent and identically distributed with the geometric law if $x+t$ is odd and $|x| < t$,
and zero otherwise. This is an instance of the discrete polynuclear growth model, see \cite{KS}.
If we extend $h(x,t)$ to every $x \in \R$ by linear interpolation then $h(x,t)$ can be thought of as the height
above $x$ at time $t$ of a randomly growing interface.

Theorem \ref{thm:1} considers the re-scaled process
\begin{equation} \label{hxt}
\mathbold{H}_T(x,t) = \frac{h(2c_1 x (tT)^{\frac{2}{3}},\, 2tT) - c_2tT}{c_3(tT)^{\frac{1}{3}}},
\end{equation}
and provides its joint distribution at the points $(x_1,t_1), \ldots, (x_p,t_p)$ in the large $T$ limit.
Since the times are distinct it does not provide all the asymptotic finite dimensional distributions of $\mathbold{H}_T$,
although those could be obtained by considering limits in the time parameters. There is in fact a limit function
$\mathbold{H}(x,t)$ that is continuous almost surely, see \cite{MQR}, which means that in principle the aforementioned
distributions do determine the law of $\mathbold{H}$. As can be seen from \eqref{KPZscaling} and \eqref{hxt}, we study 
time-like distributions of $\mathbold{H}_T$ in the $(1,1)$ direction of the $(m,n)$-plane. In other directions we expect
the distributions to become asymptotically independent since non-trivial spatial correlations only occur at a scale of $T^{2/3}$.
Therefore we look in the so called characteristic direction; see \cite{Fe} for further discussion on this.

By re-scaling variables in the kernel from Theorem \ref{thm:1} it may be seen that for every $\lambda > 0$,
$\mathbold{H}(x,\lambda t)$ has the same distribution as $\mathbold{H}(x,t)$ as functions of $x$ and $t$.
If we define $\mathbold{A}(x,t) = t^{1/3} \mathbold{H}(t^{-2/3}x,t) + t^{-1}x^2$, then this means that
$$ \lambda^{-\frac{1}{3}} \cdot \mathbold{A}(\lambda^{\frac{2}{3}} x, \, \lambda t) \overset{law}{=} \mathbold{A}(x,t).$$
The relation above is known as KPZ scale invariance which, in this context, makes the polynuclear growth model a part
of the KPZ universality class. The latter is a collection of 1+1 dimensional statistical mechanical
systems whose fluctuations demonstrate the scale invariance above. Within the KPZ universality class lies the $\mathrm{Airy}_2$
process (see \cite{CH, JoDPG, PS} for reference), which represents asymptotic spatial fluctuations in $x$ of the height function
at a fixed time $t$. So $\mathbold{A}(x,t)$ may be thought of as the space-time surface sketched out by a growing Airy interface.
Some surveys that discuss these topics in depth are \cite{BG, CoKPZ, FSW, QuKPZ}, and \cite{SeCG} is a nice introduction to the growth model.

The papers \cite{BaLi, BGW, DOV, Hm, MQR} have recently studied various aspects of limit distributions in the KPZ universality class.
Here we find for the first time a full multi-time distribution function in the KPZ-scaling limit. A multi-time distribution function is actually derived in \cite{BaLi}
for the related continuous time TASEP in a periodic setting, and the asymptotic limit is computed in the relaxation time-scale,
when the TASEP is affected by the finite geometry. It is not obvious how to get the asymptotic result of the present paper from theirs,
since it means computing asymptotics in a situation where the TASEP is not affected by the finite geometry. However, after the completion of this work, the paper \cite{Liu} derived the multi-time distribution for the continuous time TASEP in the infinite geometry.
The relation between the formulas before the limit in \cite{BaLi, Liu} and the one in this paper is not clear so far.

The present paper generalizes previous work on the two-time distribution in \cite{JoTwo}. The two-time distribution has also been
investigated in the theoretical physics literature, see \cite{NarDou, NaDoTa, dNLD} and references there. Moreover, correlation function
of the two-time distribution has been studied in \cite{BaGa, FO}. The distribution of this growth model under a different asymptotic scaling,
related to the slow decorrelation phenomenon, has also been explored in \cite{BFS, CFP, Fe, IS}. Finally, see the paper \cite{ST} for some nice experimental work involving growth interfaces in liquid crystal.
\medskip

\paragraph{\textbf{Additional remarks}}
The formula for the limiting distribution function for $\mathbold{H}(x,t)$ in Theorem \ref{thm:1} is rather complicated.
It is built from kernels given by compositions of Airy functions, which thus generalizes the
Airy kernel. In the two-time case it is possible to rewrite the formula in such a way that the limits $t_2/t_1\to 1$ and $t_2/t_1\to \infty$ may be studied in detail,
see \cite{JoLim}. It would be interesting to do the same for the Fredholm determinant in Theorem \ref{thm:1}, so that these types of limits 
can be analyzed in the multi-time case as well. The distribution can in fact be computed numerically starting from the formula in Theorem \ref{thm:1}
in the two-time case, see \cite{DNJLD}, which shows that although complicated the formula is useful nonetheless.

In this paper we study the case of geometrically distributed weights $\omega(m,n)$. The case of exponentially distributed weights can be obtained 
by taking the appropriate limit ($q\to 1$) in the discrete formula. Similarly, the Brownian directed polymer model can be obtained as a limit. The asymptotic 
analysis is completely analogous. We expect the limiting multi-time formula in Theorem \ref{thm:1} to be universal within a large class of
models. It should be possible to study the limit of Poissonian last-passage percolation (Poissonized Plancherel) ($q\to 0$) from our formula in Theorem \ref{thm:2},
but this would entail taking a limit to an infinite Fredholm determinant before the large time asymptotics are computed.

\section{Statement of results} \label{sec:theorems}

In order to state the theorems we have to introduce notation. There is quite a bit of
notation throughout the article, so in the following, we introduce notation for both statement
of theorems and ones that recur.

\subsection{Some notation and conventions} \label{sec:notation}

Consider times $0<t_1<t_2 < \cdots < t_p$, points $x_1,x_2, \ldots x_p \in\R$ and
$\xi_1,\xi_2, \ldots, \xi_p \in \R$. Introduce the scaling constants
\begin{equation}\label{scalingconstants}
c_0=q^{-\frac{1}{3}}(1+\sqrt{q})^{\frac{1}{3}},\quad c_1=q^{-\frac{1}{6}}(1+\sqrt{q})^{\frac{2}{3}},\quad
c_2=\frac{2\sqrt{q}}{1-\sqrt{q}},\quad c_3=\frac{q^{\frac{1}{6}}(1+\sqrt{q})^{\frac{1}{3}}}{1-\sqrt{q}},
\end{equation}
where $q$ is the parameter of the geometric distribution. We will investigate the asymptotics
of the probability distribution given by \eqref{ptime} under the scaling \eqref{KPZscaling}.

\paragraph{\textbf{Delta notation}}
For integers $0 \leq k_1 < k_2 \leq p$, and $y$ being $m, n$ or $a$ from \eqref{KPZscaling}, define
\begin{equation} \label{delta}
\D_{k_1, k_2} y = y_{k_2} - y_{k_1} \quad \text{and} \;\; \D_k y = y_k - y_{k-1}.
\end{equation}
Also, define
\begin{align}\label{deltanotation}
\D_{k_1, k_2} t &= t_{k_2} - t_{k_1} \quad \text{and}\;\; \D_{k} t= t_k - t_{k-1}, \\ \nonumber
\D_{k_1, k_2} x & = x_{k_2} \, \Big (\frac{t_{k_2}}{\D_{k_1, k_2}t} \Big)^{\frac{2}{3}} - x_{k_1} \, \Big(\frac{t_{k_1}}{\D_{k_1, k_2}t} \Big)^{\frac{2}{3}}
\quad \text{and}\;\; \D_k x = \D_{k-1,k} x\, , \\ \nonumber
\D_{k_1, k_2} \xi & = \xi_{k_2} \, \Big (\frac{t_{k_2}}{\D_{k_1, k_2}t} \Big)^{\frac{1}{3}} - \xi_{k_1} \, \Big(\frac{t_{k_1}}{\D_{k_1, k_2}t} \Big)^{\frac{1}{3}}
\quad \text{and}\;\;  \D_k \xi = \D_{k-1,k} \xi \, .
\end{align}
By convention, $y_0 = 0$ for $y = n,m,a,t,x,\xi$. To understand \eqref{deltanotation} note that it is such that
$\D_{k_1, k_2} n = (\D_{k_1, k_2}t) T - c_1 \D_{k_1, k_2}x \, (\D_{k_1, k_2}t \, T)^{\frac{2}{3}}$,
and similarly for the differences between $m_k$s and $a_k$s. We will also use the shorthand
\begin{equation*}
\D_{k_1,k_2} (y^1, \ldots, y^{\ell}) = (\D_{k_1,k_2}y^1, \ldots, \D_{k_1,k_2}y^{\ell}) \quad \text{and}\;\;
\D_k (y^1, \ldots, y^{\ell}) = (\D_k y^1, \ldots, \D_k y^{\ell}).
\end{equation*}

\paragraph{\textbf{Block notation}}
The matrices that appear will have a $p \times p$ block structure
with the rows and columns partitioned according to
$$\{ 1, 2, \ldots, n_p \} = (0,n_1] \cup (n_1, n_2] \cup \cdots \cup (n_{p-1}, n_p].$$
The following notation will help us with calculations that depend on this structure.
For $y = m, n, a$, set
\begin{align} \label{blocknot}
y(k) &= y_{\min \{r, p-1\}} \;\;\;  \text{if} \;\; k \in (n_{r-1},n_r], \\ \nonumber
r^{*} &= \min \{r,p-1\} \;\;\; \text{if}\;\; 1 \leq r \leq p. 
\end{align}
For an $n_p \times n_p$ matrix $M$, $1 \leq i,j \leq n_p$ and $1 \leq r,s \leq p$, write
\begin{equation*}
	M(r,i ; s,j) = \mathbf{1}_{\big\{ i \in (n_{r-1}, \, n_r], \; j \in (n_{s-1}, \, n_s] \big\} } \cdot M(i,j).
\end{equation*}
This is the $p \times p$ block structure of $M$ according to the partition of rows and columns above.

Suppose $1 \leq i \leq n_p$. For $\vec{\eps} = (\eps_1, \ldots, \eps_{p-1}) \in \{1,2\}^{p-1}$ and
$\theta = (\theta_1, \ldots, \theta_{p-1}) \in (\C \setminus 0)^{p-1}$, define the following quantities.
\begin{align} \label{thetaeps}
	\theta^{\veps}(i) &= \prod_{k=1}^{p-1} \theta_k^{2-\eps_k - \ind{i \leq n_k}}, \\ \nonumber
	\theta(r \vt \veps) &= \prod_{k=1}^{r-1} \theta_k^{2-\eps_k} \, \prod_{k=r}^{p-1} \theta_k^{1-\eps_k} \quad \text{for}\;\; 1 \leq r \leq p.
\end{align}
Observe that $\theta^{\veps}(i) = \theta(r \vt \veps)$ for every $i \in (n_{r-1}, n_r]$, so these are block functions.
Particularly notable $\veps$ will be
$$ \eps^{k} = (\overbrace{2, \ldots,2}^{k-1}, 1, \ldots, 1) \quad \text{for}\;\; 1 \leq k \leq p.$$
For these we define
\begin{equation} \label{theta}
\Theta(r \vt k) = \theta(r \vt \eps^k) - (1- \ind{r=p, k=p-2}) \cdot \theta(r \vt \eps^{k+1}) \quad \text{for}\;\; 1 \leq k < \min \{r,p-1\}, \;\; 1 \leq r \leq p.
\end{equation}
We may set $\Theta(r \vt k)$ to be zero otherwise. Let us also set
$$ (-1)^{\eps_{[k_1,k_2]}} = (-1)^{\sum_{k = \max \{1,k_1\}}^{\min \{k_2,p-1\}} \, \eps_k} \quad \text{for}\;\; 0 \leq k_1 < k_2 \leq p.$$
It will be convenient to write  $(-1)^{\eps_{[k_1,k_2]}} \cdot (-1)^x$ as $(-1)^{\eps_{[k_1, k_2]}+x}$.

Define also the indicators functions
\begin{equation} \label{chi}
\chi_{\eps}(x) =
\begin{cases}
\ind{x < 0} & \text{if}\;\; \eps \equiv 1 \mod 2,\\
\ind{x \geq 0} & \text{if}\;\; \eps \equiv 2 \mod 2.
\end{cases}
\end{equation}

\paragraph{\textbf{Complex integrands}}
Define, for $n,m,a \in \Z$ and $w \in \C \setminus \{0,1-q,1\}$,
\begin{equation}\label{gnmx}
G^{*}(w \vt n,m,a)=\frac{w^n(1-w)^{a+m}}{\big(1-\frac{w}{1-q}\big)^m},
\end{equation}
as well as the function
\begin{equation} \label{Gnmx}
G(w \vt n,m,a) = \frac{G^{*}(w \vt n,m,a)}{G^{*} \big (1-\sqrt{q} \vt n,m,a \big)}.
\end{equation}
The number $ w_c = 1- \sqrt{q} $ is the critical point around which we will perform steepest descent analysis.
During the asymptotical analysis it will be convenient to write in terms of $G$ rather than $G^{*}$.
Consider also the following function $\G$ that will become the asymptotical value of $G$.
\begin{equation} \label{Gtxx}
\G(w \vt t, x, \xi) = \exp \Big \{ \, \frac{t}{3} w^3 + t^{\frac{2}{3}} xw^2 - t^{\frac{1}{3}} \xi w \, \Big \}
\quad \text{for}\; w \in \C \;\text{and}\; t,x,\xi \in \R.
\end{equation}

\paragraph{\textbf{Contour notation}}
We will always denote the contour integral
$$ \frac{1}{2 \pi \, \mathbold{i}} \int\limits_{\gamma} dz \quad  \text{as} \quad \oint\limits_{\gamma} dz\,.$$
There will be two types of contours in our calculations: circles and vertical lines.
Throughout, $\gamma_r$ denotes a circular contour around the origin of radius $r > 0$
with counterclockwise orientation. Also, $\gamma_r(1)$ is such a circular contour
around 1. A vertical contour through $d \in \R$ oriented upwards is denoted $\Gamma_d$.

\paragraph{\textbf{Conjugations}}
Throughout the article $\mu$ will denote a sufficiently large constant used with a conjugation factor.
Its value will depend only on the parameters $q$, $t_k$, $x_k$ and $\xi_k$. It will be convenient to
not state the value of $\mu$ explicitly, although in the upcoming theorem it suffices to consider
$$ \mu > \frac{\max_k \, \{x_k t_k^{2/3} \} - \min_k \,\{x_k t_k^{2/3} \}}{\min_k \, \{ (\D_k t)^{1/3} \} } \,.$$

Define, with $c_0$ given by \eqref{scalingconstants},
$$ \nu_T = c_0 T^{1/3}.$$
Let us introduce discrete conjugation factors, which will be needed for
asymptotical analysis. Recall $n(k), m(k)$ and $a(k)$ from \eqref{blocknot}. For $1 \leq k \leq n_p$,
\begin{equation}\label{conjugation}
c(k)=G^{*} \big (1 - \sqrt{q} \vt k, m(k), a(k) \big ) \cdot e^{\mu \frac{ (n(k)-k)}{\nu_T}}.
\end{equation}
Finally, set
\begin{equation} \label{conjugation2}
c(i,j) = \exp \big \{ \mu \frac{ (n(i)-i) - (n(j) - j) }{\nu_T} \big \}.
\end{equation}

\subsection{Statement of main theorem}
For $p \geq 1$ consider the Hilbert space
$$ \hb = \underbrace{L^2(\R_{<0}) \oplus \cdots \oplus L^2(\R_{<0})}_{p-1} \oplus \, L^2(\R_{>0}).$$
A kernel $F$ on $\hb$ has a $p \times p$ block structure, and we denote by $F(r,u; s,v)$ its $(r,s)$-block. So
$$ F(u,v) = \begin{bsmallmatrix}
F(1,u;\, 1,v)& \cdots&  F(1,u;\,p,v)\\
\vdots &  & \vdots \\
F(p,u;\,1,v)& \cdots  & F(p,u;\,p,v)
\end{bsmallmatrix}_{p \times p}\,. $$

Recall the function $\G$ from \eqref{Gtxx}, the notation $r^{*} = \min \{r,p-1\}$ and likewise $s^{*}$ from \eqref{blocknot}.

\begin{defn} \label{def:F}
The following basic matrix kernels over $\hb$ will constitute a final kernel.

(1) Let $d_1 > 0$ and $D > 0$. Define
\begin{align*}
& F[{\scriptstyle p \vt p}]\, (r,u;s,v) = \ind{r=p} \, e^{\mu(v-u)} \oint\limits_{\Gamma_{-d_1}} d \zeta_1 \oint\limits_{\Gamma_{D}} d z_p\,
\frac{\G \big(z_p \vt \D_p(t,x,\xi) \big) \, e^{\zeta_1 v - z_p u}}{\G \big( \zeta_1 \vt \D_{s^{*},p}(t,x,\xi)\big) \, (z_p-\zeta_1)}.
\end{align*}
Recall $\Gamma_{d}$ is a vertical contour oriented upwards that intersects the real axis at $d$.

(2) Let $0 < d_1 < d_2$. For $0 \leq k \leq p$, define
\begin{align*}
& F[{\scriptstyle k,k \vt \emptyset}]\, (r,u; s,v)	=  \ind{ s < k < r^{*}} \, e^{\mu(v-u)}\,
\oint\limits_{\Gamma_{-d_1}} d\zeta_1 \oint\limits_{\Gamma_{-d_2}} d\zeta_2 \\
& \frac{(\zeta_1 - \zeta_2)^{-1} \, e^{\zeta_2 v - \zeta_1 u}}
{\G \big(\zeta_1 \vt \D_{k,r^{*}}(t,x,\xi)\big) \, \G \big (\zeta_2 \vt \D_{s,k}(t,x,\xi) \big)}\,.
\end{align*}

(3) Let $0 < d_3 < d_2$ and $D > 0$. For $0 \leq k \leq p$, define
\begin{align*}
& F[{\scriptstyle p,k \vt p}]\, (r,u;s,v) = \ind{r=p, \, s < k < p} \, e^{\mu(v-u)}\,
\oint\limits_{\Gamma_{-d_2}} d\zeta_2 \oint\limits_{\Gamma_{-d_3}} d\zeta_3 \oint\limits_{\Gamma_{D}} dz_p \\
& \frac{\G \big(z_p \vt \D_p(t,x,\xi) \big) (z_p-\zeta_2)^{-1} (\zeta_2-\zeta_3)^{-1} e^{\zeta_3 v - z_p u}}
{\G \big(\zeta_2 \vt \D_{k,p}(t,x,\xi) \big) \G \big(\zeta_3 \vt \D_{s,k}(t,x,\xi) \big)}\,.
\end{align*}

(4) Let $0 < d_1,  d_3 < d_2$. For $0 \leq k_1,k_2 \leq p$, define
\begin{align*}
& F[{\scriptstyle k_1,k_1,k_2 \vt \emptyset}]\, (r,u;s,v) = \ind{k_1 < r^{*}, \, s < k_2 < k_1} \, e^{\mu(v-u)}\,
\oint\limits_{\Gamma_{-d_1}} d\zeta_1 \oint\limits_{\Gamma_{-d_2}} d\zeta_2 \oint\limits_{\Gamma_{-d_3}} d\zeta_3 \\
& \frac{(\zeta_1-\zeta_2)^{-1}\, (\zeta_2-\zeta_3)^{-1}\, e^{\zeta_3 v - \zeta_1 u}}
{\G \big(\zeta_1 \vt \D_{k_1,r^{*}}(t,x,\xi)\big) \, \G \big(\zeta_2 \vt \D_{k_2,k_1}(t,x,\xi) \big )\, \G \big(\zeta_3 \vt \D_{s,k_2} (t,x,\xi) \big)}\,.
\end{align*}

The upcoming kernels are determined in terms of integer parameters $0 \leq k_1 < k_2 \leq p$ and a vector parameter $\veps = (\eps_1, \ldots, \eps_{p-1}) \in \{1,2\}^{p-1}$.
Given $k_1$, $k_2$ and $\veps$, consider any set of distinct positive real numbers $D_k$ for integers $k \in (k_1, k_2]$ that
satisfy the following pairwise ordering:
\begin{equation} \label{Dorder}
D_k < D_{k+1}\;\; \text{if}\;\; \eps_k = 1\;\; \text{while}\;\; D_k > D_{k+1}\;\; \text{if}\;\; \eps_k = 2.
\end{equation}
It is easy to see, for instance by induction, that it is always possible to order distinct real numbers such that they satisfy
these constraints imposed by $\veps$. An explicit choice would be
$$D_1 = 2^p \quad \text{and} \quad D_{k+1} = D_k + (-1)^{\eps_k+1} \, 2^k.$$
Denote the contour
$$\vec{\Gamma}_{D^{\veps}} = \Gamma_{D_{k_1+1}} \times \cdots \times \Gamma_{D_{k_2}}.$$
(5)  Let $d_1 > 0$. Define
\begin{align*}
& F^{\veps}[{\scriptstyle k_1 \vt (k_1, k_2]}]\, (r,u; s,v) = \ind{k_1 < r^{*},\, s=k_2 < p,\, k_1<k_2}\, e^{\mu(v-u)}\,
\oint\limits_{\Gamma_{-d_1}} d\zeta_1 \, \oint\limits_{\vec{\Gamma}_{D^{\veps}}} dz_{k_1+1} \cdots dz_{k_2} \\
& \frac{\prod_{k_1 < k \leq k_2} \G \big(z_k \vt \D_k (t,x,\xi) \big)\, \prod_{k_1 < k < k_2} (z_k - z_{k+1})^{-1}\, e^{z_{k_2} v - \zeta_1 u}}
{\G \big(\zeta_1 \vt \D_{k_1,r^{*}} (t,x,\xi) \big) (z_{k_1+1}- \zeta_1)}.
\end{align*}

(6) Let $d_1, d_2 > 0$. Define
\begin{align*}
& F^{\veps}[{\scriptstyle k_1,k_2 \vt (k_1, k_2]}]\, (r,u;s,v) = \ind{k_1 < r^{*}, \, s^{*} < k_2,\, k_1<k_2}\, e^{\mu(v-u)}\,
\oint\limits_{\Gamma_{-d_1}} d\zeta_1 \oint\limits_{\Gamma_{-d_2}} d\zeta_2 \, \oint\limits_{\vec{\Gamma}_{D^{\veps}}} dz_{k_1+1} \cdots dz_{k_2} \\
& \frac{\prod_{k_1 < k \leq k_2} \G \big(z_k \vt \D_k (t,x,\xi)\big )\, \prod_{k_1 < k < k_2} (z_k - z_{k+1})^{-1}\, e^{\zeta_2 v - \zeta_1 u}}
{\G \big(\zeta_1 \vt \D_{k_1,r^{*}} (t,x,\xi) \big) \, \G \big(\zeta_2 \vt \D_{s^{*},k_2}(t,x,\xi) \big)\, (z_{k_1+1}-\zeta_1)\, (z_{k_2}-\zeta_2)}.
\end{align*}

(7) Let $0 < d_1,  d_3 < d_2$ and recall $k_1 < k_2$. Define
\begin{align*}
& F^{\veps}[{\scriptstyle k_1, k_2, k_3 \vt (k_1, k_2]}] (r,u;s,v) = \ind{k_1 < r^{*}, \, s < k_3 < k_2} e^{\mu(v-u)}
\oint\limits_{\Gamma_{-d_1}} d\zeta_1 \oint\limits_{\Gamma_{-d_2}} d\zeta_2 \oint\limits_{\Gamma_{-d_3}} d\zeta_3\,
\oint\limits_{\vec{\Gamma}_{D^{\veps}}} dz_{k_1+1} \cdots dz_{k_2} \\
&  \frac{\prod_{k_1 < k \leq k_2} \G \big(z_k \vt \D_k (t,x,\xi) \big)\, \prod_{k_1 < k < k_2} (z_k - z_{k+1})^{-1}\, (\zeta_2-\zeta_3)^{-1}\, e^{\zeta_3 v - \zeta_1 u}}
{\G \big(\zeta_1 \vt \D_{k_1,r^{*}}(t,x,\xi) \big) \, \G \big(\zeta_2 \vt \D_{k_3,k_2}(t,x,\xi) \big)\, \G \big(\zeta_3 \vt \D_{s,k_3} (t,x,\xi) \big ) (z_{k_1+1}-\zeta_1)\, (z_{k_2}-\zeta_2)}.
\end{align*}
\end{defn}
When the conjugation constant $\mu$ is sufficiently large these kernels decay rapidly to be of trace class,
which will be a byproduct of the proof of Theorem \ref{thm:1}. (Specifically, their entries are bounded by quantities
of the form $e^{-\widetilde{\mu}u} \Ai(-u) \, e^{\widetilde{\mu}v}\Ai(v)$ where $\Ai$ is the Airy function.)

Using these basic kernels we compose five other as weighted sums. Let $\theta_1, \ldots, \theta_{p-1}$ be
non-zero complex numbers and $\mathbold{\theta} = (\theta_1,\ldots, \theta_{p-1})$. Recall $\theta(r \vt \veps)$
and $\Theta(r \vt k)$ from \eqref{thetaeps} and \eqref{theta}, respectively. Define the following kernels over $\hb$.
\begin{align*}
F^{(0)}(r,u; s,v) & = \sum_{0 \leq k \leq p} (1+ \Theta(r \vt k)) \cdot (1 + \Theta(k \vt s)) \cdot F[{\scriptstyle k,k \vt \emptyset}]\, (r,u;s,v).\\
F^{(1)}(r,u; s,v) & = \sum_{0 \leq k \leq p} \Theta(r \vt k)\cdot F[{\scriptstyle k,k \vt \emptyset}]\, (r,u;s,v).\\
F^{(3)}(r,u;s,v) & = \sum_{0 \leq k_1, k_2 \leq p} \Theta(r \vt k_1) \cdot \big(1 + \Theta(k_2 \vt s)\big) \cdot F[{\scriptstyle k_1,k_1,k_2 \vt \emptyset}]\, (r,u;s,v).
\end{align*}

In the following, the variables $k_1, k_2, k_3 \in \{0,\ldots, p\}$ and $\veps \in \{1,2\}^{p-1}$. They satisfy
\begin{equation} \label{sumcond}
k_1 < k_2; \quad \text{given}\;\; k_1, k_2,\; \veps = (\overbrace{2, \ldots, 2}^{\substack{\eps_i = 2 \;\text{if} \\ i < \max \{k_1,1\}}},
\overbrace{\eps_{\max \{k_1,1\}}, \ldots, \eps_{\min\{k_2,p-1\}}}^{\text{arbitrary 1 or 2}},
\overbrace{1,\ldots, 1}^{\substack{\eps_i = 1 \;\text{if} \\ i > \min\{k_2,p-1\}}}).
\end{equation}
Recall the notation $(-1)^{\eps_{[k_1,k_2]}}$ following \eqref{theta}. Define
\begin{align*}
& F^{(2)}(r,u;s,v)  = \sum_{\substack{k_1, k_2,\, \veps \\ \text{satisfies}\; \eqref{sumcond}}}
(-1)^{\eps_{[k_1,k_2]} + \ind{k_2 =p}} \cdot \theta(r \vt \veps) \, \times \\
& \quad  \Big [  F^{\veps}[{\scriptstyle k_1 \vt (k_1,k_2]}] + F^{\veps}[{\scriptstyle k_1,k_2 \vt (k_1,k_2]}] +
\ind{k_1=p-1,\,k_2=p} F[{\scriptstyle p \vt p}] \Big ] (r,u;s,v). \\ \\
& F^{(4)}(r,u;s,v) = \sum_{\substack{k_1, k_2, k_3,\, \veps \\ \text{satisfies}\; \eqref{sumcond}}}
(-1)^{\eps_{[k_1,k_2]} + \ind{k_2 =p}} \cdot \theta(r \vt \veps) \, \times \\
& \Big [ \big(1+ \Theta(k_3 \vt s)\big) F^{\veps}[{\scriptstyle k_1,k_2,k_3 \vt (k_1,k_2]}]
-\ind{k_2=p,\, k_3=p-1} \big (1 + \Theta(p \vt s) \big) F^{\veps}[{\scriptstyle k_1, p,p-1 \vt (k_1, p]}] \; + \\
& \ind{k_2<p,\, k_3=p} \big(1+\Theta(k_2\vt s)\big ) F^{\veps}[{\scriptstyle k_1,k_2 \vt (k_1,k_2]}] \; + \\
& \ind{k_1=p-1,k_2=p} \big ( 1 + \Theta(k_3 \vt s)\big ) F[{\scriptstyle p,k_3 \vt p}] 
-\ind{k_1=p-1,k_2=p,k_3=p-1} \big (1 + \Theta(p \vt s)\big) F[{\scriptstyle p,p-1 \vt p}]\big) \Big] (r,u;s,v).
\end{align*}
Finally, define the kernel
\begin{equation} \label{F}
F(\mathbold{\theta}) = -F^{(0)} + F^{(1)} + F^{(2)} - F^{(3)} - F^{(4)}.
\end{equation}

\begin{thm} \label{thm:1}
Consider the function $\mathbold{G}(m,n)$ from \eqref{gmn}. Let $n_k, m_k$ and $a_k$ be scaled
according to \eqref{KPZscaling} with respect to parameters $T$, $t_k, x_k$ and $\xi_k$. Suppose $p \geq 2$.
Then,
\begin{align*}
\lim_{T \to \infty}\, & \pr{\mathbold{G}(m_1,n_1) < a_1, \, \ldots, \, \mathbold{G}(m_p,n_p) < a_p} = \\
& \oint\limits_{\gamma_r} d\theta_1 \cdots \oint\limits_{\gamma_r} d\theta_{p-1}\, \frac{\dt{I + F(\mathbold{\theta})}_{\hb}}{\prod_k (\theta_k - 1)}
\end{align*}
where $\gamma_r$ is a counter-clockwise circular contour around the origin of radius $r > 1$ and $F(\mathbold{\theta})$ is from \eqref{F}.
Moreover, the limit defines a consistent family of probability distribution functions.
\end{thm}

When $p=2$ this theorem agrees with the two-time distribution function from \cite{JoTwo}.
In this case the only non-zero component of $F(\mathbold{\theta})$ is $F^{(2)}$, whose non-zero basic
kernels are $F[{\scriptstyle 0 \vt (0,1]}]$, $F[{\scriptstyle 2 \vt 2}]$ and $F^{\eps}[{\scriptstyle 0,2 \vt (0,2]}]$ for $\eps = 1,2$.
Our other theorem that presents a similar expression for the probability \eqref{ptime} is stated as Theorem \ref{thm:2},
towards the end of $\S$\ref{sec:fredholm}.

\subsection{A discussion of results}
\smallskip

\paragraph{\textbf{Single point law}}
When $p=1$ there is a simpler approach for the single point limit as explained in $\S$\ref{sec:singletime},
where we express $\pr{\mathbold{G}(m,n) < a}$ as a Fredholm determinant of a matrix whose entries
are in terms of a double contour integral. More precisely, $\pr{\mathbold{G}(m,n) < a} = \dt{I+M}$ with
$$M(i,j) = \oint\limits_{\gamma_{\tau}} d\zeta \oint\limits_{\gamma_r(1)} dz\, \frac{G^{*}(z \vt n-i,m,a-1)}{G^{*}(\zeta \vt n-j+1,m,a-1) (z-\zeta)}.$$
Here $1 \leq i,j \leq n$ and the radii satisfy $\tau < 1-\sqrt{q} < 1-r < 1-q$.

An asymptotical analysis of it leads to
\begin{align} \label{singletime}
\lim_{T \to \infty}\, \pr{\mathbold{G}(m_1,n_1) < a_1} & = \dt{I - K}_{L^2(\R_{>0})},\; \text{where} \\ \nonumber
K(u,v) &= \oint\limits_{\Gamma_{-d}} d\zeta \oint\limits_{\Gamma_D} dz\, \frac{\G(z \vt t_1, x_1, \xi_1)}{\G(\zeta \vt t_1,x_1,\xi_1)}
\cdot \frac{e^{\zeta v - z u}}{z - \zeta}\,.
\end{align}
One may observe that
\begin{equation} \label{Gintegral}
\oint\limits_{\Gamma_D} dz\, \G(z \vt t,x,\xi) e^{-zu} = t^{-\frac{1}{3}}\,  e^{\frac{2}{3}x^3 + (\xi + t^{-\frac{1}{3}}u)x} \, \Ai(\xi + x^2 + t^{-\frac{1}{3}}u).
\end{equation}
Using this, as well as $\oint\limits_{\Gamma_{-d}} d\zeta \, \G(\zeta \vt t,x,\xi)^{-1} e^{\zeta v} =
\oint\limits_{\Gamma_d} d\zeta \, \G(\zeta \vt t,-x,\xi) e^{- \zeta v}$, and that $(z-\zeta)^{-1} = \int_0^{\infty} d \lambda\, e^{\lambda(\zeta - z)}$,
we find that
$$e^{x(v-u)} t^{\frac{1}{3}} K(t^{\frac{1}{3}} u, t^{\frac{1}{3}} v) = \int_0^{\infty} d\lambda\, \Ai(\xi + x^2 + u + \lambda) \Ai(\xi + x^2 + v + \lambda)
=K_{\mathrm{Ai}}(\xi + x^2 +u, \xi + x^2 + v).$$
This implies that $\dt{I-K}_{L^2(\R_{>0})}$ equals $F_{\mathrm{GUE}}(\xi + x^2)$, where $F_{\mathrm{GUE}}$
is the distribution function of the GUE Tracy-Widom law from \cite{TW}. The single point law recovers a result from \cite{JoSh}.
\smallskip

\paragraph{\textbf{Kernels expressed in terms of Airy function}}
The kernels in Definition \ref{def:F} may be written as products of more basic ones.
Consider the following kernel for $x,\xi \in \R$ and $t > 0$:
$$\A[t,x,\xi](u,v) = \oint\limits_{\Gamma_D}dw\, \G(w \vt t,x,\xi) e^{w(u-v)} =
t^{-\frac{1}{3}} \Ai \big(x^2 + \xi + t^{-\frac{1}{3}} (v-u)\big) \, e^{\frac{2}{3}x^2 + x(\xi +t^{-\frac{1}{3}}(v-u))}.$$

We will show how to write $F^{\veps}[{\scriptstyle k_1, k_2 \vt (k_1,k_2]}]$ using $\A$ and the others are done similarly.
Observe $(w_1-w_2)^{-1} = \int_{0}^{\infty} d\lambda \, e^{-\lambda(w_1-w_2)\cdot \mathrm{sgn} (\Re(w_1-w_2))}$.
As a result,
\begin{align*}
& (z_k-z_{k+1})^{-1} = \int_0^{\infty} d\lambda_k \, e^{\lambda_k (-1)^{\eps_k}\, (z_{k+1}-z_k)} \;\; \text{for}\;\; k_1 < k < k_2, \\
& (z_{k_1+1}-\zeta_1)^{-1} = \int_0^{\infty} d\lambda_{k_1} e^{\lambda_{k_1}(\zeta_1-z_{k_1+1})},\;\;
(z_{k_2}-\zeta_2)^{-1} = \int_0^{\infty} d\lambda_{k_2}\, e^{\lambda_{k_2}(\zeta_2 - z_{k_2})}.
\end{align*}
Let us set $\eps_{k_1} = 1$ and $\eps_{k_2} = 2$ in the following. Then we see that
\begin{align*}
& F^{\veps}[{\scriptstyle k_1, k_2 \vt (k_1,k_2]}](r,u;s,v) = \ind{k_1 < r^{*},\, s^{*} < k_2, \, k_1 < k_2} \, e^{\mu(v-u)}\,
\int_{[0,\infty)^{[k_1,k_2]}} \prod_{k_1 \leq k \leq k_2} d\lambda_k \\
& \oint\limits_{\Gamma_{-d_1}} d\zeta_1 \, \G \big(\zeta_1 \vt \D_{k_1, r^{*}} (t,x,\xi) \big)^{-1} e^{\zeta_1(\lambda_{k_1}-u)} \,
\oint\limits_{\Gamma_{-d_2}} d\zeta_2\,  \G \big (\zeta_2 \vt \D_{s^{*},k_2} (t,x,\xi) \big )^{-1} e^{\zeta_2(\lambda_{k_2}+v)} \\
& \prod_{k_1 < k \leq k_2} \, \oint\limits_{\Gamma_{D_k}} dz_k\, \G \big(z_k \vt \D_k (t,x,\xi) \big)\,
e^{z_k[(-1)^{\eps_{k-1}} \cdot \lambda_{k-1} - (-1)^{\eps_k} \cdot \lambda_k]}.
\end{align*}

We can evaluate the $\zeta$-integrals by changing variables $\zeta \to -\zeta$ as in the single time discussion.
Let us consider also the reflection $R$ for which $R\cdot K(u,v) = K(-u,v)$.
We have $K((-1)^{\eps}u, (-1)^{\eps'}v) = R^{\eps} K R^{\eps'}(u,v)$. Then we find that
\begin{align*}
& F^{\veps}[{\scriptstyle k_1, k_2 \vt (k_1,k_2]}](r,u;s,v) = \ind{k_1 < r^{*},\, s^{*} < k_2, \, k_1 < k_2} \, e^{\mu(v-u)}\,
\int_{[0,\infty)^{]k_1,k_2]}} \prod_{k_1 \leq k \leq k_2} d\lambda_k \\
& \A[\D_{k_1,r^{*}}t,\, - \D_{k_1,r^{*}}x,\, \D_{k_1,r^{*}} \xi ]\, (u,\lambda_{k_1})
\prod\limits_{k_1 < k \leq k_2} R^{\eps_{k-1}} \A[\D_k(t,x,\xi)] R^{\eps_k} \, (\lambda_{k-1}, \lambda_k) \, \times \\
& R\, \A[\D_{s^{*}, k_2}t, \,- \D_{s^{*}, k_2}x, \,\D_{s^{*}, k_2} \xi ]\, (\lambda_{k_2},v).
\end{align*}
We note that $R^{\eps} \chi_0 R^{\eps} = \chi_{\eps}$, where the latter is from \eqref{chi}. Therefore,
\begin{align*}
& F^{\veps}[{\scriptstyle k_1, k_2 \vt (k_1,k_2]}](r,u;s,v) = \ind{k_1 < r^{*},\, s^{*} < k_2, \, k_1 < k_2} \, e^{\mu(v-u)}\, 
 \A[\D_{k_1,r^{*}}t, \, - \D_{k_1,r^{*}}x, \, \D_{k_1,r^{*}} \xi ] \, \chi_0 R \, \times\\
& \prod_{k_1 < k < k_2} \A[\D_k(t,x,\xi)] \, \chi_{\eps_k} \; \A[\D_{k_2}(t,x,\xi)]\, 
\chi_0 R \, \A[\D_{s^{*}, k_2}t, \, - \D_{s^{*}, k_2}x, \, \D_{s^{*}, k_2} \xi ]\, (u,v).
\end{align*}

We now express all of the matrix kernels from Definition \ref{def:F} like above.
We will omit the conjugation factor $e^{\mu(v-u)}$ and the variables $u,v$ from these expressions.
Let us also use the shorthand $\D_{a,b} (t,-x,\xi) = (\D_{a,b} t, \, - \D_{a,b}x,\,  \D_{a,b} \xi)$. We then have the following.
\begin{align*}
(1) & \; F[{\scriptstyle p \vt p}] (r,s) = \ind{r=p}\, R \A[\D_p(t,x,\xi)] \, \chi_0 R \, \A[\D_{s^{*},p}(t,-x,\xi)].\\
(2) & \; F[{\scriptstyle k,k\vt \emptyset}](r;s) = \ind{s < k< r^{*}} \, \A[\D_{k,r^{*}} (t,-x,\xi)] \, \chi_1 \, \A[\D_{s,k}(t,-x,\xi)]. \\
(3) & \; F[{\scriptstyle p,k \vt p}](r,s) = \ind{r=p, \, s < k < p}\, R \A[\D_p(t,x,\xi)] \, \chi_0 R \, \A[\D_{k,p}(t,-x,\xi)] \, \chi_0\, \A[\D_{s,k}(t,-x,\xi)].\\
(4) & \; F[{\scriptstyle k_1,k_1,k_2 \vt \emptyset}](r;s) = \ind{k_1 < r^{*},\, s< k_2 <k_1}\,
\A[\D_{k_1,r^{*}} (t,-x,\xi)] \, \chi_1 \, \A[\D_{k_2,k_1} (t,-x,\xi)] \, \chi_0\,  \A[\D_{s,k_2} (t,-x,\xi)]. \\
(5) & \; F^{\veps}[{\scriptstyle k_1 \vt (k_1,k_2]}](r;s) = \ind{k_1 < r^{*},\, s =k_2 < p,\, k_1 < k_2}\, \A[\D_{k_1,r^{*}} (t,-x,\xi)] \, \chi_0 R \, \times \\
& \; \prod\limits_{k_1 < k < k_2} \A[\D_k(t,x,\xi)] \, \chi_{\eps_k}\, \A[\D_{k_2}(t,x,\xi)] \, R.\\
(7) & \; F^{\veps}[{\scriptstyle k_1,k_2,k_3 \vt (k_1,k_2]}](r;s) = \ind{k_1 < r^{*},\, s < k_3 < k_2, \, k_1 < k_2}\, \A[\D_{k_1,r^{*}} (t,-x,\xi)] \, \chi_0 R \, \times\\
&   \; \prod\limits_{k_1 < k < k_2} \A[\D_k(t,x,\xi)] \, \chi_{\eps_k}\,
\A[\D_{k_2}(t,x,\xi)] \, \chi_0 R\, \A[\D_{k_3,k_2}(t,-x,\xi)] \, \chi_0\, \A[\D_{s,k_2}(t,-x,\xi)].
\end{align*}

\section{Discrete considerations: multi-point distribution function} \label{sec:discrete}
In this section we derive a determinantal expression for the probability in \eqref{ptime}.
As $\mathbold{G}(m,n)$ depends only on the values of $\mathbold{G}$ to the left or below $(m,n)$,
the joint law of $\mathbold{G}(m_1,n_1), \ldots, \mathbold{G}(m_p,n_p)$ depends on the restriction
of $\mathbold{G}$ to $[0, m_p] \times [0,n_p]$.

Let us set $N = n_p$ throughout this section. Define the vector
$$ \mathbold{\vec{G}}(m) = \big (\mathbold{G}(m,1), \mathbold{G}(m,2), \ldots, \mathbold{G}(m,N) \big) \quad \text{for}\;\; m \geq 0.$$
The process $\mathbold{\vec{G}}(m)$ is a Markov chain by definition. It turns out to have an explicit transition rule.

\subsection{Markov transition rule} \label{sec:markov}
Let $\dr$ be the finite difference operator acting on $f: \Z \to \C$ as
\begin{equation} \label{eqn:derivative}\dr f(x) = f(x+1) - f(x).\end{equation}
The operator has as inverse given by
\begin{equation} \label{eqn:inverse}\dr^{-1}f (x) = \sum_{y < x} f(y),\end{equation}
valid so long as $f$ vanishes identically to the left of some integer. This will be the
case for functions that we consider. Since $\dr f$ and $\dr^{-1}f$ are then also functions of the same type,
we may consider integer powers of $\dr$ acting on such functions.

Define the negative binomial weight
$$ w_m(x) = \binom{x+m-1}{x} (1-q)^m q^x \, \ind{x \geq 0} \quad \text{for}\;\; m \geq 1 \;\;\text{and}\;\; x \in \Z.$$
This is the probability of observing the $m$-th head at $x+m$ tosses of a coin that lands heads with probability $1-q$.
It is a probability density, being the $(0,x)$-entry of $\big (I - \frac{q}{1-q} \dr \big)^{-m}$.

Define also
$$ \W_N = \{ (x_1, \ldots, x_N) \in \Z^N : x_1 \leq \cdots \leq x_N \},$$
noting that $\mathbold{\vec{G}}$ takes values in $\W_N$.

\begin{prop}\label{prop:transition}
	The process $\mathbold{\vec{G}}(m)$ is a Markov chain with transition rule
	\begin{equation}\label{Trans}
	\pr{ \mathbold{\vec{G}}(m)=\mathbold{y} \vt \mathbold{\vec{G}}(\ell)=\mathbold{x}}=\dt{\dr^{j-i}w_{m-\ell}(y_j-x_i)}_{i,j}
	\end{equation}
	for every $\mathbold{x},\mathbold{y}\in \W_N$ and $m > \ell$.
\end{prop}
The proposition is proved in \cite{JoMar} following the paper \cite{Warr} by Warren.
It is related to determinantal expressions for non-intersecting path probabilities
that appear in Karlin-McGregor or Lindstr\"{o}m-Gessel-Viennot type arguments.
The paths in this case are trajectories of the components of $\mathbold{\vec{G}}(m)$.
The transition matrix of this chain turns out to be intertwined with a Karlin-McGregor
type matrix by way of an RSK mechanism, which allows calculation of the former.
The papers \cite{DiWa, OC} also give a systematic exposition to such computations.

\begin{rem}
Formula \eqref{Trans} has very similar structure to Sch\"{u}tz type formulas \cite{BFPS, Sas, Sch} for the transition
rule of $\mathbold{\vec{G}}$. Sch\"{u}tz's formula for the $N$-particle continuous time TASEP $\mathbf{X}(t)$ is
$$ \pr{ \mathbold{X}(t) = \mathbold{y} \,|\, \mathbold{X}(0) = \mathbold{x}} = \dt{\dr^{j-i} F_t(\tilde{y}_j - \tilde{x}_i)}_{i,j}$$
where $F_t(x) = \frac{e^{-t}t^x}{x!} \ind{x \geq 0}$ is the Poisson density. Here the finite difference operator
$\dr$ means $\dr f(x)=f(x)-f(x+1)$, and its inverse is $\dr^{-1} f(x) = \sum_{y \geq x} f(y)$. Particle
locations are ordered such that $x_1 > x_2 > \cdots > x_N$, we let $\tilde{x}_j=x_{N+1-j}$, and likewise for $\mathbold{y}$.
	
A similar formula holds for the discrete time $N$-particle TASEP with sequential updates (see \cite{DiWa, RS}),
where the rightmost particle attempts to jump first with probability $q$, followed by the particle to its left, and so on.
The transition rule above is then modified by replacing $F_t(x)$ with the binomial density $F_{t,q}(x) = (1-q)^{-1} w_{t-x+1}(x)$.
With parallel updates, discrete time TASEP becomes equivalent to the discrete polynuclear growth model as explained,
for instance, in \cite{BFS, JoSh}.
\end{rem}

Denote $Pr$ the probability \eqref{ptime} that $\mathbold{G}(m_r,n_r) < a_r$ for every $r$.
By proposition \ref{prop:transition},
\begin{equation} \label{ptimemarkov} Pr = \sum_{\substack{x^1, \ldots, \, x^p \in \W_N \\ x^r_{n_r} < a_r }} \;
\prod_{r=1}^p \dt{ \dr^{j-i}w_{m_r - m_{r-1}}(x^r_j - x^{r-1}_i)}_{i,j}
\end{equation}
with the convention that $x^0 = 0$. We will drop subscripts $i,j$ from
the determinants since all of them will be of $N \times N$ matrices
with rows indexed by $i$ and columns by $j$.

\begin{lem} \label{lem:Prsimple}
Recall the $\Delta_k$ notation: $\Delta_k y = y_k - y_{k-1}$ for $y = n,m$.
The sum \eqref{ptimemarkov} simplifies to
\begin{align} \label{simpleptime}
Pr = \sum_{\substack{x^1, \ldots, \, x^{p-1} \in \W_N \\ x^r_{n_r} < a_r}} &
\dt{\dr^{n_1-i}w_{m_1}(x^1_j)} 
\prod_{r=2}^{p-1} \dt{ \dr^{\D_r n}w_{\D_r m}(x^r_j - x^{r-1}_i)} \times \\ \nonumber
& \times \dt{\dr^{j-1-n_{p-1}}w_{\D_p m}(a_p - x^{p-1}_i)}.
\end{align}
\end{lem}
Proving this is the subject of the next section.

\subsection{Summation by parts} \label{sec:sbp}
The following is Lemma 3.2 in \cite{JoTwo} and related to Lemma 3.2 in \cite{JoMar}.

\begin{lem} \label{lem:sbp}
	Let $f, g : \Z \to \C$ be such that $f(x) = g(x) = 0$ if $x  < L$ (typically $L$ is very negative).
	Let $a_i, b_i \in \Z$ for $i = 1, \ldots, N$, and consider $k$ such that $1 \leq k \leq N$. Then,
	\begin{align} \label{sbpa}
	 &\sum_{\substack{x \in \W_N \\ x_k < A}} \dt{\dr^{j-a_i} f(x_j-y_i)} \dt{ \dr^{b_j-i} g(z_j-x_i)} \\ \nonumber
	= & \sum_{\substack{x \in \W_N \\ x_k < A}} \dt{\dr^{k-a_i} f(x_j-y_i)} \dt{ \dr^{b_j-k} g(z_j-x_i)}.
	\end{align}
Moreover,
\begin{equation} \label{sbpb}
\sum_{\substack{z \in \W_N \\ z_N < A}} \dt{\dr^{j - a_i} g(z_j-x_i)} = \dt{\dr^{j-1-a_i}g(A-x_i)}.
\end{equation}
\end{lem}
It is instructive to understand the proof of this lemma, so we will outline
the argument. It should be contrasted with the approach in \cite{Sas}, see also \cite{BFPS},
which manipulates determinants by using that $\dr^{-1}$ is a summation operator.

\begin{proof}
	For identity \eqref{sbpb}, first note that $\sum_{x=a}^{b-1} \dr f(x) = f(b) - f(a)$. Now perform the summation
	from $z_N$ down to $z_1$, using multi-linearity of the determinant, which reduces $\dr$ by 1 in the corresponding column.
	After each step one finds a difference of two determinants, and the one with minus sign is zero
	due to two consecutive columns being equal. After the $z_1$-sum, the determinant with minus sign is zero
	because its first column stabilizes to zero as $z_1 \to -\infty$. For example, during the summation over $z_N$ we have
	\begin{align*}
	&\sum_{\substack{z \in \W_{N-1} \\ z_{N-1} < A}} \, \sum_{z_N =
	z_{N-1}}^{A-1} \dt{\dr^{1 - a_i} g(z_1-x_i) \cdots \dr^{N-1-a_i}g(z_{N-1}-x_i) \, \dr^{N-a_i}g(z_N-x_i)} =\\
	& \sum_{\substack{z \in \W_{N-1} \\ z_{N-1} < A}} \dt{\dr^{1 - a_i} g(z_1-x_i) \cdots \dr^{N-1-a_i}g(z_{N-1}-x_i) \, \dr^{N-1-a_i}g(A-x_i)} - \\
	& \qquad \dt{\dr^{1 - a_i} g(z_1-x_i) \cdots \dr^{N-1-a_i}g(z_{N-1}-x_i) \, \dr^{N-1-a_i}g(z_{N-1}-x_i)} = \\
	& \sum_{\substack{z \in \W_{N-1} \\ z_{N-1} < A}} \dt{\dr^{1 - a_i} g(z_1-x_i) \cdots \dr^{N-1-a_i}g(z_{N-1}-x_i) \, \dr^{N-1-a_i}g(A-x_i)}.
	\end{align*}
	
	Identity \eqref{sbpa} in based on the following idea. First, it is enough to establish it
	for the sum over $\{x \in \W_N: x_k = A\}$. Suppose $[a_{i,j}]$ is a square
	matrix, the $\ell$-th column of which has the from $a_{i,\ell} = \dr f_{i,\ell}(x_{\ell})$,
	and variable $x_{\ell}$ appears nowhere else. Then
	$\dt{a_{i,j}} = \dr_{\ell} \dt{a_{i,1} \, \cdots \, f_{i,\ell}(x_{\ell}) \, \cdots}$, where $\dr_{\ell}$
	is the difference operator in the $x_{\ell}$ variable. Now recall the summation by parts identity:
	$$ \sum_{x=a}^b u(x) \dr [v(-x)] = \sum_{x=a}^b \dr u(x) v(-x) + u(b+1)v(-b) - u(a)v(-a +1).$$
	
	Combining these we have the following. Suppose $c_j, d_j \in \Z$ are such that for
	an index $\ell > k$, $c_{\ell} = c_{\ell+1}$ if $\ell < N$ and $d_{\ell-1} = d_{\ell} -1$. Define
	$d^{-}_j = d_j - \ind{j = \ell}$ and $c^{-}_j = c_j - \ind{j = \ell}$. Then,
	\begin{align} \label{sbpc}
	& \sum_{\substack{x \in \W_N \\ x_k = A}} \dt{\dr^{d_j - a_i}f(x_j - y_i)} \dt{\dr^{b_j-c_i} g(z_j-x_i)} \\ \nonumber
	= & \sum_{\substack{x \in \W_N \\ x_k = A}} \dt{\dr^{d^{-}_j - a_i}f(x_j - y_i)} \dt{\dr^{b_j-c^{-}_i} g(z_j-x_i)}.
	\end{align}
	In plain words, one can move a derivative from column $\ell$ of the first determinant to the second's,
	decreasing $d_{\ell}$ and $c_{\ell}$ by 1 as a result. Indeed, consider the sum over variable $x_{\ell}$
	on the l.h.s.~of \eqref{sbpc} while holding the other variables fixed. Upon transposing the second matrix
	and using the aforementioned observations in order, we see that
	\begin{align*}
		& \sum_{x_{\ell} = x_{\ell -1}}^{x_{\ell+1}} \dt{\dr^{d_j - a_i}f(x_j - y_i)} \dt{\dr^{b_i-c_j} g(z_i-x_j)} \\
		= & \sum_{x_{\ell} = x_{\ell -1}}^{x_{\ell+1}} \dt{\dr^{d^{-}_j - a_i}f(x_j - y_i)} \dt{\dr^{b_j-c^{-}_i} g(z_j-x_i)} + (\text{boundary term}).
	\end{align*}
If $\ell = N$ then $x_{\ell+1} = + \infty$, and if $\ell = 1$ then $x_{\ell-1} = -\infty$. The boundary term equals $(I) - (II)$, where
\begin{align*}
	(I) &= \dt{\dr^{d^{-}_j - a_i} f(x_j-y_i)} \Big |_{x_{\ell} \coloneqq x_{\ell+1}+1} \, \cdot \, \dt{\dr^{b_i-c_j} g(z_i-x_j)} \Big |_{x_{\ell} \coloneqq x_{\ell+1}} \\
	(II) & = \dt{\dr^{d^{-}_j - a_i} f(x_j-y_i)} \Big |_{x_{\ell} \coloneqq x_{\ell-1}} \, \cdot \, \dt{\dr^{b_i-c_j} g(z_i-x_j)} \Big |_{x_{\ell} \coloneqq x_{\ell-1}-1}.
\end{align*}
	The term $(I) = 0$ because the $\ell$ and $(\ell + 1)$-st column of the second determinant agree due to $c_{\ell} = c_{\ell +1}$
	when $\ell < N$. If $\ell = N$ then it is 0 because $\dr^{m} g(z-x) = 0$ for all sufficiently large $x$, which makes the last
	column of the second determinant 0. The term $(II) = 0$ for the same reason with respect to the first determinant
	since $d_{\ell-1} = d_{\ell} - 1$.
	
	Analogously, for an $\ell < k$, suppose $c_{\ell+1} = c_{\ell} + 1$ and $d_{\ell} = d_{\ell - 1}$ if $\ell > 1$.
	Then we may move a derivative from the $\ell$-th column of the first determinant to the second's
	in the l.h.s. of \eqref{sbpc}, which will result in $c_{\ell}$ and $d_{\ell}$ being increased by 1.
	
	Identity \eqref{sbpa} follows by first applying \eqref{sbpc} to columns $\ell = N, N-1, \ldots, k+1$ \emph{in that order}.
	The conditions on $c_{\ell}$ and $d_{\ell}$ are then satisfied during each application. Then we apply \eqref{sbpc}
	to $\ell = N, \ldots, k+2$, followed by to $\ell = N, \ldots, k+3$, and so on. The derivative in column $j > k$
	is reduced by $j-k$. Similarly, we apply the derivative incrementing procedure first to columns $\ell = 1, \ldots, k-1$,
	then to columns $\ell = 1, \ldots, k-2$, and so forth to increase the derivative in column $j < k$ by $k-j$.
\end{proof}

\paragraph{\textbf{Proof of Lemma \ref{lem:Prsimple}}}
In order to simplify the expression for $Pr$ from \eqref{ptimemarkov} we apply Lemma \ref{lem:sbp} iteratively.
Apply \eqref{sbpa} to the expression \eqref{ptimemarkov} with respect to the sum over $x^1$,
which involves the first two determinants. In doing so, set $k = n_1, a_i = i, b_j = j, f = w_{n_1}$, etc.
We find that
\begin{align*}
Pr = \sum_{\substack{x^1, \ldots, \, x^p \in \W_N \\ x^r_{n_r} < a_r }} &
\dt{ \dr^{n_1-i}w_{m_1}(x^1_j)} \dt{ \dr^{j-n_1}w_{\D_2 m}(x^2_j - x^1_i)} \times \\
 & \times \prod_{r=3}^p \dt{ \dr^{j-i}w_{\D_r m}(x^r_j - x^{r-1}_i)}.
\end{align*}
Next, apply \eqref{sbpa} to the sum over $x^2$ -- involving the 2nd and 3rd determinants -- with
$k = n_2$ and $a_i \equiv n_1$. Then,
\begin{align*}
	Pr = \sum_{\substack{x^1, \ldots, \, x^p \in \W_N \\ x^r_{n_r} < a_r }} &
	\dt{ \dr^{n_1-i}w_{m_1}(x^1_j)} \dt{ \dr^{\D_2 n}w_{\D_2 m}(x^2_j - x^1_i)} \dt{ \dr^{j-n_2}w_{\D_3 m}(x^3_j - x^2_i)}  \\
	& \times \prod_{r=4}^p \dt{ \dr^{j-i}w_{\D_r m}(x^r_j - x^{r-1}_i)}.
\end{align*}
After iterating like this for all the variables, we finally use \eqref{sbpb} to perform the sum over $x^p$
with $x^p_{N} < a_p$ (recall $n_p = N$). This gives the expression \eqref{simpleptime}.
\qed

We would like to express $Pr$ as a single $N \times N$ determinant.
This would ordinarily be done by using the Cauchy-Binet identity iteratively over each of
the sums. However, the constraints $x^r_{n_r} < a_r$ prevent
a direct application. This is addressed in the following section.

\subsection{Cauchy-Binet identity} \label{sec:cb}
Let us manipulate the expression from  \eqref{simpleptime} in the following way.
First, consider $N \times N$ matrices $A = [a_{ij}]$ and $B = [b_{ij}]$ such that $\dt{A} \cdot \dt{B} = 1$.
In fact, we will chose $A$ and $B$ to be triangular with $a_{ii} = b_{ii}^{-1}$. We multiply the matrix of
the first determinant from \eqref{simpleptime} by $A$ and of the last one by $B$. Doing so will set
us up for the orthogonalization procedure of the next section.

Formally, introduce functions $f_{0,1}, f_{1,2}, \ldots, f_{p-1,p}$ as follows. We assume that $p \geq 2$.
When $p =1$ we can use a simpler approach as explained in $\S$\ref{sec:singletime}.
For $1 \leq i,j \leq N$ as well as $x,y \in \Z$,
\begin{align} \label{orthogf}
	f_{0,1}(i,x) &= \sum_{k=1}^N a_{ik} \, \dr^{n_1-k} w_{m_1}(x+a_1) \cdot (-1)^{n_1}, \\ \nonumber
	f_{r-1,r}(x,y) & = \dr^{\D_r n} w_{\D_r m}(y-x + \D_r a) \cdot (-1)^{\D_r n} \;\; \text{for}\; 1 < r < p, \\ \nonumber
	f_{p-1,p}(x,j) & = \sum_{k=1}^N \dr^{k-1-n_{p-1}} w_{\D_p m}(\D_p a - x) \, b_{kj} \cdot (-1)^{n_{p-1}}.
\end{align}

Then $Pr$ equals
\begin{equation} \label{fptime}
Pr = \sum_{\substack{x^1, \, \ldots, \, x^{p-1} \in \W_N \\ x^k_{n_k} < 0}} \;
\dt{f_{0,1}(i,x^1_j)} \prod_{k=2}^{p-1} \, \dt{f_{k-1,k}(x^{k-1}_i, x^k_j)} \dt{f_{p-1,p}(x^{p-1}_i,j)}.
\end{equation}
The summation constraints became $x^k_{n_k} < 0$ because we have shifted
$x^k_i \mapsto x^k_i + a_r$ in defining $f_{k-1,k}$. Also, the powers of $-1$
in the $f$s do not change the product of the determinants because they
factor out as $(-1)^{N \cdot(n_1 + \D_2n + \cdots + \D_{p-1}n + n_{p-1})} = (-1)^{2Nn_{p-1}}$.

Consider $\mathbold{\theta} = (\theta_1, \ldots, \theta_{p-1})$ where each $\theta_k \in \C \setminus 0$.
Define an $N \times N$  matrix $L = L(i,j \vt \mathbold{\theta})$ as follows with the convention that $\theta_k^0 = 1$.
\begin{equation} \label{matrixL}
L(i,j \vt \mathbold{\theta}) = \sum_{(x_1, \ldots, x_{p-1}) \in \Z^{p-1}} f_{0,1}(i,x_1)
\prod_{k=2}^{p-1} f_{k-1,k}(x_{k-1},x_k) \, f_{p-1,p}(x_{p-1},j)
\prod_{k=1}^{p-1} \theta_k^{\ind{x_k < 0} - \ind{i \leq n_k}}.
\end{equation}
The sum is actually finite because $f_{r-1,r}(x,y)$ vanishes for all sufficiently large $x$ or small $y$.
Apart from the factors involving $\theta$, $L$ is the convolution $f_{0,1} * \cdots * f_{p-1,p}$ or,
if we think of the $f$s as matrix kernels, then it is the product $f_{0,1} \cdots f_{p-1,p}$.
The conclusion of this section is
\begin{lem} \label{lem:CB}
	Let $\gamma_r$ be a counterclockwise circular contour of radius $r > 1$. Set
	$\gamma_r^{p-1} = \overbrace{\gamma_r \times \cdots \times \gamma_r}^{p-1}$.
	\begin{equation} \label{discreteformula}
		Pr = \oint\limits_{\gamma_r^{p-1}} d\theta_1 \cdots d \theta_{p-1} \; \frac{\dt{L(i,j \,|\, \mathbold{\theta})}}{\prod_{k=1}^{p-1} (\theta_k-1)} \,.
	\end{equation}
\end{lem}

\begin{proof}
For $x \in \W_N$, the condition $x_n < 0$ is equivalent to $\# \{x_j < 0\} \geq n$.
Now for $\ell \in \Z$,
$$ \ind{ \ell \geq 0} = \oint\limits_{\gamma_r} d\theta \, \frac{\theta^{\ell}}{\theta -1}\,.$$
Consequently,
\begin{equation} \label{constraintint}
\ind{ \# \, \{x_j < 0\} \, \geq n} = \oint\limits_{\gamma_r} d\theta \, \frac{\prod_{j=1}^N \theta^{\ind{x_j < 0}}}{\theta^n (\theta -1)}.
\end{equation}
If we apply \eqref{constraintint} to the expression \eqref{fptime} for $Pr$ we find
\begin{align*}
	Pr = \oint\limits_{\gamma_r^{p-1}} & d\theta_1 \cdots d \theta_{p-1} \, \prod_{k=1}^{p-1}\frac{\theta_k^{-n_k}}{\theta_k -1} \,
	\Bigg [  \sum_{\substack{x^k \in \W_n \\ 1 \leq k < p}} \dt{f_{0,1}(i,x^1_j) \theta_1^{\ind{x^1_j < 0}}} \times \\
	 & \prod_{k=2}^{p-1} \dt{f_{k-1,k}(x^{k-1}_i, x^k_j)\theta_k^{\ind{x^k_j < 0}}} \dt{f_{p-1,p}(x^{p-1}_i,j) } \Bigg ].
\end{align*}

We push $\theta_k^{-n_k}$ into the first determinant by inserting $\theta_k^{-1}$ into its first $n_k$
rows. Then, by the Cauchy-Binet identity, the quantity inside square brackets is $\dt{L(i,j \,| \mathbold{\theta})}$.
\end{proof}

Expression \eqref{discreteformula} is a discrete determinantal formula for the multi-point
distributions functions \eqref{ptime}. However, matrix $L$ does not have good asymptotical behaviour
for the KPZ scaling limit (or numerical estimates). It is necessary to express $\dt{L}$
as a Fredholm determinant over a space free of parameter $N$. This is the subject of
the following section.

\section{Orthogonalization: representation as a Fredholm determinant} \label{sec:fredholm}

Recall the triangular matrices $A$ and $B$ from $\S$\ref{sec:cb}. Multiplication by them is essentially performing
elementary row and column operations, which is an orthogonalization procedure. The entries of $A$ and $B$,
vaguely put, will be like inverses to entries of the first and last determinant in \eqref{simpleptime}. These are
obtained by extending $\dr^n w_m(x)$ to negative $m$, which motivates the following. Later in $\S$\ref{sec:singletime}
we provide intuition for this orthogonalization by explaining it for the single point law.

\subsection{Contour integrals} \label{sec:cint}
Recall the functions $G^{*}$ and $G$ from \eqref{gnmx} and \eqref{Gnmx}.
The 3-parameter family $G^{*}(\cdot \vt n,m,a)$ and $G(\cdot \vt n,m,a)$
form a group in that for $w \neq 0, 1-q, 1$:
\begin{align} \label{Hgroup}
 G^{*}(w \vt n+n', m+m', a+a') &= G^{*}(w \vt n,m,a) \cdot G^{*}(w\,|\, n',m',a'), \\ \nonumber
 G^{*}(w \vt -n, -m, -a) &= G^{*}(w \vt n,n,a)^{-1}, \\ \nonumber
 G^{*}(w \vt 0,0,0) = 1,
\end{align}
and analogously for $G$. The \emph{group property} will make it convenient to
follow upcoming calculations and give further intuition for the orthogonalization procedure.

From the generating function $(1+z)^{-k} = \sum_{x \geq 0} \binom{-k}{x} z^x$ for negative binomials, it follows that
$$ w_m(x) = \oint\limits_{\gamma_{\rho}} dz \, \Big(\frac{1-qz}{1-q} \Big)^{-m}z^{-x-1},$$
where $\rho < 1$. Changing variables $z \mapsto (1-z)^{-1}$ gives a contour
integral representation of $w_m(x)$ that, upon applying integer powers of $\dr$ according to
\eqref{eqn:derivative} and \eqref{eqn:inverse}, shows that
\begin{equation}\label{Deltakwm}
\dr^n w_m(x)=(-1)^{n-1}\oint\limits_{\gamma_r(1)}dz\, G^{*}(z \vt n,m,x-1)
\end{equation}
with radius $r>1$ (so $\gamma_r(1)$ encloses all possible poles at $z = 0,1-q,1$). The condition $r > 1$
ensures that the summation needed to apply $\dr^{-1}$ to $G^{*}(z \vt n,m,x)$ in the $x$-variable,
is legal throughout $z \in \gamma_r(1)$.
The right hand side of \eqref{Deltakwm} continues $\dr^n w_m(x)$ to integer values of all parameters.

Define the matrices $A$ and $B$ as follows. Let $c(k)$ be the conjugation factor defined in \eqref{conjugation},
and recall $m(k)$ and $a(k)$ from \eqref{blocknot}.
Consider any radius $\tau < 1-q$.
\begin{align}\label{ABdef}
a_{ik} & =c(i)(-1)^{k} \, \oint\limits_{\gamma_{\tau}} d\zeta \, \frac{1}{G^{*}\big (\zeta \vt i-k+1, m(i), a(i)-1\big)}\, , \\ \nonumber
b_{kj} & =c(j)^{-1}(-1)^k \, \oint\limits_{\gamma_{\tau}} d\zeta \, \frac{1}{G^{*}\big (\zeta \vt k-j+1, m_p - m(j), a_p - a(j)\big)}\, .
\end{align}
The matrices $A$ and $B$ are lower-triangular with $a_{ii} = c(i) (-1)^i = b_{ii}^{-1}$; so $\dt{A} \dt{B} = 1$.
This is because
\begin{equation*}
\oint\limits_{\gamma_{\tau}} d \zeta \, \frac{1}{G^{*}(\zeta \vt n+1, m, a)} = \begin{cases}
	1 & \text{if}\;\; n = 0 \\
	0 & \text{if}\;\; n < 0.
\end{cases}
\end{equation*}

\begin{lem} \label{lem:Gintegral}
The following identities hold.
\begin{enumerate}
	\item If $1 \leq i \leq N$ and $|z| > \tau$,
	\begin{equation*}
	\oint\limits_{\gamma_\tau} d\zeta\, \frac{1}{G^{*}(\zeta \vt i,m,a) \, (z-\zeta)}\, =
	\sum_{k=1}^N \, \oint\limits_{\gamma_{\tau}} d\zeta \, \frac{z^{-k}}{G^{*}(\zeta \vt i-k+1, m, a)}.
	\end{equation*}
	\item If $1 \leq j \leq N$ and $|z| > \tau$,
	\begin{equation*}
		\oint\limits_{\gamma_\tau} d\zeta\, \frac{z^{N+1}}{G^{*}(\zeta \vt N+1-j,m,a)\, (z-\zeta)}\, =
		\sum_{k=1}^N \, \oint\limits_{\gamma_{\tau}} d\zeta \, \frac{z^k}{G^{*}(\zeta \vt k-j+1, m, a)}.
	\end{equation*}
\end{enumerate}
\end{lem}
\begin{proof}
The first identity follows by expanding $(z-\zeta)^{-1}$ in powers of $\zeta/z$. The contribution
of terms on the r.h.s.~with $k > i$ is zero. The second one follows from the first by re-indexing
$k \mapsto N+1 -k$ and substituting $i = N+1-j$.
\end{proof}

For the rest of this section we will deduce an expression for $L(i,j\vt \mathbold{\theta})$ in terms of contour integrals.
Recalling the $f_{r-1,r}$s from \eqref{orthogf}, then \eqref{Deltakwm} and \eqref{ABdef}, we infer the
following.
\begin{align*}
	f_{0,1}(i,x_1) &= - c(i) \oint\limits_{\gamma_{\tau_1}} d\zeta_1 \oint\limits_{\gamma_{R_1}(1)} d z_1 \,
	\frac{G^{*}\big(z_1 \vt n_1, m_1, a_1 + x_1 -1\big)}{G^{*}\big(\zeta_1 \vt i, m(i), a(i)-1\big) \, (z_1 - \zeta_1)}\, , \\
	f_{r-1,r}(x_{r-1}, x_r) &= - \oint\limits_{\gamma_{R_r}(1)} d z_r \, G^{*}(z_r \vt \D_r n, \D_r m, \D_r a -1) \quad \text{for}\; 1 < r < p, \\
	f_{p-1,p}(x_{p-1},j) &= c(j)^{-1} \oint\limits_{\gamma_{\tau_2}} d\zeta_2 \oint\limits_{\gamma_{R_p}(1)} d z_p \,
	\frac{G^{*}\big(z_p \vt \D_p n, \D_p m, \D_p a -x_{p-1} -1\big)}{G^{*} \big(\zeta_2 \vt n_p -j +1, m_p-m(j), a_p-a(j)\big) \, (z_p - \zeta_2)}\, .
\end{align*}
The contours above are circular and arranged as follows. Contours $\gamma_{\tau_1}$ and $\gamma_{\tau_2}$
are around the origin with $\tau_2 < \tau_1 < 1-q$ ($\tau_1$ and $\tau_2$ are ordered for definiteness).
Contours $\gamma_{R_k}(1)$ are around 1 with every $R_k > 1+ \tau_1$, that is, they enclose the contours around
the origin and the numbers $0,1-q,1$. In deriving expressions for $f_{0,1}$ and $f_{p-1,p}$ we have used
Lemma \ref{lem:Gintegral}.

Upon multiplying all the $f$s we get $(-1)^{p-1} c(i)c(j)^{-1} \times \big(\text{a}\;(p+2)-\text{fold contour integral}\big)$.
In this integral we would like to replace every $G^{*}$ by the corresponding $G$. In doing so we obtain
factors of $G^{*}(1-\sqrt{q} \vt \cdot, \cdot, \cdot)$, which, by the group property of $G^{*}$, multiply to
$$G^{*}\big(1-\sqrt{q} \vt j-i-1, m(j)-m(i),a(j)-a(i)\big).$$
When multiplied by $c(i)c(j)^{-1}$ this equals $c(i,j)/ (1- \sqrt{q})$,
where $c(i,j)$ is the conjugation factor \eqref{conjugation2}.

We may plug the product above into the definition of $L(i,j \vt \mathbold{\theta})$ from \eqref{matrixL}.
There we have a sum over $\vec{x} \in \Z^{p-1}$ and a product involving $\mathbold{\theta}$. Let us write
the product of $\theta_k$s as follows, recalling $\chi_{1}(x) = \ind{x < 0}$ and $\chi_{2}(x) = \ind{x \geq 0}$ from \eqref{chi}.
Note $\theta^{\ind{x<0}} = \theta^{2-1} \, \chi_{1}(x) + \theta^{2-2} \, \chi_{2}(x)$. Therefore,
\begin{align*}
	\prod_{k=1}^{p-1} \theta_k^{\ind{x_k < 0} - \ind{i \leq n_k}} &= \sum_{\vec{\eps} \in \{1,2\}^{p-1}}
	\prod_{k=1}^{p-1} \theta_1^{2-\eps_k-\ind{i \leq n_k}} \cdot \chi_{\eps_1}(x_1) \cdots \chi_{\eps_{p-1}}(x_{p-1}) \\
	& =  \sum_{\vec{\eps} \in \{1,2\}^{p-1}} \theta^{\vec{\eps}}(i) \, \chi_{\vec{\eps}}(\vec{x}),
\end{align*}
where $\chi_{\vec{\eps}}(\vec{x}) = \prod_{k=1}^{p-1} \chi_{\eps_k}(x_k)$ and $\theta^{\veps}(i)$ is notation from \eqref{thetaeps}.
From this expression we find that
\begin{equation} \label{Lepssum}
L(i,j \vt \mathbold{\theta}) = \sum_{\vec{\eps} \in \{1,2\}^{p-1}}
\frac{(-1)^{p-1} c(i,j)}{1-\sqrt{q}} \, \theta^{\vec{\eps}}(i) \, L^{\vec{\eps}}(i,j),
\end{equation}
where $L^{\vec{\eps}}(i,j)$ is the sum over $x \in \Z^{p-1}$ of $\chi_{\vec{\eps}}(x)$ times
the aforementioned $(p+2)$-fold contour integral.

\begin{lem} \label{lem:Leps}
Given $\vec{\eps} = (\eps_1, \ldots, \eps_{p-1}) \in \{1,2\}^{p-1}$, $L^{\vec{\eps}}(i,j)$ has the following
contour integral form. Consider radii $\tau_2 < \tau_1  < 1 - q$, as well as radii $R_1, \ldots, R_{p}$
such that every $R_k > 1+\tau_1$ and they satisfy the following pairwise ordering.
\begin{equation} \label{radiicondition}
R_k < R_{k+1} \;\;\text{if}\;\; \eps_k = 2 \quad \text{while} \quad R_k > R_{k+1} \;\;\text{if}\;\; \eps_k = 1.
\end{equation}
There is such a choice of radii, and given these,
\begin{align*}
&L^{\vec{\eps}}(i,j) = (-1)^{\eps_1 + \cdots + \eps_{p-1}} \, \oint_{\gamma_{\tau_1}} d\zeta_1 \, \oint_{\gamma_{\tau_2}} d\zeta_2 \,
\oint_{\gamma_{R_1}(1)} dz_1 \cdots \oint_{\gamma_{R_p}(1)} dz_p \\
& \frac{\prod_{k=1}^p G \big(z_k \vt \D_k (n,m,a) \big) \prod_{k=1}^{p-1} \big(z_k - z_{k+1}\big)^{-1} \, \big (\frac{1-\zeta_1}{1-z_1} \big)}
{G\big(\zeta_1 \vt i, m(i), a(i) \big) \, G\big (\zeta_2 \vt n_p -j +1,m_p-m(j), a_p-a(j)\big) \, (z_1 - \zeta_1)\, (z_p - \zeta_2)}\,.
\end{align*}
\end{lem}

\begin{proof}
From the discussion preceeding the lemma we see that
\begin{align*}
& L^{\vec{\eps}}(i,j) = \sum_{(x_1, \ldots, x_{p-1}) \in \Z^{p-1}}
\oint_{\gamma_{\tau_1}} d\zeta_1 \, \oint_{\gamma_{\tau_2}} d\zeta_2 \, \oint_{\gamma_{R_1}(1)} dz_1 \cdots \oint_{\gamma_{R_p}(1)} dz_p \;
\chi_{\eps_1}(x_1) \cdots \chi_{\eps_{p-1}}(x_{p-1}) \\
& \frac{\prod_{k=1}^{p-1} G(z_k \vt \D_k n, \D_k m, \D_k a +\D_k x -1) \, G(z_p \vt \D_p n, \D_p m, \D_k a - x_{p-1} -1)}
{G\big(\zeta_1 \vt i, m(i), a(i) -1\big) \, G\big (\zeta_2 \vt n_p -j +1,\, m_p-m(j), \, a_p-a(j)\big) \, (z_1 - \zeta_1)\, (z_p - \zeta_2)}\,.
\end{align*}
From the group property, $G(z \vt n,m,a + x -1) = G(z \vt n,m,a) (1-z)^{x-1}$.
Using this, we factor out every $(1-z_k)^{\D_k x -1}$, $(1-z_p)^{-x_{p-1}-1}$ and $(1-\zeta_1)^{-1}$.
Their contribution is
$$ \prod_{k=1}^{p-1} \left (\frac{1-z_k}{1-z_{k+1}} \right )^{x_k} \; \cdot \; \frac{1- \zeta_1}{\prod_{k=1}^p (1-z_k)}\,.$$
Now suppose $z \in \gamma_{\rho_1}(1)$, $w \in \gamma_{\rho_2}(1)$ and $\eps \in \{1,2\}$. Then,
$$\sum_{x \in \Z} \chi_{\eps}(x)\, \left ( \frac{1-z}{1-w} \right )^x = (-1)^{\eps}\, \frac{1-w}{z-w},$$
so long as $\rho_1 < \rho_2$ in the case $\eps = 2$ or $\rho_1 > \rho_2$ in the case $\eps = 1$.
The radii $R_1, \ldots, R_p$ have been chosen precisely to satisfy these constraints imposed by $\vec{\eps}$.
That it is possible to do so may be seen by induction on $p$ as follows.

The base case of $p=2$ is trivial. Now suppose there is an arrangement of radii $R_1, \ldots, R_p$
that satisfy the constraints given by $\eps_1, \ldots, \eps_{p-1}$, and we introduce an $\eps_p \in \{1,2\}$.
Find previous radii $R_a$ and $R_b$ such that $R_a < R_p < R_b$ (one of these may be vacuous).
Now choose any radius $R_{p+1} > 1 + \tau_1$ such that if $\eps_p = 1$ then $R_a < R_{p+1} < R_p$,
while if $\eps_p = 2$ then have $R_p < R_{p+1} < R_b$. This proves the claim. An explicit choice
of such radii is the following:
\begin{equation} \label{radii}
R_1 \; \text{satisfies}\; R_1 \cdot \big (1 -\frac{1}{2} - \cdots - \frac{1}{2^{p-1}} \big) > 1 + \tau_1; \quad
R_k = R_1 \cdot \big (1 + \sum_{j=1}^{k-1} \frac{(-1)^{\eps_j}}{2^j} \big).
\end{equation}

Finally, using the summation identity above to carry out the sum over every $x_k$, and simplifying the
resulting integrand, we get the representation of $L^{\vec{\eps}}(i,j)$ stated in the lemma.
\end{proof}

We conclude the section with a presentation of $L(i,j \vt \theta)$ that will be used to get
a Fredholm determinant form in the next section and also for its asymptotics.
Consider the contour integral form of $L^{\veps}(i,j)$ in Lemma \ref{lem:Leps}.
Deform each contour $\gamma_{R_k}(1)$ to a union of a contour around 0,
say $\gamma_{\rho_k}(0)$, and a contour around 1, say $\gamma_{\rho'_k}(1)$.
The first of these should enclose $\gamma_{\tau_1}$ and $\gamma_{\tau_2}$ and lie
within the circle of radius $1 - \sqrt{q}$. That is,
$$\tau_2 < \tau_1 < \rho_k < 1- \sqrt{q} \;\; \text{for every}\;\; k.$$
The second should enclose non-zero poles in variable $z_k$ and lie outside the circle of radius $1 - \sqrt{q}$. That is,
$$1 - \sqrt{q} < 1-\rho'_k < 1- q \;\;\text{for every}\;\; k.$$
See Figure \ref{fig:contours} for an illustration.
\begin{figure}[htpb!]
\centering
\includegraphics[scale=0.4]{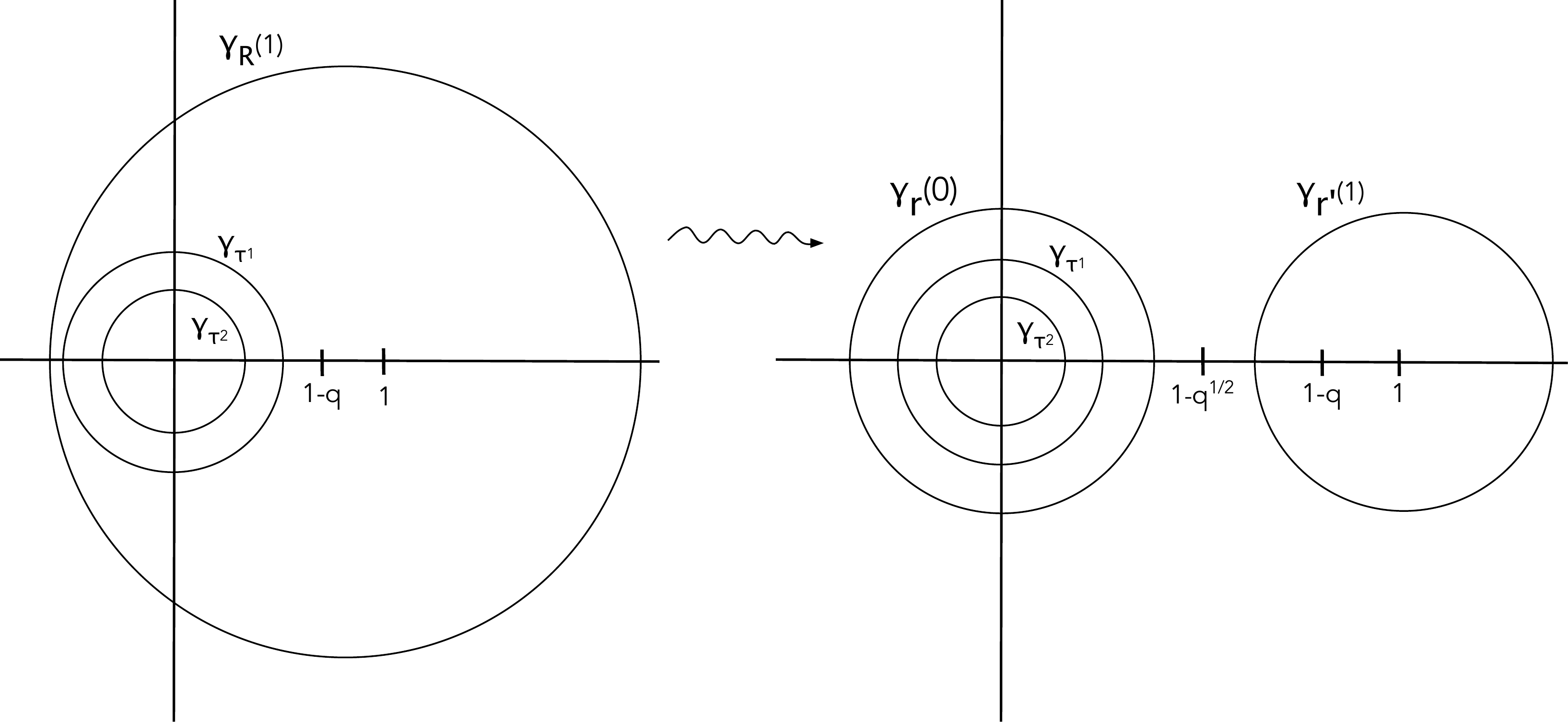}
\caption{The deformation of $\gamma_R(1)$ into two contours $\gamma_{r}(0)$ and $\gamma_{r'}(1)$.}
\label{fig:contours}
\end{figure}

The radii of the contours should be arranged so that the ordering imposed by $\veps$ remains,
that is, if $\eps_k =2$ then $\rho_k < \rho_{k+1}$ and $\rho'_k < \rho'_{k+1}$, etc. In order to simplify
notation, we denote $\gamma_{\rho_k}(0)$ as $\gamma_{R_k}(0)$ and $\gamma_{\rho'_k}(1)$ as
$\gamma_{R_k}(1)$. In this notation we write the contour integral for $L^{\veps}(i,j)$ as a sum of
$2^p$ contour integrals, where for each integral we make a choice of contours $z_1 \in \gamma_{R_1}(\delta_1),
z_2 \in \gamma_{R_2}(\delta_2), \ldots, z_p \in \gamma_{R_p}(\delta_p)$, and
$\vec{\delta} = (\delta_1, \ldots, \delta_p) \in \{0,1\}^p$. Thus,
\begin{equation} \label{Lepsdelta}
L(i,j \vt \mathbold{\theta}) = \sum_{\vdel \in \{0,1\}^p} \sum_{\veps \in \{1,2\}^{p-1}}\,
(-1)^{p-1 + \eps_1 + \cdots + \eps_{p-1}} \, \frac{c(i,j)}{1- \sqrt{q}} \, \theta^{\veps}(i) \, L^{\veps}_{\vdel}(i,j).
\end{equation}
The entry $L^{\veps}_{\vdel}(i,j)$ looks the same as the integral in Lemma \ref{lem:Leps}
except $\gamma_{R_k}(1)$ is replaced by $\gamma_{R_k}(\delta_k)$ in our simplified notation.

\subsection{Fredholm determinant form} \label{sec:fdet}

Looking at \eqref{Lepsdelta}, the identity matrix in the Fredholm determinantal form for $L$
will come from the contribution at $\vdel = \vec{0}$. So we define some matrices by
which the $L^{\veps}_{\vdel}$s will be expressed. Recall notations from $\S$\ref{sec:notation}.

\begin{defn} \label{def:Lk}
Let $L_0 = 0$. For $1 \leq k \leq p$, define a matrix $L_k$ as follows. For $1 \leq i, j \leq N$ (recall $N = n_p$),
	\begin{align*}
		&L_k(i,j) = \frac{1}{1-\sqrt{q}}\, \oint\limits_{\gamma_{\tau_1}} d\zeta_1\,  \oint\limits_{\gamma_{\tau_2}} d \zeta_2\,
		\frac{G \big(\zeta_1 \vt n_k - i, m_k- m(i), a_k-a(i) \big)}{G \big(\zeta_2 \vt n_k - j+1, m_k-m(j), a_k-a(j)\big) \,(\zeta_1 - \zeta_2)}.
	\end{align*}
The radii should satisfy $\tau_2 < \tau_1 < 1-\sqrt{q}$.
\end{defn}

\begin{defn} \label{def:Lcd}
	Suppose $0 \leq k_1 < k_2 \leq p$ and $\veps  \in \{1,2\}^{p-1}$. Let $\tau_2 < \tau_1 < 1-\sqrt{q}$.
	Consider radii $R_{k_1+1}, \ldots, R_{k_2}$ such that $q < R_k < \sqrt{q}$ for every $k$, and they are ordered in the following way:
	$$ R_k < R_{k+1} \;\; \text{if}\;\; \eps_k = 2 \quad \text{while} \quad R_k > R_{k+1} \;\;\text{if}\;\; \eps_k = 1.$$
	Note this depends only on $\eps_{k_1+1}, \ldots, \eps_{k_2-1}$.
	(It is possible to arrange the radii according to $\veps$ as shown in Lemma \ref{lem:Leps}.)
	Set $\vec{\gamma}_{R^{\veps}} = \gamma_{R_{k_1+1}}(1) \times \cdots \times \gamma_{R_{k_2}}(1)$.
	Define a matrix $L^{\veps}_{(k_1,k_2]}$ as follows.
	\begin{align*}
		&L^{\veps}_{(k_1,k_2]}(i,j) = \ind{i > n_{k_1}, \, j \leq n_{k_2}} \, \frac{1}{1-\sqrt{q}}\,
		\oint\limits_{\gamma_{\tau_1}} d \zeta_1 \oint\limits_{\gamma_{\tau_2}} d \zeta_2 \,
		\oint\limits_{\vec{\gamma}_{R^{\veps}}} d z_{k_1+1} \, d z_{k_1+2} \cdots d z_{k_2} \\
		& \frac{\prod_{k_1 < k \leq k_2} G \big ( z_k \vt \D_k (n,m,a) \big) \, \prod_{k_1 < k < k_2} (z_k - z_{k+1})^{-1}
			\left (\frac{1-\zeta_1}{1-z_1}\right)^{\ind{k_1 =0}}(z_{k_1+1}-\zeta_1)^{-1} (z_{k_2}-\zeta_2)^{-1}}
		{G\big(\zeta_1 \vt  i-n_{k_1}, m(i)-m_{k_1}, a(i)-a_{k_1} \big) \, G \big(\zeta_2 \vt n_{k_2}-j+1, m_{k_2}-m(j), a_{k_2}-a(j) \big)}.
	\end{align*}
\end{defn}

\begin{lem} \label{lem:deltazero}
	Suppose $\vdel$ is identically zero.
	Then $L^{\veps}_{\vec{0}} = 0$ unless $\veps = (\overbrace{2, \ldots, 2}^{k-1}, \overbrace{1, \ldots, 1}^{p-k})$ for some $k$.
	In other words, it is the zero matrix unless there is a $k \in [1,p]$ such that the radii of contours $\gamma_{R_1}(0), \ldots, \gamma_{R_p}(0)$
	satisfy $R_1 < R_2 < \cdots < R_k$ and $R_p < \cdots < R_{k+1} < R_k$.
\end{lem}

\begin{proof}
	The contour integral for $L^{\veps}_{\vec{0}}$ has every contour arranged around the origin.
	The poles of the integrand in $z$-variables come from the term $(z_1-\zeta_1) (z_p-\zeta_2) \prod_k (z_k - z_{k+1})$
	in the denominator. Given $\veps$, suppose there is an index $\ell$ with $1 < \ell < p$ such that $R_{\ell} < R_{\ell-1}$
	and $R_{\ell} < R_{\ell +1}$. Then we may contract the $z_{\ell}$-contour without passing
	any poles in that variable. Hence, $L^{\veps}_{\vec{0}}(i,j) = 0$. It follows that $L^{\veps}_{\vec{0}}$ can only
	be non-zero if there is no such $\ell$, which is the condition on $\eps$ in the lemma.
\end{proof}

\begin{lem} \label{lem:deltanonzero}
	Suppose $\vdel$ is not identically zero. Then $L^{\veps}_{\vdel} = 0$ unless $\vdel = (0, \ldots, 0, 1, \dots, 1, 0, \ldots, 0)$,
	i.e., $\vdel$ consists of a run of 0s (possibly empty), followed by a run of 1s (non-empty), and ending with a run of 0s (again, possibly empty).
	Moreover, suppose $\vdel$ equals 1 for indices on the interval $(k_1,k_2]$ with $0 \leq k_1 < k_2 \leq p$. Then for $L^{\veps}_{\vdel}$ to be
	non-zero it must be that $\eps_1 = \cdots = \eps_{k_1-1} = 2$ and $\eps_{k_2+1} = \cdots \eps_{p-1} = 1$, i.e., $R_1 < \cdots < R_{k_1}$
	and $R_{k_2+1} > \cdots > R_{p}$ (some of these may be vacuous).
\end{lem}

\begin{proof}
	Given $\vdel = (\delta_1, \ldots, \delta_p)$ suppose there are indices $k_1 < k_2$ such that $\delta_{k_1} =1$, $\delta_{k_1+1} = 0$
	and $\delta_{k_2} = 1$. Consider the integral of $L^{\eps}_{\vdel}(i,j)$ involving the $z_{k_1+1}$-contour, which is around 0.
	As the $z_{k_1}$-contour is around 1, we may contract the $z_{k_1+1}$-contour to 0 unless the $z_{k_1+2}$-contour lies
	below it (around 0). But then we may contour that one unless the $z_{k_1+3}$-contour lies below it, and so on, until we
	get to the $z_{k_2-1}$-contour. In that case, we can always contact the $z_{k_2-1}$-contour because the $z_{k_2}$-contour is around 1.
	So $L^{\veps}_{\vdel}(i,j) = 0$ for such $\vdel$, which implies the condition on $\vdel$ in the lemma.
	
	Now suppose $\vdel = (0, \ldots, 0, \overbrace{1, \ldots, 1}^{k_2-k_1}, 0, \ldots, 0)$. Consider the contours in the integral
	for $L^{\veps}_{\vdel}(i,j)$ in variables $z_{1}, \ldots, z_{k_1}$. They lie around 0 and we may contract
	the $z_{k_1}$-contour unless the $z_{k_1-1}$-contour lies below it, and so forth, which shows $L^{\veps}_{\vdel}(i,j) = 0$
	unless $R_1 < R_2 < \cdots < R_{k_1}$. Similarly, it will be zero unless $R_p < \cdots < R_{k_2+1}$. This proves
	the condition stipulated on $\eps$.
\end{proof}

\begin{lem} \label{lem:Lk}
	For $1 \leq k \leq p$, set $\eps^k = (\overbrace{2, \ldots, 2}^{k-1}, \overbrace{1, \ldots, 1}^{p-k})$.
	Then, $L^{\eps^k}_{\vec{0}} = (-1)^{k-1} (1-\sqrt{q}) \big( L_k - L_{k-1} \big)$.
\end{lem}

\begin{proof}
	Look at the contour integral presentation of $L^{\veps}_{\vec{0}}(i,j)$ from Lemma \ref{lem:Leps}.
	Since $\vdel = \vec{0}$, all contours are around the origin. We will contract the $z$-contours
	$\gamma_{R_1}, \ldots, \gamma_{R_p}$ in the order specified by $\eps^k$, and use the group property
	of $G$ to simplify the integrand. We have $R_1 < \cdots < R_k$ and $R_p < \cdots < R_{k+1}< R_k$.
	
	First we contract the $z_p$-contour and pick up residue at $z_p =\zeta_2$. This eliminates the variable $z_p$
	from the integral. We continue by contracting the $z_{p-1}$-contour, again with residue at $z_{p-1} = \zeta_2$, and so
	on until variable $z_{k+1}$ is eliminated. Next, we contract the $z_1$-contour and gain a residue at $z_1 = \zeta_1$. We keep
	doing so until we have contracted all contours except for the variables $\zeta_1, \zeta_2$ and $z_k$. We will also obtain a
	factor of $(-1)^{k-2}$ while eliminating variables $z_2, \ldots, z_{k-1}$ due to the factor $(\zeta_1-z_2) \cdots (z_{k-1}-z_k)$
	in the integrand. Factoring out another $-1$ shows that
	\begin{align*}
		L^{\eps^k}_{\vec{0}}(i,j) = & (-1)^{k-1}  \oint\limits_{\gamma_{\tau_1}} d \zeta_1
		\oint\limits_{\gamma_{\tau_2}} d \zeta_2 \oint\limits_{\gamma_{R_k}} dz_k \\
		 & \frac{G \big(z_k \vt \D_k n, \D_k m, \D_k a \big) \, G \big(\zeta_1 \vt n_{k-1}-i, m_{k-1} - m(i), a_{k-1} - a(i) \big)
		 	\big (\frac{1-\zeta_1}{1-z_1}\big)^{\ind{k=1}}}
		{G\big (\zeta_2 \vt n_k - j+1, m_k - m(j), a_k - a(j)\big )\, (z_k-\zeta_1) \, (z_k-\zeta_2)}.
	\end{align*}
	Finally, we eliminate the $z_k$-contour and gain a residue at $z_k = \zeta_1$ followed by one at $z_k = \zeta_2$
	(recall $\tau_1 > \tau_2$). This gives the difference $(1-\sqrt{q}) \big( L_k(i,j) - L_{k-1}(i,j)\big)$.
\end{proof}
We remark that the identity matrix in the Fredholm determinantal representation for $L(i,j \vt \mathbold{\theta})$
will appear from the sum $\sum_k \theta^{\eps^k}(i) L^{\eps^k}_{\vec{0}}(i,j)$ by way of Lemma \ref{lem:Lk}.

\begin{lem} \label{lem:Lcd}
	Consider $0 \leq k_1 < k_2 \leq p$ and
	$$\veps = (2, \ldots, 2, \eps_{\max \{k_1,1\}}, \ldots, \eps_{\max \{k_2,p-1\}}, 1, \ldots, 1) \in \{1,2\}^{p-1}.$$
	Suppose $\vdel$ equals 1 on indices over the interval $(k_1,k_2]$ and 0 elsewhere. Then,
	$ L^{\veps}_{\vdel} = (-1)^{k_1} \, (1-\sqrt{q}) L^{\veps}_{(k_1,k_2]}$.
	Furthermore, $L^{\veps}_{(p-1,p]} $ equals $L_p$ where
	\begin{equation*}
	L_p(i,j)  = \frac{\ind{i > n_{p-1}}}{1-\sqrt{q}} \intz{2} \oint_{\gamma_{R_p}(1)} dz_p\,
	 \frac{G \big(z_p \vt n_p-i, \D_p m, \D_p a\big)}{G\big (\zeta_2 \vt n_p-j+1, m_p-m(j), a_p-a(j)\big)\, (z_p-\zeta_2)}.
	\end{equation*}
\end{lem}

\begin{proof}
	By Lemma \ref{lem:deltanonzero}, $L^{\veps}_{\vdel} = 0$ unless $\veps$ is as given in the
	statement of this lemma. Consider again the contour integral presentation of $L^{\veps}_{\vdel}(i,j)$
	from Lemma \ref{lem:Leps}. The contours around 0 are those in variables $z_1, \ldots, z_{k_1}$ and $z_{k_2 +1}, \ldots, z_p$.
	We also have $R_1 < \cdots < R_{k_1}$ and $R_p < \cdots < R_{k_2+1}$.
	
	As in the proof of the previous lemma, we contract the contours around 0, gaining residues, and present
	$L^{\veps}_{\vdel}(i,j)$ as an integral involving variables $\zeta_1, \zeta_2, z_{k_1 +1}, \ldots, z_{k_2}$.
	The calculation of this is straightforward and we omit the details. The reason a factor
	$\big ( \frac{1-\zeta_1}{1-z_1} \big )^{\ind{k_1=0}}$ appears is that when $k_1=0$ the $z_1$-contour
	is not contracted, so no residue is obtained at $z_1 = \zeta_1$.
	
	The final result is a presentation of $L^{\veps}_{\vdel}(i,j)$ that appears like $(1-\sqrt{q}) L^{\veps}_{(k_1,k_2]}(i,j)$
	from Definition \ref{def:Lcd} except the indicator $\ind{i > n_{k_1}, \, j \leq n_{k_2}}$ is absent. To see why we
	may assume $i > n_{k_1}$, observe the variable $\zeta_1$ appears in the integrand of $L^{\veps}_{(k_1,k_2]}(i,j)$ as
	$$\frac{G(\zeta_1 \vt n_{k_1}-i, m_{k_1}-m(i), a_{k_1}-a(i))}{z_{k_1 +1}-\zeta_1}.$$
	When $n_{k_1} - i \geq 0$, there is no pole in the $\zeta_1$ variable inside $\gamma_{\tau_1}$ and the contour
	may be contracted to 0. Similarly, if $j > n_{k_2}$, there is no pole in $\zeta_2$ inside $\gamma_{\tau_2}$.
	
	To simplify $L^{\veps}_{(p-1,p]} $ note that it does not depend on $\veps$ as there is a single contour around 1 (the $z_p$-contour).
	Since $i > n_{p-1}$, its integrand decays at least to the order $\zeta_1^{-2}$ in the $\zeta_1$ variable
	(the dependence is displayed above). Further, $m(i) = m_{p-1}$ and $a(i) = a_{p-1}$. So there are no poles at $\zeta_1 = 1-q$ and $1$,
	and the $\zeta_1$-contour can be contracted to $\infty$. In doing so we gain a residue at $z_1 = z_p$ whose value is
	$G(z_p \vt n_{p-1}-i,0,0)$. Then simplifying the integrand using the group property gives the desired
	expression for $L^{\veps}_{(p-1,p]}$.
\end{proof}

The following simplifies $L^{\veps}_{(k_1,k_2]}$ when $k_2 < p$.
\begin{lem} \label{lem:Lcdhelper}
If $0 \leq k_1 < k_2 < p$ and $\veps \in \{1,2\}^{p-1}$ then
$$L^{\veps}_{(k_1,k_2]}(i,j) = \ind{j\leq n_{k_2-1}} \, L^{\veps}_{(k_1,k_2]}(i,j) + \ind{i > n_{k_1}, \, j \in (n_{k_2-1}, n_{k_2}]} \, J^{\veps}_{(k_1,k_2]}(i,j),$$
where
\begin{align*}
& J^{\veps}_{(k_1,k_2]}(i,j) = \frac{1}{1-\sqrt{q}}\, \oint\limits_{\gamma_{\tau_1}} d\zeta_1 \oint\limits_{\vec{\gamma}_{R^{\veps}}}\, dz_{k_1+1} \cdots d z_{k_2} \\
& \frac{\prod_{k_1 < k < k_2} G \big(z_k \vt \D_k (n,m,a)\big) G \big( z_{k_2} \vt j-1-n_{k_2-1}, \, \D_{k_2}(m, a)\big) \,
\left (\frac{1-\zeta_1}{1-z_1}\right)^{\ind{k_1=0}}}
{G\big(\zeta_1 \vt  i-n_{k_1}, m(i)-m_{k_1}, a(i)-a_{k_1} \big) \prod_{k_1 < k < k_2}(z_k - z_{k+1}) \, (z_{k_1 +1} - \zeta_1)}.
\end{align*}
The contours in $J^{\veps}_{(k_1,k_2]}$ are arranged like those in $L^{\veps}_{(k_1,k_2]}$.
\end{lem}

\begin{proof}
Consider $L^{\veps}_{(k_1,k_2]}(i,j)$ when $j \in (n_{k_2-1},n_{k_2}]$. Since $k_2 < p$, we have $m(j) = m_{k_2}$ and $a(j) = a_{k_2}$.
Therefore, the integrand depends on $\zeta_2$ according to the term $G\big(\zeta_2 \vt n_{k_2} -j+1,0,0\big)\, (z_{k_2}-\zeta_2)$
in the denominator. Since $n_{k_2} -j \geq 0$, we may contract the $\zeta_2$-contour to infinity with residue at
$\zeta_2 = z_{k_2}$ to find that
$$\oint\limits_{\gamma_{\tau_2}} d \zeta_2\, \frac{1}{G(\zeta_2 \vt n_{k_2}-j+1,0,0) (z_{k_2}-\zeta_2)} = \frac{1}{G(z_{k_2} \vt n_{k_2}-j+1,0,0)}.$$
So we evaluate the integral in $\zeta_2$ and simplify the integrand using the group property of $G$, which results in $J^{\veps}_{(k_1,k_2]}$.
\end{proof}

We may now write $L(i,j \vt \theta)$ from \eqref{Lepsdelta} in the following way by using Definitions \ref{def:Lk} and \ref{def:Lcd},
as well as Lemmas \ref{lem:deltazero}, \ref{lem:deltanonzero}, \ref{lem:Lk} and \ref{lem:Lcd}.
Observe that for $\veps = \eps^k$ as in Lemma \ref{lem:Lk}, $(-1)^{p-1 + \sum_i \eps^k_i + k-1} = 1$.
Also, for $0 \leq k_1 < k_2 \leq p$ and $\veps$ as in Lemma \ref{lem:Lcd},
$$ (-1)^{p-1 + \sum_i \eps_i + k_1} = (-1)^{k_1 + \min \{k_2,p-1\}} \cdot (-1)^{\eps_{[k_1,k_2]}}, \quad \text{where}\;\;
(-1)^{\eps_{[k_1,k_2]}} \;\;\text{is around}\;\; \eqref{theta}.$$
Putting all this together with \eqref{Lepsdelta} we find that
\begin{align} \label{Lkcd}
L(i,j \vt \mathbold{\theta}) & = \sum_{k=1}^p c(i,j)\, \theta^{\eps^k}(i) \,
\big(L_k - L_{k-1} \big)(i,j) \;\; + \\ \nonumber
& \sum_{0 \leq k_1 < k_2 \leq p}
\sum_{\substack{\veps \in \{1,2\}^{p-1} \\ \eps_i = 2 \;\text{if}\; i < \max \, \{k_1,1\} \\ \eps_i = 1 \;\text{if}\; i > \min \,\{k_2,p-1\}}}
(-1)^{\eps_{[k_1,k_2]}+k_1 + \min \{k_2,p-1\}} \, c(i,j) \, \theta^{\veps}(i) \, L^{\veps}_{(k_1,k_2]} (i,j)\,.
\end{align}

It will be convenient to write the matrices associated to $L(i,j \vt \theta)$ from \eqref{Lkcd}
in the $p \times p$ block form, which motivates the following definition.
\begin{defn} \label{def:blockform}
$A(\theta)$ and $B(\theta)$ are $N \times N$ matrices with a $p \times p$ block form as follows.
Recall Definitions \ref{def:Lk} and \ref{def:Lcd}, and notation introduced in $\S$\ref{sec:notation}.
In particular, from \eqref{blocknot}, that the $(r,s)$-block of a matrix $M$ is denoted $M(r,i;s,j)$,
and that $r^{*} = \min \, \{r,p-1\}$.

(1) Define matrix $B(\mathbold{\theta})$, $\mathbold{\theta} = (\theta_1, \ldots, \theta_{p-1})$, by
\begin{align*}
&B(r,i; s,j \vt \mathbold{\theta}) = (1 + \Theta(r \vt s)) \cdot c(r,i;s,j) \cdot \ind{ s < r^{*}}\,
\frac{1}{1-\sqrt{q}} \oint\limits_{\gamma_{\tau}} dw\, \frac{1}{G\big(w \vt i-j+1, \, \D_{s, r^{*}}(m,a) \big)},
\end{align*}
where the circular contour $\gamma_{\tau}$ around 0 had radius $\tau < 1-\sqrt{q}$ and $\Theta(r \vt s)$ is given by \eqref{theta}.

(2) Define matrix $A(\mathbold{\theta}) = A_1(\mathbold{\theta}) + A_2(\mathbold{\theta})$ as follows.
\begin{align*}
 A_1(r,i; s,j \vt \mathbold{\theta}) &=  \sum_{k=0}^p \Theta(r \vt k) \cdot L [ {\scriptstyle k,k \vt \emptyset } ](r,i;s,j), \; \text{where}\\
 L [{\scriptstyle k,k \vt \emptyset}]\, (r,i;s,j) &= c(r,i;s,j) \, \ind{s < k < r^{*}} \cdot L_k(r,i;s,j).
\end{align*}
Let $0 \leq k_1, k_2 \leq p$ and $\veps \in \{1,2\}^{p-1}$. Set
\begin{align*}
 A_2(r,i;s,j \vt \mathbold{\theta})  = & 
 \sum_{\substack{k_1 < k_2, \, \veps \\ \eps_k = 2 \;\text{if}\; k < \max \, \{k_1,1\} \\ \eps_k = 1 \;\text{if}\; k > \min \,\{k_2,p-1\}}}
 (-1)^{\eps_{[k_1,k_2]} + k_1 + k_2^{*}} \cdot \theta(r \vt \veps) \, \times \\
&  \Big [ L^{\veps} [{\scriptstyle k_1, k_2 \vt (k_1, k_2]}]  +
L^{\veps} [{ \scriptstyle k_1 \vt (k_1, k_2]}] +
\ind{k_1=p-1,k_2=p} L[{\scriptstyle p \vt p}] \Big ] (r,i;s,j),
\end{align*}
where recalling $L_p$ and $J^{\veps}_{(k_1,k_2]}$ from Lemmas \ref{lem:Lcd} and \ref{lem:Lcdhelper}, respectively, we define
\begin{align*}
L^{\veps} [{\scriptstyle k_1, k_2 \vt (k_1, k_2]}] (r,i;s,j)  &= c(r,i;s,j)\, \ind{k_1 < r^{*}, \; s^{*} < k_2,\, k_1<k_2} \cdot L^{\veps}_{(k_1,k_2]}(r,i;s,j). \\
L^{\veps} [{ \scriptstyle k_1 \vt (k_1, k_2]}] (r,i;s,j) &= c(r,i;s,j) \, \ind{k_1 < r^{*}, \; s = k_2 < p,\, k_1<k_2} \cdot J^{\veps}_{(k_1,k_2]}(r,i;s,j).\\
L[{\scriptstyle p \vt p}](r,i;s,j) & = c(r,i;s,j)\, \ind{r=p} \cdot L_p(r,i;s,j).
\end{align*}
\end{defn}
Some comments on these matrices. In terms of the $p \times p$ block structure, $B(\theta)$ is lower
triangular with zeroes on the diagonal blocks. Its last two column blocks are zero as well. The matrix $A_1(\theta)$
is also strictly block-lower-triangular with the last three column blocks being zero. The matrix $L^{\veps}[{\scriptstyle k_1, k_2\vt (k_1, k_2]}]$
has non-zero blocks strictly above row $k_1$ ($r > k_1$) and at or below column $k_2$. The matrix $L^{\veps}[{\scriptstyle k_1 \vt (k_1,k_2]}]$
has non-zero blocks only on column $k_2 < p$ and above row $k_1$. The matrix $L[{\scriptstyle p \vt p}]$ has non-zero
block only on row $p$.

\begin{thm} \label{thm:2}
Let $\mathbold{G}$ be the growth function defined by \eqref{gmn}. Let $A(\mathbold{\theta})$ and $B(\mathbold{\theta})$ be
from Definition \ref{def:blockform}, and suppose $p \geq 2$. For $m_1 < m_2 < \cdots < m_p$ and $n_1 < n_2 < \cdots < n_p$,
we have
\begin{align*}
&\pr{\mathbold{G}(m_1,n_1) < a_1, \mathbold{G}(m_2,n_2) < a_2, \ldots, \mathbold{G}(m_p,n_p) < a_p} = \\
& \quad \oint\limits_{\gamma_r^{p-1}} d\theta_1 \cdots d\theta_{p-1}\,
\frac{1}{\prod_{k=1}^{p-1} (\theta_k -1)}\, \dt{I + A(\mathbold{\theta}) + B(\mathbold{\theta})}.
\end{align*}
Here, $\gamma_r^{p-1} = \gamma_r \times \cdots \times \gamma_r$ ($p-1$ times) and $\gamma_r$
is a counter-clockwise, circular contour around the origin of radius $r > 1$.
\end{thm}

In order to prove the theorem we need the following.
\begin{lem} \label{lem:Lkhelper}
Set, for $0< \tau < 1- \sqrt{q}$,
$$B(i,j) = \frac{1}{1-\sqrt{q}} \, \oint\limits_{\gamma_{\tau}}dw \, \frac{1}{G \big(w \vt i-j+1,\, m(i)-m(j), \, a(i) - a(j)\big)}.$$
Then,
\begin{align*}
& (L_k-L_{k-1})(i,j) =  \mathbf{1} \{ \, i,j \in (n_{k-1}, n_k] \, \} \cdot \mathbf{1} \{i=j\} \;\; + \\
& \qquad  \mathbf{1} \{ \, i \in (n_{k-1}, n_k], \; j \leq n_{\min \{k-1,\, p-2\}} \, \} \cdot B(i,j) \;\; + \\
& \qquad  \mathbf{1} \{ \, i > n_k, \; j \leq n_k, \; k \leq p-2 \,\} \cdot L_k(i,j) \,-\,
 \mathbf{1} \{\, i > n_{k-1}, \; j \leq n_{k-1}, \; k \leq p-1 \,\} \cdot L_{k-1}(i,j).
\end{align*}
\end{lem}

\begin{proof}
Recall from Definition \ref{def:Lk}:
\begin{align*}
		&L_k(i,j) = \frac{1}{1-\sqrt{q}}\, \oint\limits_{\gamma_{\tau_1}} d\zeta_1\,  \oint\limits_{\gamma_{\tau_2}} d \zeta_2\,
		\frac{G \big(\zeta_1 \vt n_k - i, m_k- m(i), a_k-a(i) \big)}{G \big(\zeta_2 \vt n_k - j+1, m_k-m(j), a_k-a(j)\big) \,(\zeta_1 - \zeta_2)}.
	\end{align*}
\begin{itemize}
\item If $j > n_k$ then there is no pole at $\zeta_2 = 0$ in the above and we can contract the $\zeta_2$-contour to 0.
So $L_k(i,j) = 0$, which means $L_k(i,j) = \ind{j \leq n_k} L_k(i,j)$.
\item If $i > n_k$ and $m(i) = m_k$ (so $a(i) = a_k$ as well), then $L_k(i,j) = 0$ because the $\zeta_1$-contour may be contracted to $\infty$.
The condition $i > n_k$ and $m(i) = m_k$ is the same as $i > n_k$ and $k \geq p-1$. Indeed, if $i > n_k$ and $k \geq p-1$ then
$m(i) = m_k = m_{p-1}$ ($i > n_p$ is vacuous). Therefore,
$L_k(i,j) = \ind{i \leq n_k, \, j\leq n_k} L_k(i,j) + \ind{i > n_k, \, j \leq n_k, \, k \leq p-2} L_k(i,j)$.
\item When $i \leq n_k$ we can contract the $\zeta_1$-contour to 0, picking up a residue at $\zeta_1 = \zeta_2$, which equals $B(i,j)$.
Also, $B(i,j) = 0$ if $j > i$ because there is no pole at $w =0$ in that case. Consequently,
$$L_k(i,j) = \ind{i \leq n_k, \, j \leq n_k, \, j \leq i} B(i,j) + \ind{i > n_k, \, j \leq n_k, \, k \leq p-2} L_k(i,j).$$
\item If $m(i) = m(j)$ then
$$B(i,j) = (1-\sqrt{q})^{i-j} \oint\limits_{\gamma_{\tau}} d\zeta\, \zeta^{j-i-1} = \ind{i=j}.$$
\end{itemize}
Putting all this together we infer that
\begin{align*}
L_k(i,j)  =& \;  \mathbf{1} \{\, i \leq n_k, \; j \leq n_k, \; i = j \,\} \;\; + \\
& \;  \mathbf{1} \{ \, i \leq n_k, \; j \leq n_k, \; j \leq i, \; m(i) \neq m(j)\,\} \cdot B(i,j) \;\; + \\
& \; \mathbf{1} \{ \, i > n_k, \; j \leq n_k, \; k \leq p-2 \, \} \cdot L_k(i,j)
\end{align*}
Taking the difference of $L_k(i,j)$ from $L_{k-1}(i,j)$ by using the expression above gives the expression in
the lemma, except that the indicator in front of $B(i,j)$ reads $i \in (n_{k-1}, n_k]$, $j \leq n_{k-1}$ and
$m(i) \neq m(j)$. However, when $j \leq n_{k-1}$, the condition $m(i) \neq m(j)$ is precisely $j \leq n_{\min \{k-1,p-2\}}$.
\end{proof}

\paragraph{\textbf{Proof of Theorem \ref{thm:2}}}
We have the basic integral expression for the multi-point probability from Lemma \ref{lem:CB}.
The matrix $L(i,j \vt \mathbold{\theta})$ is given by \eqref{Lkcd}, which we will prove to equal
$I + A(\mathbold{\theta}) + B(\mathbold{\theta})$.

The matrix $A_2(\mathbold{\theta})$ is the one written in the second line of equation \eqref{Lkcd}.
We should explain the conditions $k_1 < \min \{r,p-1\}$ and $\min \{s,p-1\} < k_2$ in $L^{\veps} [{\scriptstyle k_1,k_2 \vt (k_1, k_2]} ]$.
Also, why is it that $k_1 < \min \{r,p-1\}$ and $s = k_2 < p$ in $L^{\veps} [{\scriptstyle k_1 \vt (k_1,k_2]}]$.

The condition $k_1 < r$ appears because in the definition of $L^{\veps}_{(k_1,k_2]}(i,j)$ we have
$i < n_{k_1}$, while we know $i \in (n_{r-1}, n_r]$. The condition $k_1 < p-1$ appears because
$L^{\veps}_{(k_1,k_2]}$ is zero if $k_1 \geq p-1$ by Lemma \ref{lem:Lcd}.
The condition on $s$ arises from the decomposition of $L^{\veps}_{(k_1,k_2]}$ in Lemma \ref{lem:Lcdhelper}.
Since $j \in (n_{s-1}, n_s]$, we have $s \leq k_2$, which we decompose
into two conditions: (a) $\ind{s \leq k_2, k_2 = p} + \ind{s < k_2, k_2 < p} = \ind{\min \{s,p-1\} < k_2}$ and (b) $\ind{s = k_2 < p}$.
In case (b) the matrix $L^{\veps}_{(k_1,k_2]}$ becomes $J^{\veps}_{(k_1,k_2]}$ by Lemma \ref{lem:Lcdhelper}, and this
results in the matrix $L^{\veps} [{\scriptstyle k_1 \vt (k_1,k_2]} ]$.

We have to show that the matrix associated to the first line in \eqref{Lkcd} equals
$I + A_1(\mathbold{\theta}) + B(\mathbold{\theta})$. If we write the statement of Lemma \ref{lem:Lkhelper}
in block notation, it reads
\begin{align} \label{eqn:Lkdiff}
& (L_k-L_{k-1})(r,i; s,j) =  \ind{r=k=s} \cdot \ind{i=j} \; + \;
\ind{r = k, \, s+1 \, \leq\,  \min \{r,p-1\}} \cdot B(r,i; s,j) \;\; + \\ \nonumber
& \qquad \ind{r > k, \, s \leq k, \, k \leq p-2} \cdot L_k(r,i;s,j) \,-\,
\ind{r > k-1, \, s \leq k-1, \, k \leq p-1} \cdot L_{k-1}(r,i; s,j).
\end{align}
We need to consider the weighted sum $\sum_k \theta^{\eps^k}(i) \cdot c(r,i;s,j) \times \eqref{eqn:Lkdiff}$.

Observe that if $i \in (n_{k-1}, n_k]$ then
$$\theta^{\eps^k}(i) = \theta(k \vt \eps^k) = \theta_1^{- \ind{i \leq n_1}} \cdots \theta_{k-1}^{-\ind{i \leq n_{k-1}}}\,
\theta_k^{\ind{i > n_k}} \cdots \theta_{p-1}^{\ind{i > n_{p-1}}} = 1.$$
Therefore, summing $\theta^{\eps^k}(i) \ind{r=k=s} \ind{i=j}$ over $k$ and
multiplying by $c(r,i;s,j)$ gives the matrix $\ind{i=j} c(r,i;s,j)$, which is the identity since $c(r,i;s,j)$ is a conjugation factor.

Consider the third term on the r.h.s.~of \eqref{eqn:Lkdiff} containing the difference between $L_k$ and $L_{k-1}$.
This term is zero unless $s < r$, and $k$ satisfies $s \leq k \leq r$. When $s < k < r$, it equals
$\ind{k < p-1} (L_k - L_{k-1})(r,i;s,j)$. Also, the condition $s < k < r$ is vacuous unless $s < r-1$.
When $k=s$, the term becomes $\ind{s < p-1} L_s(r,i;s,j)$. When $k=r$, it equals $\ind{r < p} L_{r-1}(r,i;s,j)$.
We will see in the following paragraph that $L_s(r,i;s,j) = B(r,i;s,j)$. Thus, we find appearances of $B(r,i;s,j)$ in
the third term from $L_k$ when $k=s$, and from $L_{k-1}$ when $k = s+1$. Accounting for these $B(r,i;s,j)$,
we find the weighted sum
\begin{align*}
& \sum_k \theta^{\eps^k}(i) \big ( \text{third term of}\;\; \eqref{eqn:Lkdiff} \big)  = (I) + (II), \;\;\text{where} \\
(I) & = \ind{s < r, \, s < p-1} \Big ( \theta(r \vt \eps^s) - \big (\ind{s+1 <r, \, s+1 <p-1} + \ind{s+1=r, \, r <p}\big ) \theta(r \vt \eps^{s+1}) \Big ) B(r,i;s,j), \\
(II) & = \ind{s+1 \,< r}  \Big ( \sum_{k: \, s+1 < k < r, \, k <p-1} \theta(r \vt \eps^k) \big (L_k-L_{k-1} \big)(r,i;s,j)\; + \\
 & \qquad \ind{s < p-2} \, \theta(r \vt \eps^{s+1}) L_{s+1}(r,i;s,j) - \ind{r < p} \, L_{r-1}(r,i;s,j) \Big ).
\end{align*}
We have used that $\theta^{\eps^k}(i) = \theta(r \vt \eps^k)$.

Consider term (I). If $s < r$ and $s < p-1$ then
$$\ind{s+1 < r, \, s+1<p-1} + \ind{s+1=r, \, r< p} = 1 - \ind{r=p,\, s=p-2},$$
which gives the coefficient $\Theta(r \vt s)$ in term (I) if we recall its definition from \eqref{theta}.
If we take this contribution of $\ind{s < r, \, s < p-1} \Theta(r \vt s) B(r,i;s,j)$, and combine it with
$$\sum_k \theta^{\eps^k}(i) \, \ind{r=k,\, s < r, \, s<p-1} B(r,i;s,j) = \ind{s < \min\{r,p-1\}}\, B(r,i;s,j)$$
coming from the $k$-summation of the second term of \eqref{eqn:Lkdiff}, then, after conjugation
by $c(r,i;s,j)$, we get the matrix $B(\mathbold{\theta})$ from Definition \ref{def:blockform}.

Now consider term (II). If we express it as a sum involving the $L_k(r,i;s,j)$ then the coefficient of
$L_k(r,i;s,j)$ is $\ind{s < k < \min \{r,p-1\}} \cdot (\theta(r \vt \eps^k) - \theta(r \vt \eps^{k+1}))$. Recalling
$\Theta(r \vt k)$, we see that $\theta(r \vt \eps^k) - \theta(r \vt \eps^{k+1})
= \Theta(r \vt k)$ because $s < p-2$, due to $s < k < \min \{r,p-1\}$. Hence, the contribution of $L_k$ appears
as $\Theta(r \vt k) \, L_k(r,i;s,j)$. The sum over $k$ followed by multiplication by $c(r,i;s,j)$ equals
the matrix $A_1(\mathbold{\theta})$.

Finally, we show that $L_s(i,j) = B(i,j)$ for $j \in (n_{s-1},n_s]$ and $s \leq p-2$ as is the case above.
Indeed, we have $m(j) = m_s$ and $a(j) = a_s$, which means that
\begin{align*}
		&L_s(i,j) = \oint\limits_{\gamma_{\tau_1}} d\zeta_1\,  \oint\limits_{\gamma_{\tau_2}} d \zeta_2\,
		\frac{G \big(\zeta_1 \vt n_s - i, m_s- m(i), a_s-a(i) \big)}{G \big(\zeta_2 \vt n_s - j+1, 0, 0 \big) \,(\zeta_1 - \zeta_2)}.
\end{align*}
We can contract the $\zeta_2$-contour to $\infty$, since $j \leq n_s$, but doing so leaves a residue at $\zeta_2 = \zeta_1$.
Its value is $B(i,j)$.
\qed

\subsection{Distribution function of the single point law} \label{sec:singletime}
When $p=1$ one can write a Fredholm determinantal expression for $\pr{\mathbold{G}(m,n) < a}$
where the matrix is in terms of a double contour integral. Such formulas are nowadays frequent
as discrete approximations to Tracy-Widom laws, so this section is meant to provide some intuition
for our orthogonalization procedure.

From Lemma \ref{lem:Prsimple} we see that $\pr{\mathbold{G}(m,n) < a} = \dt{\dr^{j-i-1} w_m(a)}_{n \times n}$.
Consider the following matrix $B = [b_{kj}]$, which is a slight variant of $B$ from \eqref{ABdef}.
$$ b_{kj}  = \oint\limits_{\gamma_{\tau}} d\zeta \, \frac{1}{G^{*}\big (\zeta \vt k-j+1, m, a -1\big)}\, .$$
The radius $\tau < 1-q$. The matrix is lower triangular with 1s on the diagonal, so $\dt{B} = 1$. We have
$$ \pr{\mathbold{G}(m,n) < a} = \dt{\ell_{ij}}, \;\; \ell_{ij} = \sum_{k=1}^N (-1)^{k+i} \dr^{k-i-1} w_m(a) b_{kj}\,.$$
Using \eqref{Deltakwm} and Lemma \ref{lem:Gintegral} we find that
$$\ell_{ij} = \oint\limits_{\gamma_{\tau}} d\zeta \oint\limits_{\gamma_R} dz\, \frac{G^{*}(z \vt n-i,m,a-1)}{G^{*}(\zeta \vt n-j+1,m,a-1) (z-\zeta)}\,.$$
The radii $\tau < 1-q$ and $R > 1$. By collecting residue of the $z$-integral at $z = \zeta$, we infer that
\begin{align*}
\ell_{ij} &= \oint\limits_{\gamma_{\tau}} d\zeta\, \zeta^{j-i-1} \;+\; 
\oint\limits_{\gamma_{\tau}} d\zeta \oint\limits_{\gamma_r(1)} dz\, \frac{G^{*}(z \vt n-i,m,a-1)}{G^{*}(\zeta \vt n-j+1,m,a-1) (z-\zeta)} \\
&= \ind{i=j} + M(i,j).
\end{align*}
Now we arrange the radii to have $\tau < 1 - \sqrt{q} < 1 - r < 1 - q$.

If we write $i = \lfloor c_0 n^{1/3} u \rfloor$ and $j = \lceil c_0 n^{1/3} v \rceil$, then a direct asymptotical analysis of
$M(i,j)$ leads to the Airy kernel \eqref{singletime} under KPZ scaling.

\section{Asymptotics: formulation in the KPZ-scaling limit} \label{sec:asymptotics}
In order to prove Theorem \ref{thm:1} we will consider the limit of the determinantal expression
from Theorem \ref{thm:2} under KPZ scaling. We will do so in several steps. In $\S$\ref{sec:asy1} we define
the Hilbert space where all matrices are embedded in the pre and post limit. The proof of convergence of
the determinant will be based on a steepest descent analysis of the matrix entries. In $\S$\ref{sec:asy2}
we provide contours of descent and behaviour of the entries around critical points. The proof of
convergence is in $\S$\ref{sec:asy3}. There is a technical addendum in $\S$\ref{sec:asy4}, where it is
also proved that the limit from Theorem \ref{thm:1} is a probability distribution.

\subsection{Setting for asymptotics} \label{sec:asy1}

Consider the space $X = \overbrace{\R_{< 0} \oplus \cdots \oplus \R_{<0}}^{p-1} \, \oplus \, \R_{>0}$ and a measure
$\lambda$ on it defined by
$ \int\limits_{X} d\lambda \, f = \sum_{k=1}^{p-1} \int_{-\infty}^{0} dx\, f(k,x) \; + \; \int_{0}^{\infty} dx\, f(p,x)$.
Define the Hilbert space
\begin{equation} \label{Hspace}
\hb = L^2(X,\lambda) \cong \underbrace{L^2(\R_{<0}, dx) \oplus \cdots \oplus L^2(\R_{<0},dx)}_{p-1} \, \oplus \,L^2(\R_{>0},dx)\,.
\end{equation}
Recall the partition $\{1\, \ldots, N\} = (0,n_1] \cup \cdots (n_{p-1},n_p]$.
Embed indices from $\{1, \ldots, N\}$ into $X$ by mapping each index $i$ into a unit length interval in the following manner.
\begin{equation} \label{indexembed}
 i \mapsto \begin{cases}
        \text{points}\; (k, u ) & \text{for}\;\; i-1 < n_k +u\leq i \;\; \text{if}\; i \in (n_{k-1},n_k] \;\text{and}\; k < p, \\
        \text{points}\; (p, u) & \text{for} \;\;  i-1 < n_{p-1}+u \leq i \;\; \text{if}\; i \in (n_{p-1},n_p].
       \end{cases}
\end{equation}
Observe that for $k <p$ the block $(n_{k-1},n_k]$ is mapped to the interval $(-\D_k n, 0]$ and for $k=p$ it is mapped to $(0, \D_p n]$.

An $N \times N$ matrix $M$ embeds as a kernel $\widetilde{M}$ on $\hb$ by
\begin{equation} \label{matrixembed}
\widetilde{M}(r,u ; \, s,v) = M \big(r, \, n_{\min \{r,p-1\}} + \lceil u \rceil; \; s, \, n_{\min \{s,p-1\}} + \lceil v \rceil \big).
\end{equation}
Here we have used the block notation \eqref{blocknot} and $\lceil u \rceil$ is the integer part of $u$ after rounding up.
The range of $u$ and $v$ lie in the aforementioned intervals determined by each block, but we may extend it to all of
$\R_{<0}$ (and to $\R_{>0}$ for the final blocks) by making $\widetilde{M}$ zero.
Then, by design,
$$ \dt{I + \widetilde{M}}_{\hb} = \dt{I+M}_{N \times N}$$
where
$$ \dt{I + \widetilde{M}}_{\hb} = 1 + \sum_{k\geq 1} \frac{1}{k!} \int\limits_{X^k} d\lambda(r_1, u_1) \cdots d \lambda(r_k, u_k)\,
\dt{\widetilde{M}(r_i,u_i;\, r_j,u_j)}_{k \times k}.$$
This is because $\widetilde{M}$ is constant to $M(i,j)$ on a square of the form
$[\widetilde{i}-1,\widetilde{i}) \times [\widetilde{j}-1,\widetilde{j})$ determined according to the
correspondence \eqref{indexembed}, and zero elsewhere.

In order to perform asymptotics we should rescale variables of $\widetilde{M}$ according to KPZ scaling \eqref{KPZscaling}.
In this regard, recalling $\nu_T = c_0 T^{1/3}$, we change variables $(r,u) \mapsto (r, \nu_T \cdot u)$
in the Fredholm determinant of $\widetilde{M}$ above. So if we define a new matrix kernel
\begin{equation} \label{Fredholmembed}
F(r,u ;\, s,v) = \nu_T \, \widetilde{M} \big ( r,\, \nu_T \cdot u;\; s,\, \nu_T \cdot v\big ),
\end{equation}
then
$$ \dt{I+ F}_{\hb} = \dt{I + M}_{N \times N}.$$

We will use the following estimate about Fredholm determinants.
\begin{lem} \label{lem:detest}
Let $A$ and $E$ be matrix kernels over a space $L^2(X, \mu)$, which satisfy the following for some
positive constants $C_1$, $C_2$ and $\eta \leq 1$. There are non-negative functions $a_1(x), a_2(x), e_1(x), e_2(x)$ on $X$
such that
$$|A(x,y)| \leq a_1(x) a_2(y) \quad \text{and} \quad |E(x,y)| \leq \eta \, e_1(x) e_2(y).$$
Moreover, both $a_1(x), e_1(x) \leq C_1$ and both $\int_{X} d \mu(x) \, a_2(x)$, $\int_{X} d \mu(x)\, e_2(x) \leq C_2$.
Then there is a constant $C_3 = C_3(C_1,C_2)$ such that
$$ \left | \dt{I + A + E}_{L^2(X,\mu)} - \dt{I+A}_{L^2(X,\mu)} \right | \leq \eta \, C_3.$$
\end{lem}

\begin{proof}
For $x_1, \ldots, x_k \in X$, consider the determinant of $[A(x_i,x_j) + E(x_i,x_j)]$.
Using multi-linearity, Hadamard's inequality, and the bounds on $a_1(x)$ and $e_1(x)$, we find that
$$ \left | \dt{A(x_i,x_j) + E(x_i,x_j)} - \dt{A(x_i,x_j)} \right | \leq
\sum_{S \subset [k], \, S \neq \emptyset} \, \eta^{|S|} k^{k/2} C_1^k \prod_{j\in S} e_2(x_j) \prod_{j \notin S} a_2(x_j).$$
If we integrate the above over every $x_j$, use the bound on the integrals of $a_2(x)$ and $e_2(x)$, and then
collect contributions of $\eta$, we see that
$$\int_{X^k} d\mu(x_1) \cdots d\mu(x_k) \,\left | \dt{A(x_i,x_j) + E(x_i,x_j)} - \dt{A(x_i,x_j)} \right | \leq
k^{k/2} (C_1 C_2)^k ((1+\eta)^k - 1).$$
Since $0 \leq \eta \leq 1$ we have that $(1+\eta)^k -1 \leq \eta 2^k$. Consequently,
$$\left |\dt{I+A+E}_{L^2(X,\mu)} - \dt{I+A}_{L^2(X,\mu)} \right | \leq
\eta \sum_{k \geq 1} \frac{k^{k/2}}{k!} (2C_1 C_2)^k =: \eta C_3. \qedhere $$
\end{proof}

We will use the following nomenclature for matrix kernels in the proof of convergence.
\begin{defn} \label{kerneldef}
Let $M_1, M_2, \ldots$, be a sequence of matrices where $M_N$ is an $N \times N$ matrix understood
in terms of the $p \times p$ block structure above. Let $\widetilde{M}_N$ be the embedding of $M_N$
into $\hb$ as in \eqref{matrixembed}, and $F_N$ the rescaling according to \eqref{Fredholmembed}.
\begin{itemize}
	\item The matrices are \emph{good} if there are non-negative, bounded and integrable functions
	$g_1(x)$, $\ldots, g_p(x)$ on $\R$ such that following holds. For every $N$,
	$$ |F_N(r;u,\, s,v)| \leq g_r(u) g_s(v) \quad \text{for every}\;\; 1\leq r,s \leq p \;\; \text{and}\;\; u,v \in \R.$$
	\item The matrices are \emph{convergent} if there is a matrix kernel $F$ on $\hb$ such that
	the following holds uniformly in $u,v$ restricted to compact subsets of $\R$.
	$$ \lim_{N \to \infty}\, F_N(r,u; \,s,v) = F(r,u;\, s,v) \quad \text{for every}\;\; 1\leq r,s \leq p.$$
	\item The matrices are \emph{small} if there is a sequence $\eta_N \to 0$ and functions $g_1, \ldots, g_p$ as for
	good matrices such the the following holds.
	$$ |F_N(r;u,\, s,v)| \leq \eta_N \, g_r(u) g_s(v) \quad \text{for every}\;\; 1\leq r,s \leq p \;\; \text{and}\;\; u,v \in \R.$$
\end{itemize}
\end{defn}
Remark in the above definition that $u$ and $v$ will be negative or positive depending on the blocks,
and we can think of $F_N$ being zero outside the stipulated domain. It will be convenient to hide dependence
of parameter $N$ when discussing matrices and call a matrix good/convergent/small with $N$ understood implicitly.
The following are straightforward consequences of the definitions, dominated convergence theorem and Lemma \ref{lem:detest}.
\begin{enumerate}
	\item If $M_1, M_2, \ldots$ are good and convergent matrices with limit $F$ on $\hb$ then
	$$\dt{I+ F_N}_{\hb} \to \dt{I+F}_{\hb} < \infty.$$
	$F$ satisfies the same goodness bound as its approximants.
	\item If $M_1, M_2, \ldots$ are good and $S_1, S_2, \ldots$ are small then
	$$ \dt{I+F_{M_N} + F_{S_N}}_{\hb} - \dt{I + F_{M_N}}_{\hb} \to 0,$$
	where $F_{M_N}$ is the rescaling of $M_N$ according to \eqref{Fredholmembed} and similarly for $F_{S_N}$.
\end{enumerate}

\subsection{Preparation} \label{sec:asy2}
In order to apply the method of steepest descent to the determinant from Theorem \ref{thm:2},
we have to identify the limit of matrix kernels and also establish some decay estimates
for them at infinity, so that the series expansion of the Fredholm determinant converges.
To do this we need three things regarding the function $G(w \vt n,m,a)$.

First, we need to understand the asymptotic behaviour of $G(w \vt n,m,a)$ locally around its
critical point under KPZ scaling of $n,m,a$. This is the content of Lemma \ref{lem:Glimit}.
Second, we have to find descent contours for $\gamma_{\tau}$ and $\gamma_R(1)$
that appear in the description of $A(\theta)$ and $B(\theta)$. These are provided by Definition \ref{def:contours}.
Third, we have to establish decay of $G$ along these contours, which is the subject of Lemma \ref{lem:Gest}.

Recall $G(w \vt n,m,a)$ from \eqref{Gnmx} with the indices scaled as
\begin{align}\label{klbscaling}
	n&=K-c_1x K^{2/3}+c_0vK^{1/3},\\ \nonumber
	m&=K+c_1x K^{2/3}, \\ \nonumber
	a&=c_2K+c_3\xi K^{1/3}.
\end{align}
The constants $c_i$ are given by \eqref{scalingconstants}. When $n=m$ and $a = c_2 n$ the function
$$ \log G(w \vt n,m,a)=n\log w+(m+a)\log(1-w)- m \log(1-\frac{w}{1-q}) - \log(G^{*}(1-\sqrt{q}\vt n,m,a))$$
has a double critical point at
\begin{equation}\label{wc}
w_c=1-\sqrt{q}.
\end{equation}

\begin{lem}\label{lem:Glimit}
	Assume that we have the scaling \eqref{klbscaling} and that $|x|, |\xi|, |v| \leq L$ for a fixed $L$.
	Then uniformly in $x,\xi, v$ and $w \in \C$ restricted to compact subsets, 
	\begin{equation}\label{Hstarlimit}
	\lim_{K\to\infty} G\left(w_c+\frac{c_4 \cdot w}{K^{1/3}} \, \Big | \, n,m,a \right)=\G(w \vt 1, x,\xi-v) =
	\exp \left \{ \frac{1}{3}w^3+ x w^2-(\xi-v)w \right\},
	\end{equation}
	where
	\begin{equation}\label{c4}
	c_4=\frac{q^{1/3}(1-\sqrt{q})}{(1+\sqrt{q})^{1/3}} = \frac{w_c}{c_0}.
	\end{equation}
\end{lem}
The lemma is proved in Lemma 5.3 of \cite{JoTwo} by considering the Taylor expansion of $\log G$
with the scaling \eqref{KPZscaling}.

The circular contours $\gamma$ around 0 and 1 will be chosen according to the following two contours
with appropriate values for the parameters.
\begin{defn} \label{def:contours}
Let $K>0$ and $0 < d < K^{1/3}$. For $|\sigma|\le\pi K^{1/3}$, set
\begin{equation}\label{w1}
w_0(\sigma)=w_0(\sigma;d)=w_c(1-\frac d{K^{1/3}})e^{\mathrm{i}\sigma K^{-1/3}}
\end{equation}
and
\begin{equation}\label{w2}
w_1(\sigma)=w_1(\sigma;d)=1- \sqrt{q}(1-\frac d{K^{1/3}})e^{\mathrm{i}\sigma K^{-1/3}}\,.
\end{equation}
Thus, $w_0$ is a circle around the origin of radius $w_c(1-\frac d{K^{1/3}})$
and $w_1$ is a circle around 1 of radius $ \sqrt{q}(1-\frac d{K^{1/3}})$.
\end{defn}

Recall the notation $(v)_{+} = \max \{ v,0\}$ and $(v)_{-} = \max \{-v,0\}$.
\begin{lem}\label{lem:Gest}
	Assume $|x|, |\xi| \leq L$ for some fixed $L>0$. Consider the scaling \eqref{klbscaling} where $v$ is such that $n \geq 0$.
	There are positive constants $C_0, C_1,C_2,C_3,C_4, C_5$ that depend on $q$ and $L$ such that the following holds.
	Let $0 < \delta \leq C_0$.  There are positive constants $\mu_1$ and $\mu_2$ that depend on $q,L, \delta$ with
	the following property. If $K \geq C_5$, there is a choice of $d = d(v)$ such that
	\begin{equation}\label{Gest0}
	\big|G(w_0(\sigma;d(v)) \vt n,m,a)\big|^{-1}\leq C_3e^{-C_4\sigma^2-\mu_1(v)_{-}^{3/2}+\mu_2(v)_+}
	\end{equation}
	and
	\begin{equation}\label{Gest1}
	\big|G(w_1(\sigma;d(v))\vt n,m,a)\big|\leq C_3e^{-C_4\sigma^2-\mu_1(v)_{-}^{3/2}+\mu_2(v)_+}
	\end{equation}
	for every $|\sigma|\leq \pi K^{1/3}$. If $v \geq 0$ then $d(v)$ may be any point in the interval
	$[C_1, C_2K^{1/3}]$ ($C_1 < C_2 < 1$). If $v < 0$ then $d(v)$ may be any point in the interval
	$[C_1 + \delta \cdot (v)_{-}^{1/2}, C_2 K^{1/3}]$. 
\end{lem}
The lemma is proved in combination of Lemmas 5.6 and 5.7 in \cite{JoTwo}.
It is based on a direct critical point analysis of the real parts of $\log G(w_0(\sigma,d))$
and $\log G(w_1(\sigma,d))$ with the scaling \eqref{KPZscaling}.

Now we mention the choice of conjugation constant $\mu$ from \eqref{conjugation2}.
During asymptotic analysis we have to set $\mu$ and the parameter $\delta$ from Lemma \ref{lem:Gest}
such that they satisfy the following bounds (in addition to $0 < \delta \leq C_0$).
\begin{equation} \label{mudelta}
\delta < C_2 c_0^{1/2}t_p^{-1/2} \cdot \min_k \, \{(\D_k t)^{1/2} \} \quad \text{and} \quad \mu > \mu_2 \cdot \max_k \,\{(\D_k t)^{-1/3}\}.
\end{equation}
So long as $t_k, |x_k|, |\xi_k| \leq L$, these constraints depend only on $q$, $L$ and $\min_k \, \{ \D_k t\}$.

The goodness and smallness of matrices will be certified as follows. Write
\begin{equation} \label{psi}
\psi(x) = -\mu_1 \cdot (x)_{-}^{3/2} + \mu_2 \cdot (x)_{+}
\end{equation}
where $\mu_1$ and $\mu_2$ are according to Lemma \ref{lem:Gest} and $\delta$ is set to satisfy \eqref{mudelta}.
(The parameters $t_k$, $x_k$ and $\xi_k$ from \eqref{KPZscaling} are now fixed.) Suppose $\D \geq \min_k \, \{ (\D_k t)^{1/3} \} > 0$
and $\mu$ is as in \eqref{mudelta}. Then,
\begin{align} \label{psiproperty}
(1) & \;\; e^{-\mu x + \psi(x/\D)} \leq e^{\frac{4 (\mu \D)^3}{27 \mu_1^2}} \;\;\text{for}\;\; x \in \R.\; \text{So it is bounded.} \\ \nonumber
(2) & \;\; \int_{-\infty}^{\infty} dx\, e^{-\mu x + \psi(x/\D)} =
\int_{- \infty}^0 e^{-\mu_1 \cdot (x/\D)_{-}^{3/2} + \, \mu \cdot (x)_{-}} + \, \int_{0}^{\infty} dx\, e^{(\frac{\mu_2}{\D} - \mu)\cdot (x)_{+}} < \infty. \\ \nonumber
(3) & \;\; e^{-\mu x + \psi(x/\D)} \to 0 \;\; \text{as}\;\; x \to \pm \infty. \\ \nonumber 
(4) & \;\; \int_{-\infty}^{0} dx\, e^{\mu x + \psi(x/\D)} < \infty.
\end{align}

\subsection{Convergence of the determinant} \label{sec:asy3}
In order to prove Theorem \ref{thm:1} by using Theorem \ref{thm:2}, it suffices to show
there is uniform convergence of $\dt{I + A(\mathbold{\theta}) + B(\mathbold{\theta})}$ to $\dt{I - F(\mathbold{\theta})}_{\hb}$
in terms of $\mathbold{\theta}$ over the integration contour $\gamma_{r}^{p-1}$.
Parameter $\mathbold{\theta}$ enters the matrices in terms of $\theta(r \vt \veps)$ and $\Theta(r \vt k)$
from \eqref{thetaeps} and \eqref{theta}. These quantities will play no role in the asymptotical analysis as all estimates
will involve the basic matrices $L[\cdots]$. So all error terms will be uniform in $\mathbold{\theta}$,
and we may suppress $\mathbold{\theta}$ from notation as convenient.

The matrix $A$ is good and convergent but $B$ is not. (Under KPZ scaling, entries of $B$ converge to
entries of the form $\Ai(v-u)$, which does not have finite Fredholm determinant).
On the other hand, $B^{p-1} = 0$ because $B$ is strictly block-lower-triangular with last two column blocks being zero.
So $(I+B)^{-1} = I - B + B^2 + \cdots + (-1)^{p-2}B^{p-2}$. We may then consider instead the determinant
of $I +A - AB + \cdots + (-1)^{p-2}AB^{p-2}$. These matrices turn out to be small from
$AB^2$ onward, and the first 2 are good and convergent. These considerations motivate the following.

Since $\dt{I - B} = 1$,
$$ \dt{I + A + B} = \dt{I + A +B} \dt{I - B} = \dt{I - B^2 + A - AB}.$$
We will see in Lemma \ref{lem:Bsquare} that $B^2 = B_1 - B_2$, where $B_1$ is good and convergent.
Proposition \ref{prop:Alimit} will prove that $A$ is good and convergent. We will also find, from
Proposition \ref{prop:ABlimit}, that $AB = (AB)_g + (AB)_s$ with $(AB)_g$ being good and convergent
while $(AB)_s$ is small. Thus, under KPZ scaling, as $T \to \infty$,
$$\dt{I + A +B} \approx \dt{I +B_2 + (A - (AB)_g - B_1)}.$$
Proposition \ref{prop:small} will prove that $P = A - (AB)_g - B_1$ is such that $P B_2$ is small. So
$$\dt{I + B_2 + P} \approx \dt{I + B_2 + P + P B_2} = \dt{I+P} \dt{I+B_2}.$$
The matrix $B_2$ is strictly block-lower-triangular due to $B$ being such. So $\dt{I+B_2} = 1$.
This means that
$$\dt{I + A + B} \approx \dt{I + P},$$
and the latter determinant converges under KPZ scaling. The limit of $P$ is precisely the matrix kernel $F$ from \eqref{F}.
So we will have proved Theorem \ref{thm:1} after proving the upcoming lemmas and propositions.

\begin{lem} \label{lem:Bsquare}
The matrix $B^2 = B_1 - B_2$, where $B_1$ and $B_2$ are as follows.
Recall $w_c = 1- \sqrt{q}$, $r^{*} = \min \, \{r,p-1\}$ and likewise for $s^{*}$.
\begin{align*}
B_1(r,i;s,j) = \sum_{k=0}^p (1 + \Theta(r \vt k)) \cdot (1 + \Theta(k \vt s)) \cdot L[{\scriptstyle k,k \vt \emptyset}]\, (r,i;s,j).\\
B_2(r,i;s,j = \sum_{k=0}^p (1 + \Theta(r \vt k)) \cdot (1 + \Theta(k \vt s)) \cdot (\sm L)[{\scriptstyle k,k \vt \emptyset}]\, (r,i;s,j).
\end{align*}
The matrix $(\sm L)$ is given by
\begin{align*}
(\sm L)[{\scriptstyle k,k \vt \emptyset}]\, (r,i;s,j) &= \ind{s < k < r^{*}}\, c(r,i;s,j)\, 
\frac{1}{w_c} \oint\limits_{\gamma_{\tau_1}} d\zeta_1 \oint\limits_{\gamma_{\tau_2}} d\zeta_2 \\
& \frac{(\zeta_1 - \zeta_2)^{-1}}{G \big(\zeta_1 \vt i - n_{k-1},\, \D_{k, r^{*}} (m,a)\big) \, G \big(\zeta_2 \vt n_{k-1}-j+1, \, \D_{s, k}(m,a)\big)}.
\end{align*}
The matrix $B_1$ is good and convergent in the KPZ scaling limit with limiting kernel on $\hb$ given by
$$F^{(0)}(r,u;s,v) = \sum_{k=0}^p (1 + \Theta(r \vt k)) \cdot (1 + \Theta(k \vt s)) \cdot F[{\scriptstyle k,k \vt \emptyset}]\, (r,i;s,j).$$
(Recall $F$s from Definition \eqref{def:F}.)
\end{lem}

\begin{prop} \label{prop:Alimit}
The matrix $A$ is good and convergent due to the following. Suppose $0 \leq k_1 < k_2 \leq p$.
\begin{enumerate}
\item The matrix $L^{\veps}[{\scriptstyle k_1, k_2 \vt (k_1, k_2]}]$ is good and convergent with limit
$(-1)^{k_2-k_1} \, F^{\veps}[{\scriptstyle k_1, k_2 \vt (k_1, k_2]}]$.
\item The matrix $L^{\veps}[{\scriptstyle k_1 \vt (k_1, k_2]}]$ is good and convergent with limit
$(-1)^{k_2-k_1} \, F^{\veps}[{\scriptstyle k_1 \vt (k_1, k_2]}]$.
\item The matrix $L[{\scriptstyle k,k \vt \emptyset}]$ is good and convergent with limit $F[{\scriptstyle k,k \vt \emptyset}]$.
\item The matrix $L[{\scriptstyle p \vt p}]$ is good and convergent with limit $- F[{\scriptstyle p \vt p}]$.
\end{enumerate}
\end{prop}

\begin{lem} \label{lem:LcdB}
Suppose $0 \leq k_1 < k_2 < p$. We have
$$L^{\veps}[{\scriptstyle k_1, k_2 \vt (k_1, k_2]}] \cdot B = \sum_{k_3=0}^p \big (1 + \Theta(k_3 \vt s)\big) \Big [
L^{\veps}[{\scriptstyle k_1, k_2,k_3 \vt (k_1, k_2]}] - (\sm L)^{\veps}[{\scriptstyle k_1, k_2,k_3 \vt (k_1, k_2]}] \Big ].$$
\begin{align*}
& L^{\veps}[{\scriptstyle k_1, k_2,k_3 \vt (k_1, k_2]}]\, (r,i;s,j) = \ind{k_1 < r^{*}, \, s < k_3 < k_2} \, c(r,i;s,j)\, \times \\
& \qquad \frac{1}{w_c} \oint\limits_{\gamma_{\tau_1}} d\zeta_1 \oint\limits_{ \gamma_{\tau_2}} d\zeta_2 \oint\limits_{\gamma_{\tau_3}} d\zeta_3
\oint\limits_{\vec{\gamma}_{R^{\veps}}} dz_{k_1+1} \cdots dz_{k_2} \, \left ( \frac{1-\zeta_1}{1-z_1} \right)^{\ind{k_1=0}} \times \\
& \frac{\prod_{k_1 < k \leq k_2} G \big(z_k \vt \D_k (n,m,a) \big) \prod_{k_1 < k < k_2}(z_k - z_{k+1})^{-1} \,
(z_{k_1+1}-\zeta_1)^{-1} (z_{k_2}-\zeta_2)^{-1} (\zeta_2-\zeta_3)^{-1}}
{G \big(\zeta_1 \vt i-n_{k_1}, \D_{k_1, r^{*}}(m,a)\big) \,G \big(\zeta_2 \vt \D_{k_3,k_2}(n,m,a)\big)
\,G \big(\zeta_3 \vt n_{k_3}-j+1, \D_{s, k_3}(m,a)\big )}.
\end{align*}
The contours are arranged such that $\tau_2 < \tau_1, \tau_3 < 1- \sqrt{q}$.
Also, $\vec{\gamma}_{R^{\veps}} = \gamma_{R_{k_1+1}}(1) \times \cdots \times \gamma_{R_{k_2}}(1)$,
and these are same as the equally denoted contours in $L^{\veps}[{\scriptstyle k_1, k_2 \vt (k_1, k_2]}]$ (see Definition \eqref{def:Lcd}).
\begin{align*}
& (\sm L)^{\veps}[{\scriptstyle k_1, k_2,k_3 \vt (k_1, k_2]}]\, (r,i;s,j) = \ind{k_1 < r^{*}, \, s < k_3 < k_2} \, c(r,i;s,j) \, \times  \\
& \qquad \frac{1}{w_c} \oint\limits_{\gamma_{\tau_1}} d\zeta_1 \oint\limits_{\gamma_{\tau_2}} d\zeta_2  \oint\limits_{\gamma_{\tau_3}} d\zeta_3
\oint\limits_{\vec{\gamma}_{R^{\veps}}} dz_{k_1+1} \cdots dz_{k_2} \left ( \frac{1-\zeta_1}{1-z_1} \right)^{\ind{k_1=0}} \times \\
& \frac{\prod_{k_1 < k \leq k_2} G \big(z_k \vt \D_k (n,m,a)\big) \prod_{k_1 < k < k_2}(z_k - z_{k+1})^{-1} \,
(z_{k_1+1}-\zeta_1)^{-1} (z_{k_2}-\zeta_2)^{-1} (\zeta_2-\zeta_3)^{-1}}
{G \big(\zeta_1 \vt i-n_{k_1}, \D_{k_1, r^{*}}(m,a)\big) \, G \big(\zeta_2 \vt n_{k_2} - n_{k_3-1}, \D_{k_3,k_2}(m,a)\big) \,
G \big(\zeta_3 \vt n_{k_3-1}-j+1, \D_{s, k_3}(m,a) \big)}.
\end{align*}
The difference here from $L^{\veps}[{\scriptstyle k_1, k_2,k_3 \vt (k_1, k_2]}]$ is that the number $n_{k_3}$ is replaced by $n_{k_3-1}$
in the second and third $G$-functions of the denominator.

The matrix $L^{\veps}[{\scriptstyle k_1, k_2,k_3 \vt (k_1, k_2]}]$ is good and convergent.
Its limit is $(-1)^{k_2-k_1}\, F^{\veps}[{\scriptstyle k_1, k_2,k_3 \vt (k_1, k_2]}]$. The matrix
$(\sm L)^{\veps}[{\scriptstyle k_1, k_2,k_3 \vt (k_1, k_2]}]$ is small.

When $k_2 = p$ there is an additional term in the representation above:
\begin{align*} L^{\veps}[{\scriptstyle k_1, p \vt (k_1, p]}] \cdot B &= \sum_{k_3=0}^p \big (1 + \Theta(k_3 \vt s)\big) \big [
L^{\veps}[{\scriptstyle k_1, p,k_3 \vt (k_1, p]}] \big ] - (1 + \Theta(p \vt s))\cdot L^{\veps}[{\scriptstyle k_1, p,p-1 \vt (k_1, p]}] \\
& \;\; - \sum _{k_3=0}^{p-1} (\sm L)^{\veps}[{\scriptstyle k_1, k_2,k_3 \vt (k_1, k_2]}] \Big ].
\end{align*}
\end{lem}

\begin{lem} \label{lem:JcdB}
Suppose $0 \leq k_1 < k_2 < p$. We have
$$L^{\veps}[{\scriptstyle k_1 \vt (k_1, k_2]}] \cdot B = \big (1 + \Theta(k_2 \vt s) \big) \Big [ L^{\veps}[{\scriptstyle k_1, k_2 \vt (k_1, k_2]}] -
(\sm L)^{\veps}[{\scriptstyle k_1, k_2 \vt (k_1, k_2]}] \Big ], \;\; \text{where}$$
\begin{align*}
& (\sm L)^{\veps}[{\scriptstyle k_1, k_2 \vt (k_1, k_2]}]\, (r,i;s,j) = \ind{k_1 < r^{*},\, s < k_2} \, c(r,i;s,j)\, \frac{1}{w_c}
\oint\limits_{\gamma_{\tau_1}} d\zeta_1 \oint\limits_{\gamma_{\tau_2}} d\zeta_2
\oint\limits_{\vec{\gamma}_{R^{\veps}}} dz_{k_1+1} \cdots d z_{k_2} \\
&\frac{\prod_{k_1 < k < k_2} G\big(z_k \vt \D_k (n,m,a)\big) \, G \big(z_{k_2} \vt 0,\D_{k_2}(m,a)\big)\prod_{k_1 < k < k_2} (z_k - z_{k+1})^{-1}
\left ( \frac{1-\zeta_1}{1-z_1} \right)^{\ind{k_1=0}}}
{G\big(\zeta_1 \vt i-n_{k_1}, \D_{k_1, r^{*}}(m,a)\big) \, G \big(\zeta_2 \vt n_{k_2-1}-j+1, \, \D_{s,k_2}(m,a)\big)\,
(z_{k_1+1}-\zeta_1)\, (z_{k_2}-\zeta_2)}.
\end{align*}
The contours are as in the lemma above. The matrix $(\sm L)^{\veps}[{\scriptstyle k_1, k_2 \vt (k_1, k_2]}]$ is small.
\end{lem}

\begin{lem} \label{lem:LkB}
Suppose $0 \leq k_1 \leq p$. We have
$$ L[{\scriptstyle k_1, k_1 \vt \emptyset}] \cdot B = \sum_{k_2 = 0}^p \big(1 + \Theta(k_2 \vt s)\big) \Big [ L[{\scriptstyle k_1, k_1, k_2 \vt \emptyset}] -
(\sm L)[{\scriptstyle k_1, k_1,k_2 \vt \emptyset}] \Big ], \;\; \text{where}$$
\begin{align*}
& L[{\scriptstyle k_1, k_1, k_2 \vt \emptyset}]\, (r,i;s,j) = \ind{k_1 < r^{*}, \, s < k_2 < k_1}\, c(r,i;s,j)\, \frac{1}{w_c}
\oint\limits_{\gamma_{\tau_1}} d\zeta_1 \oint\limits_{\gamma_{\tau_2}} d\zeta_2 \oint\limits_{\gamma_{\tau_3}} d\zeta_3 \\
& \frac{(\zeta_1 - \zeta_2)^{-1} (\zeta_2 - \zeta_3)^{-1}}
{G\big(\zeta_1 \vt i-n_{k_1}, \D_{k_1, r^{*}}(m,a)\big) \,G \big(\zeta_2 \vt \D_{k_2, k_1} (n,m,a)\big)\,
G \big(\zeta_3 \vt n_{k_2}-j+1, \D_{s,k_2}(m,a)\big)}.
\end{align*}
We arrange the radii $\tau_2 < \tau_1, \tau_3 < 1-\sqrt{q}$.
\begin{align*}
& (\sm L)[{\scriptstyle k_1, k_1, k_2 \vt \emptyset}]\, (r,i;s,j) = \ind{k_1 < r^{*}, \, s < k_2 < k_1}\, c(r,i;s,j)\, \frac{1}{w_c}
\oint_{\gamma_{\tau_1}} d\zeta_1 \oint_{\gamma_{\tau_2}} d\zeta_2 \oint_{\gamma_{\tau_3}} d\zeta_3 \\
& \frac{(\zeta_1 - \zeta_2)^{-1} (\zeta_2 - \zeta_3)^{-1}}
{G \big(\zeta_1 \vt i-n_{k_1}, \D_{k_1, r^{*}}(m,a)\big) \, G \big(\zeta_2 \vt n_{k_1}-n_{k_2-1}, \D_{k_2, k_1} (m,a)\big)\,
G \big(\zeta_3 \vt n_{k_2-1}-j+1, \, \D_{s,k_2}(m,a)\big)}.
\end{align*}
The difference from $L[{\scriptstyle k_1, k_1, k_2 \vt \emptyset}]$ is that the number $n_{k_2}$ is replaced by $n_{k_2-1}$
in the second and third $G$-functions of the denominator.

The matrix $L[{\scriptstyle k_1, k_1, k_2 \vt \emptyset}]$ is good and convergent with limit
$F[{\scriptstyle k_1, k_1, k_2 \vt \emptyset}]$. The matrix $(\sm L)[{\scriptstyle k_1, k_1, k_2 \vt \emptyset}]$ is small.
\end{lem}

\begin{lem} \label{lem:LpB}
For the matrix $L[{\scriptstyle p \vt p}]$ we have
\begin{align*}
L[{\scriptstyle p \vt p}] \cdot B(r,i;s,j) &=
\sum_{k=0}^p \big (1 + \Theta(k \vt s)\big) L[{\scriptstyle p,k \vt p}]\,(r,i;s,j) - \big (1 + \Theta(p \vt s) \big) L[{\scriptstyle p,p-1 \vt p}]\,(r,i;s,j) \\
& - \sum_{k=0}^p \big (1 + \Theta(k \vt s) \big) (\sm L)[{\scriptstyle p,k \vt p}]\,(r,i;s,j).
\end{align*}
The matrices $L[{\scriptstyle p,k \vt p}]$ and $(\sm L)[{\scriptstyle p,k \vt p}]$ are as follows.
\begin{align*}
& L[{\scriptstyle p,k \vt p}]\,(r,i;s,j) = \ind{r=p, \, s < k < p}\, c(r,i;s,j)\, \frac{1}{w_c}
\oint\limits_{\gamma_{\tau_2}} d\zeta_2 \oint\limits_{\gamma_{\tau_3}} d\zeta_3 \oint\limits_{\gamma_{R_p}(1)} d z_p \\
&\frac{G \big (z_p \vt n_p-i, \D_p(m,a)\big) (z_p-\zeta_2)^{-1} (\zeta_2 - \zeta_3)^{-1}}
{G \big (\zeta_2 \vt n_p-n_k, \D_{k,p}(m,a)\big) G \big( \zeta_3 \vt n_k-j+1, \D_{s,k}(m,a) \big)},\\
&(\sm L) [{\scriptstyle p,k \vt p}]\,(r,i;s,j) = \ind{r=p, \, s < k < p}\, c(r,i;s,j)\, \frac{1}{w_c}
\oint\limits_{\gamma_{\tau_2}} d\zeta_2 \oint\limits_{\gamma_{\tau_3}} d\zeta_3 \oint\limits_{\gamma_{R_p}(1)} d z_p \\
&\frac{G \big (z_p \vt n_p-i, \D_p(m,a)\big) (z_p-\zeta_2)^{-1} (\zeta_2 - \zeta_3)^{-1}}
{G \big (\zeta_2 \vt n_p-n_{k-1}, \D_{k,p}(m,a)\big) G \big( \zeta_3 \vt n_{k-1}-j+1, \D_{s,k}(m,a) \big)}.
\end{align*}
The radii are arranged such that $\tau_2 < \tau_3 < 1- \sqrt{q}$.
(The difference between $L[{\scriptstyle p, k, \vt p}]$ and $(\sm L)[{\scriptstyle p, k, \vt p}]$ is that the number
$n_k$ is changed to $n_{k-1}$ in the second and third $G$-functions of the denominator.)

The matrix $L[{\scriptstyle p,k \vt p}]$ is good and convergent with limit $- F[{\scriptstyle p,k \vt p}]$.
The matrix $(\sm L)[{\scriptstyle p,k \vt p}]$ is small.
\end{lem}

\begin{prop} \label{prop:ABlimit}
	The matrix $A B = (AB)_g + (AB)_s$, where $(AB)_g$ is good and convergent and $(AB)_s$ is small.
	This is due to the following reasons, which also provides the limit of $(AB)_g$. Recall from Definition \eqref{def:blockform}
	that $A = A_1 + A_2$. Then $(AB)_g = (A_1 B)_g + (A_2 B)_g$, given as follows.
	\begin{align*}
	&(A_1 B)_g(r,i;s,j) = \sum_{0\leq k_1,k_2 \leq p} \Theta(r \vt k_1)\cdot (1 + \Theta(k_2 \vt s)) \cdot
	L[{\scriptstyle k_1,k_1,k_2 \vt \emptyset}]\, (r,i;s,j). \\
	& (A_2 B)_g(r,i;s,j)  = \sum_{\substack{0 \leq k_1,k_2,k_3 \leq p, \, \veps \\ \mathrm{satisfies}\; \eqref{sumcond}}}
	(-1)^{\eps_{[k_1,k_2]} + k_1 + k_2^{*}} \cdot \theta(r \vt \veps) \, \times \\
	& \Big [ (1+\Theta(k_3 \vt s)) L^{\veps}[{\scriptstyle k_1,k_2,k_3 \vt (k_1,k_2]}]
	-\ind{k_2=p,\, k_3=p-1} (1 + \Theta(p \vt s)) L^{\veps}[{\scriptstyle k_1,p,p-1\vt (k_1,p]}] \; + \\
	& \ind{k_2 < p, k_3=p} (1 + \Theta(k_2 \vt s)) L^{\veps}[{\scriptstyle k_1, k_2 \vt (k_1,k_2]}] \; + \\
	& \ind{k_1=p-1, k_2 =p}(1 + \Theta(k_3 \vt s)) L[{\scriptstyle p,k_3 \vt p}]
	- \ind{k_1=p-1, k_2 =p, k_3 =p-1} (1+ \Theta(p\vt s)) L[{\scriptstyle p,p-1 \vt p}] \Big ] (c,i;s,j).
	\end{align*}
	The summation variables $k_i$ range over $0,1, \ldots, p$. The matrix $(AB)_s$ looks the same as $(AB)_g$
	except that every $L$ is replaced by $\sm L$.
\end{prop}

\begin{proof}
We see in Definition \ref{def:blockform} that $A$ is a weighted sum - involving the $\theta_k$s - of the matrices
$L[{\scriptstyle k,k \vt \emptyset}]$, $L^{\veps}[{\scriptstyle k_1,k_2 \vt (k_1,k_2]}]$,
$L^{\veps}[{\scriptstyle k_1 \vt (k_1,k_2]}]$ and $L[{\scriptstyle p \vt p}]$. When we multiply $A$ by $B$ we replace
every $L[{\scriptstyle \cdots}]$ by $L[{\scriptstyle \cdots }] \cdot B$. Then if we substitute the representation of
these matrices by using Lemmas \ref{lem:LcdB}, \ref{lem:JcdB}, \ref{lem:LkB} and \ref{lem:LpB}, we get the
representation $(AB)_g + (AB)_s$ as given by the statement of the proposition.
\end{proof}

Lemma \ref{lem:Bsquare} along with Propositions \ref{prop:Alimit} and \ref{prop:ABlimit} imply that the matrix
$P = A - (AB)_g - B_1$ has limit $F$ from \eqref{F}. Specifically, the limit of $B_1$ is $F^{(0)}$. The limit of
$A_1$ is $F^{(1)}$ and that of $A_2$ is $F^{(2)}$. The limit of $(A_1 B)_g$ is $F^{(3)}$ and the one of $(A_2 B)_g$ is $F^{(4)}$.
Let us also remark that when comparing the matrix $A$ with $F$, we see the factors
$(-1)^{\eps_{[k_1,k_2]} + k_1 + k_2^{*}}$ have become $(-1)^{\eps_{[k_1,k_2]} + \ind{k_2 = p}}$.
This is because limits of the $L^{\veps}$ are of the form $(-1)^{k_2-k_1}F^{\veps}$,
and $k_2^{*} + k_2 = 2k_2 - \ind{k_2 = p}$. Likewise for $L[{\scriptstyle p \vt p}]$ with $k_1 = p-1$ and $k_2=p$.

We then arrive at the conclusion of Theorem \ref{thm:1} once we have proved
\begin{prop} \label{prop:small}
The matrix $P B_2$ is small, where $P = A - (AB)_g - B_1$ and $B_2$ is from Lemma \ref{lem:Bsquare}.
\end{prop}
The proof of this is in the next section. For the remainder of this section we will prove Proposition \ref{prop:Alimit} and
the aforementioned lemmas. The proofs will be on a case by case basis, where we consider each of the three types
of matrices $L[{\scriptstyle k,k, \vt \emptyset}]$, $L[{\scriptstyle k_1, k_2 \vt (k_1,k_2]}]$ and $L[{\scriptstyle k_1 \vt (k_1,k_2]}]$,
and then prove the propositions claimed about them. 

The following lemma will be used again and again to multiply matrices by $B$.
\begin{lem} \label{lem:multiply}
Suppose $0 \leq N_1 < N_2$ are integers and $w_1 \neq w_2$ belong to $\C \setminus \{0,1,1-q\}$. Then,
\begin{align*}
&\sum_{N_1 < \ell \leq N_2} \frac{1}{G(w_1 \vt n-\ell + 1, m,a)\, G(w_2 \vt \ell - n', m',a')} = \frac{w_c}{w_1 - w_2} \,\times \\
& \left [ \frac{1}{G(w_1 \vt n-N_2, m,a)\, G(w_2 \vt N_2-n',m',a')} - \frac{1}{G(w_1 \vt n-N_1, m,a)\, G(w_2 \vt N_1-n',m',a')} \right ]
\end{align*}
\end{lem}

\begin{proof}
Due to the group property of $G$, the sum over $\ell$ can be written as
$$ \frac{1}{G(w_1 \vt n,m,a)\, G(w_2 \vt -n',m',a')}\, 
\sum_{N_1 < \ell \leq N_2} \left (\frac{w_1}{w_c}\right)^{\ell-1} \left ( \frac{w_c}{w_2} \right)^{\ell}.$$
The geometric sum evaluates to
\begin{align*}
& \frac{w_c}{w_1-w_2}\left [ (w_1/w_2)^{N_2} - (w_1/w_2)^{N_1} \right ] = \\
& \frac{w_c}{w_1-w_2} \left [ \frac{1}{G(w_1 \vt -N_2,0,0) \, G(w_2 \vt N_2,0,0)} - \frac{1}{G(w_1 \vt -N_1,0,0) \, G(w_2 \vt N_1,0,0)} \right ].
\end{align*}
Then by the group property we obtain the expression on the r.h.s.~of the identity .
\end{proof}

\paragraph{\textbf{Proof of Lemma \ref{lem:Bsquare}}}
We have that
$$B^2(r,i;s,j) = \sum_{k=0}^p \, \sum_{n_{k-1} < \ell \leq n_k} B(r,i;k,\ell) B(k,\ell; s,j).$$ 
Let us recall
$$B(r,i;s,j) = \ind{s < r^{*}} c(r,i;s,j) \frac{1+ \Theta(r \vt s)}{w_c} \oint\limits_{\gamma_{\tau}} d\zeta
\frac{1}{G \big(\zeta \vt i-j+1, \D_{s,r^{*}}(m,a)\big)}.$$
The conjugation factor satisfies $c(r,i;k,\ell) c(k,\ell;s,j) = c(r,i;s,j)$. Therefore,
\begin{align*}
B^2(r,i;s,j) = &c(r,i;s,j) \sum_{k=0}^p \ind{k < r^{*},\, s < k^{*}}\, (1+ \Theta(r \vt k)) \cdot (1 + \Theta(k\vt s))
\frac{1}{w_c^2} \oint\limits_{\gamma_{\tau_1}} d\zeta_1 \oint\limits_{\gamma_{\tau_2}} d\zeta_2 \\
& \sum_{n_{k-1} < \ell \leq n_k} \frac{1}{G \big(\zeta_1 \vt i-\ell + 1, \D_{k,r^{*}}(m,a)\big)\, G \big(\zeta_2 \vt \ell - j+1, \D_{s,k^{*}}(m,a)\big)}.
\end{align*}
Observe that $k^{*} = k$ because $k < r^{*} < p$. By Lemma \ref{lem:multiply}, the sum over $\ell$ gives
the difference of the integrand of $L[{\scriptstyle k,k\vt \emptyset}](r,i;s,j)$ from that of $(\sm L) [{\scriptstyle k,k\vt \emptyset}](r,i;s,j)$.
Consequently, the expressions for $B_1$ and $B_2$ follow and we have $B^2 = B_1 - B_2$.
That $B_1$ is good and convergent will follow due to every $L[{\scriptstyle k,k \vt \emptyset}]$ being such,
which will be shown in the proof of Proposition \ref{prop:ABlimit} below. \qed

Throughout the remaining argument we will assume the following.
\begin{enumerate}
	\item The parameters $t_k, x_k, \xi_k$ are bounded in absolute value by $L$ and $\min_k \, \{\D_k t\} > 0$.
	\item $C_{q,L}$ is a constant whose value may change from one appearance to the next, but depends on $q$ and $L$ only.
\end{enumerate}

\subsubsection{Proof of claims regarding $L^{\veps}[{\scriptstyle k_1, k_2 \vt (k_1, k_2]}]$}
The matrix $L^{\veps}[{\scriptstyle k_1, k_2 \vt (k_1, k_2]}]$ has the from
\begin{align} \label{Alimit}
& L^{\veps}[{\scriptstyle k_1, k_2 \vt (k_1, k_2]}](r,i;s,j) = \ind{k_1 < r^{*}, s^{*} < k_2} c(r,i;s,j) \frac{1}{w_c}
\oint\limits_{\gamma_{\tau_1}} d\zeta_1 \oint\limits_{\gamma_{\tau_2}} d\zeta_2 \\ \nonumber
&\frac{f(\zeta_1, \zeta_2)}{G(\zeta_1 \vt i-n_{k_1}, \D_{k_1,r^{*}}(m,a) )\, G(\zeta_2 \vt n_{k_2}-j+1, \D_{s^{*},k_2}(m,a))}, \;\;\text{where} \\ \nonumber
& f(\zeta_1, \zeta_2) = \oint\limits_{\vec{\gamma}_{R^{\veps}}} dz_{k_1+1} \cdots dz_{k_2}\,
\frac{\prod_{k_1 < k \leq k_2} G(z_k \vt \D_k (n,m,a))\, \left ( \frac{1-\zeta_1}{1-z_1}\right)^{\ind{k_1=0}}}
{\prod_{k_1 < k <k_2} (z_k - z_{k+1}) \, (z_{k_1+1}-\zeta_1) (z_{k_2} - \zeta_2)}.
\end{align}

Let us fix $k_1,k_2$ and $\veps$. Let $F_T$ be the KPZ re-scaling of our matrix according to \eqref{Fredholmembed}.
The indices $i$ and $j$ on the $(r,s)$-block are re-scaled as
\begin{equation} \label{ijrescale}
i = n_{r^{*}} + \lceil \nu_T u \rceil \quad \text{and}\quad j = n_{s^{*}} + \lceil \nu_T v \rceil\,.
\end{equation}
It is convenient to ignore the rounding as it makes no difference in the asymptotic analysis. Consequently,
\begin{align} \label{inmarescale}
i-n_{k_1} &=  \D_{k_1, r^{*}} t \, T - c_1 (\D_{k_1, r^{*}}x) \cdot (\D_{k_1, r^{*}} t\, T)^{\frac{2}{3}} +
c_0 \frac{u}{(\D_{k_1, r^{*}} t)^{1/3}} (\D_{k_1, r^{*}} t\, T)^{\frac{1}{3}}\,, \\ \nonumber
\D_{k_1, r^{*}} m & = \D_{k_1, r^{*}} t T + c_1 (\D_{k_1, r^{*}}x) \cdot(\D_{k_1, r^{*}} t\,T)^{\frac{2}{3}}\,, \\ \nonumber
\D_{k_1, r^{*}} a & = c_2 \, \D_{k_1, r^{*}}t\,T + c_3 (\D_{k_1, r^{*}} \xi) \cdot (\D_{k_1, r^{*}} t\,T)^{\frac{1}{3}}.
\end{align}
Similarly,
\begin{align} \label{jnmarescale}
n_{k_2}-j &=  \D_{s^{*},k_2} t\,T - c_1 (\D_{s^{*},k_2}x) \cdot (\D_{s^{*},k_2} t\,T)^{\frac{2}{3}}
+ c_0 \frac{-v}{(\D_{s^{*},k_2} t)^{1/3}} (\D_{s^{*},k_2} t\,T)^{\frac{1}{3}}\,, \\ \nonumber
\D_{k_1, r^{*}} m & = \D_{s^{*},k_2} t\,T + c_1 (\D_{s^{*},k_2}x) \cdot (\D_{s^{*},k_2} t\,T)^{\frac{2}{3}}\,, \\ \nonumber
\D_{k_1, r^{*}} a & = c_2 \, \D_{s^{*},k_2} t\,T + c_3 (\D_{s^{*},k_2} \xi) \cdot (\D_{s^{*},k_2} t\,T)^{\frac{1}{3}}.
\end{align}
We note that $\D_{k_1, r^{*}} t > 0$ and $\D_{s^{*},k_2} t > 0$ due to the conditions $k_1 < r^{*}$ and $s^{*} < k_2$.

Recalling Definition \ref{def:contours}, choose the contours $\gamma_{\tau_1}$ and $\gamma_{\tau_2}$ as follows.
\begin{equation*} \label{zetacontours}
\gamma_{\tau_1} = w_0(\sigma_1, d_1) \;\; \text{with}\;\; K \coloneqq \D_{k_1, r^{*}} t\, T; \;\;
\gamma_{\tau_2} = w_0(\sigma_2, d_2)\;\; \text{with}\;\; K \coloneqq \D_{s^{*},k_2} t\, T.
\end{equation*}
The choices for $d_1$ and $d_2$ will be made later.

With the re-scaling \eqref{ijrescale} the conjugation factor satisfies
\begin{equation} \label{rescaleconj}
c(r,i;s,j) = e^{\mu(v-u)} \, (1 + C_{q,L} T^{-1/3}).
\end{equation}

\paragraph{\textbf{Proof that $L^{\veps}[{\scriptstyle k_1, k_2 \vt (k_1, k_2]}]$ is good}}
From Lemma \ref{lem:Gest} we see there is a choice of $d_1 = d(u)$ such that we have
the following uniformly in $\zeta_1 = \zeta_1(\sigma_1) \in w_0(\sigma_1, d_1)$.
$$ |G(\zeta_1(\sigma_1) \vt i-n_{k_1}, \D_{k_1,r^{*}}(m,a))|^{-1} \leq C_3 e^{-C_4 \sigma_1^2 + \Psi \big(u/(\D_{k_1,r^{*}}t)^{1/3} \big)}.$$
Recall $\Psi(x) = -\mu_1 \cdot (x)_{-}^{3/2} + \mu_2 \cdot (x)_{+}$. Also, there is a choice of $d_2 = d(-v)$
such that the following holds uniformly in $\zeta_2 = \zeta_2(\sigma_2) \in w_0(\sigma_2, d_2)$.
$$ |G(\zeta_2(\sigma_2) \vt n_{k_2}-j+1, \D_{s^{*},k_2}(m,a))|^{-1} \leq C_3 e^{-C_4 \sigma_2^2 + \Psi \big(-v /(\D_{s^{*},k_2} t)^{1/3}\big)}.$$
We will see below that $f$ from \eqref{Alimit} satisfies the following uniformly in $\sigma_1$ and $\sigma_2$.
\begin{equation} \label{fbound}
| f(\zeta_1(\sigma_1), \zeta_2(\sigma_2))| \leq C_{q,L} T^{1/3}.
\end{equation}
When we change variables $\zeta_1 \mapsto \sigma_1$ and $\zeta_2 \mapsto \sigma_2$
we have $| d\zeta_{\ell} / d\sigma_{\ell}| \leq C_{q,L} T^{-1/3}$ for $\ell = 1,2$.
The conjugation factor also satisfies \eqref{rescaleconj}. Therefore,
\begin{align*}
|F_T(r,u;s,v)| & \leq C_{q,L} \, \nu_T \, T^{-2/3} \, e^{\mu(v-u)} \int_{\R^2} d \sigma_2 d \sigma_2\,
|f(\zeta_1(\sigma_1),\zeta_2(\sigma_2))| \, e^{-C_4(\sigma_1^2+\sigma_2^2)} \times \\
& \qquad \times e^{\Psi \big((u/(\D_{k_1,r^{*}}t)^{1/3})\big)} \cdot e^{\Psi \big(-v / (\D_{s^{*},k_2} t)^{1/3}\big)} \\
& \leq C_{q,L} e^{-\mu u + \Psi \big((u/(\D_{k_1,r^{*}}t)^{1/3})\big)} \cdot e^{\mu v + \Psi \big(-v / (\D_{s^{*},k_2} t)^{1/3}\big)}\,.
\end{align*}
Recall from \eqref{psiproperty} that $e^{-\mu x + \Psi(x/\D)}$ is bounded and integrable over $\R$
if $\mu$ satisfies the bound from \eqref{mudelta} and $\D \geq \min_k \, \{(\D_k t)^{1/3} \}$.
This is the case for us and the matrix is good.

\paragraph{\textbf{Proof of estimate \eqref{fbound} for $f(\zeta_1,\zeta_2)$}}
First, $|(1-\zeta_1)/(1-z_1)| \leq 2/(1-q)$. Suppose that $\zeta_1 \in w_0(\sigma, d_1)$ for some $d_1$
and $K = \kappa_1 T$, and $z_{k_1+1} \in w_1(\sigma,d_2)$ for some $d_2$ and $K = \kappa_2 T$.
Then $|\zeta_1 - z_{k_1+1}| \geq T^{-1/3} ((d_1/\kappa_1) + (d_2/\kappa_2))$. In our case, $d_1,d_2,\kappa_2,\kappa_2$
all remain uniformly positive in $T$, and depends on $q$ and $L$. Therefore,
$|\zeta_1 - z_{k_1+1}|^{-1} \leq C_{q,L} T^{1/3}$. Similarly, $|\zeta_2 - z_{k_2}|^{-1} \leq C_{q,L} T^{1/3}$
if $z_{k_2} \in w_1(\sigma,d_2)$.

The parameters $\D_k(n,m,a)$ are re-scaled according to
\begin{align} \label{nmarescale}
\D_k n &= \D_k t, T - c_1 (\D_k x)\cdot (\D_k t\, T)^{\frac{2}{3}},\\ \nonumber
\D_k m &= \D_k t, T + c_1 (\D_k x)\cdot (\D_k t\, T)^{\frac{2}{3}},\\ \nonumber
\D_k a &= c_2 \, \D_k t, T +c_3 (\D_k \xi)\cdot (\D_k t\, T)^{\frac{1}{3}}.
\end{align}

We choose $z_k$ to lie on the contour $w_1(\sigma_k, D_k)$ with the choice $K = \D_k t \,T$.
The number $D_k$ is chosen so that the estimate \eqref{Gest1} from Lemma \ref{lem:Gest} holds,
namely, uniformly in $\sigma_k$,
$$|G(z_k(\sigma_k) \vt \D_k(n,m,a)| \leq C_3 e^{-C_4 \sigma_k^2}.$$
This is for every $k_1 < k \leq k_2$.

We need the $D_k$s to be ordered according to $\veps$. The $D_k$s may be chosen from
an interval with length of order $T^{1/3}$. So we can choose them from the interval $[1,2p]$, say,
which ensures that they can be ordered accordingly and also that their pairwise distance is at least 1.
Consequently, $|z_k - z_{k+1}|^{-1} \leq C_{q,L} T^{1/3}$ for every $k$.

When we change variables $z_k \mapsto \sigma_k$ we have $|d z_k / d\sigma_k| \leq C_{q,L} T^{-1/3}$.
Thus, if $\zeta_1 \in w_0(\sigma, d_1)$ and $\zeta_2 \in w_0(\sigma',d_2)$, then uniformly
in $\zeta_1$ and $\zeta_2$,
\begin{align*}
 |f(\zeta_1,\zeta_2)| &\leq C_{q,L} \,(T^{-1/3})^{k_2-k_1} \int_{\R^{k_2-k_1}} \, d\sigma_{k_1+1} \cdots d\sigma_{k_2}
e^{-C_4 \sum_k \sigma_k^2} \, \cdot\, (T^{1/3})^{k_2-k_1-1 + 2} \\
&\leq C_{q,L} T^{1/3}.
\end{align*}

\paragraph{\textbf{Proof that $L^{\veps}[{\scriptstyle k_1,k_2 \vt (k_1,k_2]}]$ is convergent}}
Now we assume the kernel variables $u$ and $v$ in $F_T$ remain bounded and we are on the $(r,s)$-block.
We will choose contours for all the variables in the following way.
\begin{align*}
\zeta_1 &= \zeta_1(\hat{\sigma}_1) \in w_0 \left ( \frac{c_4}{\sqrt{q}} \hat{\sigma}_1, \frac{c_4 d_1}{\sqrt{q}} \right ), \quad K \coloneqq \D_{k_1,r^{*}}t\,T.\\
\zeta_2 &= \zeta_1(\hat{\sigma}_2) \in w_0 \left ( \frac{c_4}{\sqrt{q}} \hat{\sigma}_2, \frac{c_4 d_2}{\sqrt{q}} \right ), \quad K \coloneqq \D_{s^{*},k_2}t\,T.\\
z_k & = z_{\ell}(\sigma_k) \in w_1 \left ( \frac{c_4}{\sqrt{q}} \sigma_k, \frac{c_4 D_k}{\sqrt{q}} \right ), \quad K \coloneqq \D_k t\, T.
\end{align*}
The constant $c_4$ is from \eqref{c4}. The numbers $d_1$ and $d_2$ are as in the proof of goodness so that the estimate \eqref{Gest0}
holds. Since $u$ and $v$ are bounded, we may absorb the terms $e^{\Psi(u)}$ and $e^{\Psi(-v)}$ into the constant $C_3$ of the estimate.
The number $D_k$ are chosen so that the estimate \eqref{Gest1} holds. They are also to be ordered according to $\veps$.
As before, we may choose them so that they have pairwise distance at least 1 and are ordered accordingly;
the condition of the ordering is \eqref{Dorder}.

Due to this choice of contours, arguing as before, we find the following estimates. We have
$z_k = z_k(\sigma_k)$ and $\zeta_{\ell} = \zeta_{\ell}(\hat{\sigma}_{\ell})$.
\begin{align*}
& \frac{\prod_{k} |G(z_k \vt \D_k(n,m,a))|}
{|G(\zeta_1 \vt i-n_{k_1}, \D_{k_1,r^{*}}(m,a)) \cdot G(\zeta_2 \vt n_{k_2}-j+1, \D_{s^{*},k_2}(m,a))|}
\leq C_{q,L} \, e^{-C_4 (\sum_k \sigma_k^2 + \hat{\sigma}_1^2 + \hat{\sigma}_2^2 )}.\\
& \nu_T \cdot \big | \prod_{k_1 < k < k_2} (z_k-z_{k+1})^{-1} (z_{k_1+1} - \zeta_1)^{-1} (z_{k_2} - \zeta_2)^{-1} \big|
\cdot \prod_{k_1 < k \leq k_2} |\frac{d z_k}{d\sigma_k}| \cdot \prod_{\ell =1,2} |\frac{d\hat{\zeta_{\ell}}}{\hat{\sigma}_{\ell}}| \leq C_{q,L}.
\end{align*}
These estimates allow us to use the dominated convergence theorem to get the limit of the integral in $F_T(r,u;s,v)$.
So we consider the point-wise limit of the integral.

Suppose $\sigma_k$ and $\hat{\sigma}_{\ell}$ lie on compact subsets of $\R$. We have
\begin{align*}
\zeta_1(\hat{\sigma}_1) &= w_c + \frac{c_4}{(\D_{k_1,r^{*}}t\, T)^{1/3}} \, (\mathbold{i} \hat{\sigma}_1 + d_1) + C_{q,L}T^{-2/3}.\\
\zeta_2(\hat{\sigma}_2) &= w_c + \frac{c_4}{(\D_{s^{*},k_2}t\, T)^{1/3}} \, (\mathbold{i} \hat{\sigma}_2 + d_2) + C_{q,L}T^{-2/3}.\\
z_k(\sigma_k) &= w_c + \frac{c_4}{(\D_k t \, T)^{1/3}} \, (-\mathbold{i} \sigma_k + D_k) + C_{q,L}T^{-2/3}.
\end{align*}
Let us write $z_k' = (-\mathbold{i} \sigma_k + D_k)/\D_k t$, $\zeta'_1 = (\mathbold{i} \hat{\sigma}_1 + d_1)/(\D_{k_1,r^{*}}t)$
and $\zeta'_2 = (\mathbold{i} \hat{\sigma}_2 + d_2)/(\D_{s^{*},k_2}t)$.
With the new variables, in the large $T$ limit, the contour $\gamma_{\tau_{\ell}}$ becomes the vertical contour
$\Gamma_{-d_{\ell}}$ intersecting the real axis at $-d_{\ell}$ (recall $\zeta'_{\ell}$ now remains bounded).
The contour $\gamma_{R_k}(1)$ becomes the vertical contour $\Gamma_{D_k}$ oriented downward.
It is downward because $\gamma_{R_k}(1)$ crosses the real axis at the point $1-R_k$ (which is the one near $w_c$)
in the downward direction. If we re-orient the contours upward then we gain a factor of $(-1)^{k_2-k_1}$.

We see from Lemma \ref{lem:Glimit} that
\begin{align*}
G\big(z_k \vt \D_k(n,m,a)\big) &\longrightarrow \G \big(\D_k t \cdot z'_k \vt 1, \D_k(x,\xi)) = \G(z'_k \vt \D_k(t,x,\xi)\big). \\
G\big (\zeta_1 \vt i-n_{k_1}, \D_{k_1,r^{*}}(m,a)\big) & \longrightarrow
\G \big (\D_{k_1,r^{*}}t \cdot \zeta'_1 \vt 1, \D_{k_1,r^{*}} x, \, \D_{k_1,r^{*}}\xi - (\D_{k_1,r^{*}} t)^{-1/3}u \big)\\
&= \G \big (\zeta'_1 \vt \D_{k_1,r^{*}} (t, x,\xi) \big) \, e^{\zeta'_1 u}.\\
G\big (\zeta_2 \vt n_{k_2}-j+1, \D_{s^{*},k_2}(m,a)\big) & \longrightarrow
\G \big (\D_{s^{*},k_2}t \cdot \zeta'_2 \vt 1, \D_{s^{*},k_2} x,\, \D_{s^{*},k_2}\xi + (\D_{s^{*},k_2} t)^{-1/3}v\big)\\
&= \G \big (\zeta'_1 \vt \D_{s^{*},k_2} (t, x,\xi)\big) \, e^{-\zeta'_2 v}.
\end{align*}
These limits are uniformly so in $\zeta_{\ell}$ and $z_k$, as well as in $u$ and $v$, because
these variables are now restricted to compact subsets of their domains. We also have the following.
\begin{align*}
a) \, & \prod_{k_1 < k < k_2} (z_k - z_{k+1}) = (c_4)^{k_2-k_1-1} (T^{-1/3})^{k_2-k_1-1} \prod_{k_1 < k < k_2} (z'_k -z'_{k+1})
\,+ \,C_{q,L} (T^{-1/3})^{k_2-k_1}.\\
b) \, & \prod_{k_1 < k \leq k_2} d z_k \cdot \nu_T = c_0 (c_4)^{k_2-k_1} (T^{-1/3})^{k_2-k_1-1} \prod_{k_1 < k \leq k_2} d z'_k
\,+ \,C_{q,L} (T^{-1/3})^{k_2-k_1}.\\
c) \, & (z_k - \zeta_{\ell})^{-1} d\zeta_{\ell} = (z'_k - \zeta'_{\ell})^{-1} d \zeta'_{\ell} \,+\, C_{q,L}T^{-1/3};\;\; (k,\ell) = (k_1+1,1) \;\text{or}\; (k_2,2).
\end{align*}

Next, we have that $c_0 c_4 = 1-\sqrt{q} = w_c$, which is a factor we obtain from the ratio of the second product above to the first's.
This cancels the factor $1/w_c$ in $F_T(r,u;s,v)$. Also, as $T \to \infty$, the term $(\frac{1-\zeta_1}{1-z_1})^{\ind{k_1=0}} \to 1$
and the conjugation factor $c(r,i;s,j) \to c(r,u;s,v) = e^{\mu(v-u)}$ by \eqref{rescaleconj}.

Putting all this together we see that the limit of the kernel $F_T(r,u;s,v)$ is the kernel
$(-1)^{k_2-k_1} \times$ $F^{\veps}[{\scriptstyle k_1,k_2 \vt (k_1,k_2]}](r,u;s,v)$, the latter from part (3) of Definition \ref{def:F}.
This proves part (1) of Proposition \ref{prop:Alimit}. This same argument will be used with minor changes
to show goodness and convergence of all the other matrices.\qed

\paragraph{\textbf{Proof of Lemma \ref{lem:LcdB}}}
First we will prove the decomposition of $L^{\veps}[{\scriptstyle k_1,k_2 \vt (k_1,k_2]}]\cdot B$ given in the lemma.
We keep to the notation there. Using Lemma \ref{lem:multiply} we find that
$$L^{\veps}[{\scriptstyle k_1, k_2 \vt (k_1, k_2]}] \cdot B = \sum_{k_3=0}^p \big (1 + \Theta(k_3 \vt s)\big) \big [ \hat{L}_{k_3} - \sm \hat{L}_{k_3} \big ].$$
\begin{align*}
& \hat{L}_{k_3} (r,i;s,j) = \ind{k_1 < r^{*}, \, s < k_3^{*} < k_2} \, c(r,i;s,j)\, \times \\
& \qquad \frac{1}{w_c} \oint\limits_{\gamma_{\tau_1}} d\zeta_1 \oint\limits_{ \gamma_{\tau_2}} d\zeta_2 \oint\limits_{\gamma_{\tau_3}} d\zeta_3
\oint\limits_{\vec{\gamma}_{R^{\veps}}} dz_{k_1+1} \cdots dz_{k_2} \, \left ( \frac{1-\zeta_1}{1-z_1} \right)^{\ind{k_1=0}} \times \\
& \frac{\prod_{k_1 < k \leq k_2} G \big(z_k \vt \D_k (n,m,a) \big) \prod_{k_1 < k < k_2}(z_k - z_{k+1})^{-1} \,
(z_{k_1+1}-\zeta_1)^{-1} (z_{k_2}-\zeta_2)^{-1} (\zeta_2-\zeta_3)^{-1}}
{G \big(\zeta_1 \vt i-n_{k_1}, \D_{k_1, r^{*}}(m,a)\big) \,G \big(\zeta_2 \vt n_{k_2}-n_{k_3}, \D_{k_3^{*},k_2}(m,a)\big)
\,G \big(\zeta_3 \vt n_{k_3}-j+1, \D_{s, k_3^{*}}(m,a)\big )}.
\end{align*}
The matrix $\sm \hat{L}_{k_3}$ looks the same as $\hat{L}_{k_3}$ with the difference being that $n_{k_3}$ is changed to $n_{k_3-1}$
in the two $G$-functions corresponding to variables $\zeta_2$ and $\zeta_3$.

The matrix $\hat{L}_{k_3}$ looks the same as $L^{\veps}[{\scriptstyle k_1,k_2,k_3 \vt (k_1,k_2]}]$ except that
$k_3^{*}$ appears instead of $k_3$ in $\ind{s < k_3^{*} < k_2}$, $\D_{k_3^{*},k_2}(m,a)$ and $\D_{s, k_3^{*}}(m,a)$.
Now $k_3^{*} = k_3$ if $k_3 < p$. An exception occurs if $k_3 = k_2 = p$. In this case $n_{k_2}-n_{k_3} = 0$,
so there is no pole at $\zeta_2 = 0$ in the integrand. The $\zeta_2$-contour is the innermost one since $\tau_2 < \tau_3$,
and it can be contracted to zero. So we may assume $k_3 < p$, and then replace $k_3^{*}$ with $k_3$ in the above.
This results in $L^{\veps}[{\scriptstyle k_1,k_2,k_3 \vt (k_1,k_2]}]$.

Now consider $\sm \hat{L}_{k_3}$. It will also equal $(\sm L)^{\veps}[{\scriptstyle k_1,k_2,k_3 \vt (k_1,k_2]}]$ unless
$k_3 = k_2 = p$. In the latter case, since $k_3^{*} = p-1$, the matrix is $L^{\veps}[{\scriptstyle k_1,p,p-1 \vt (k_1,k_2]}]$.
Accounting for this case we get the representation of $L^{\veps}[{\scriptstyle k_1,k_2 \vt (k_1,k_2]}]\cdot B$ given by the lemma.

Next we prove that $L^{\veps}[{\scriptstyle k_1,k_2,k_3 \vt (k_1,k_2]}]$, which we simply write $L$, is good.
Fix $k_1,k_2,k_3$ and an $(r,s)$-block such that $k_1 < r^{*}$ and $s < k_3 < k_2$. The argument is the same
as the one for goodness of $L^{\veps}[{\scriptstyle k_1,k_2 \vt (k_1,k_2]}]$ since these matrices have the
same structure. The variable $\zeta_3$ now has the same role as the variable $\zeta_2$ did for $L^{\veps}[{\scriptstyle k_1,k_2 \vt (k_1,k_2]}]$,
i.e., it carries the $j$-index. The difference now is that $\zeta_2$ appears in $(\zeta_2-\zeta_3)^{-1}/G(\zeta_2 \vt \D_{k_3,k_2}(n,m,a))$.

We choose $\zeta_2$ to lie on the contour $\gamma_{\tau_2} = w_0(\sigma_2, d_2)$ with $K \coloneqq \D_{k_3,k_2} t\, T$.
The number $d_2$ is to be chosen so that we have the estimate \eqref{Gest0} from Lemma \ref{lem:Gest}, i.e.,
$$\big |G\big (\zeta_2(\sigma_2) \vt \D_{k_3,k_2}(n,m,a)\big) \big|^{-1} \leq C_3 e^{-C_4\sigma_2^2}.$$
As before, $\zeta_3$ is chosen to lie on $\gamma_{\tau_3} = w_0(\sigma_3, d(-v))$ so that we have the estimate
$$\big |G\big (\zeta_3(\sigma_3) \vt \D_{s, k_3} n - \nu_T v, \D_{s,k_3}(m,a)\big) \big|^{-1} \leq
C_3 e^{-C_4 \sigma_3^2 + \Psi \big (-v/(\D_{s,k_3} t)^{1/3} \big)}.$$
We have $|\zeta_2 - \zeta_3|^{-1} \leq C_{q,L} T^{1/3}$ uniformly over the contours, and also $| d\zeta_2/ d\sigma_2| \leq C_{q,L} T^{-1/3}$.

Due to the term $(\zeta_2 - \zeta_3)^{-1}$ we have to ensure that the contours are chosen so that they remain ordered, i.e.,
$\tau_2 < \tau_3$. This means we want $d(-v) < (d_2-1) \cdot (\D_{s,k_3} t / \D_{k_3,k_2} t)^{1/3} \leq C_{q,L} d_2$, say.
Since the column block $s < p$, we have $v \leq 0$, and both $d_2$ and $d(-v)$ can be chosen from
intervals of order $T^{1/3}$ in length. So we can order the contours.

Using the estimates above and arguing as in the proof of goodness of $L^{\veps}[{\scriptstyle k_1,k_2 \vt (k_1,k_2]}]$
we find that $L$ is good as well. Specifically, if $F_T$ is the re-scaled kernel of $L$ according to \eqref{Fredholmembed}, then
$$|F_T(r,u;s,v)| \leq C_{q,L} e^{-\mu u + \Psi \big(u / (\D_{k_1,r^{*}} t)^{1/3} \big)}\cdot e^{\mu v + \Psi \big(-v/ (\D_{s,k_3}t)^{1/3} \big)}.$$
This bound certifies goodness.

Now we argue that $L$ is convergent to $(-1)^{k_2-k_1} F^{\veps}[{\scriptstyle k_1,k_2,k_3 \vt (k_1,k_2]}]$.
This is the same as the earlier proof of convergence of $L^{\veps}[{\scriptstyle k_1,k_2 \vt (k_1,k_2]}]$.
In the KPZ scaling limit the function $G(\zeta_2 \vt \D_{k_3,k_2}(n,m,a))$ converges to $\G(\zeta'_2 \vt \D_k(t,x,\xi))$.
Then the KPZ re-scaled kernel is seen to converge as before.

Finally, we prove that the matrices $\sm L^{\veps}[{\scriptstyle k_1, k_2, k_3 \vt (k_1,k_2]}]$ are small.
Let us fix $k_1,k_2,k_3$, and consider a block $(r,s)$ such that $k_1 < r^{*}$ and $s < k_3 < k_2$.
We have that
\begin{align*}
& \sm L^{\veps}[{\scriptstyle k_1, k_2, k_3 \vt (k_1,k_2]}]\, (r,i;s,j) = \frac{c(r,i;s,j)}{w_c} \intz{1} \intz{2} \intz{3}\, f(\zeta_1,\zeta_2) \,\times \\
& \frac{(\zeta_2-\zeta_3)^{-1}}{G \big(\zeta_1 \vt i-n_{k_1}, \D_{k_1,r^{*}}(m,a)\big) G \big(\zeta_2 \vt n_{k_2} - n_{k_3-1}, \D_{k_3,k_2}(m,a)\big)
G \big(\zeta_3 \vt n_{k_3-1}-j+1, \D_{s,k_3}(m,a)\big)}.
\end{align*}
The function $f(\zeta_1,\zeta_2)$ is from \eqref{Alimit} and satisfies the bound \eqref{fbound}.
The contours are ordered such that $\tau_2 < \tau_3$.

For convenience, introduce
$$\D_1 = (\D_{k_1,r^{*}}t)^{1/3}, \; \D_2 = (\D_{k_3,k_2}t)^{1/3}, \; \D_3 = (\D_{s,k_3}t)^{1/3}, \;
\lambda = \frac{\D_{k_3} n}{\nu_T} = \frac{\D_{k_3}t}{c_0}\,T^{2/3} + C_{q,L}T^{1/3}.$$
We find, ignoring rounding, that
\begin{align*}
(i-n_{k_1}, \D_{k_1,r^{*}}m, \D_{k_1,r^{*}} a) &= \D_{k_1,r^{*}}(n,m,a) + c_0 \big(u/\D_1,0,0\big) \cdot (\D_{k_1,r^{*}}t\, T)^{1/3},\\
(n_{k_2} - n_{k_3-1}, \D_{k_3,k_2}m, \D_{k_3,k_2}a) &= \D_{k_3,k_2}(n,m,a) + c_0 \big (\lambda/\D_2,0,0\big) \cdot (\D_{k_3,k_2} t\,T)^{1/3},\\
(n_{k_3-1}-j, \D_{s,k_3}m, \D_{s,k_3}a) &= \D_{s,k_3}(n,m,a) + c_0 \big(-(v+\lambda)/\D_3,0,0 \big) \cdot (\D_{s,k_3}t\, T)^{1/3}.
\end{align*}
Note $n_{k_3-1}-j \geq 0$ because $j \in (n_{s-1},n_s]$ and $s < k_3$.

Now we choose contours for the variables. We choose $\gamma_{\tau_1}$ to be $w_0(\sigma_1, d(u))$ with $K \coloneqq (\D_1)^3 T$
such that we have the estimate \eqref{Gest0} , namely,
$$|G(\zeta_1(\sigma_1) \vt i-n_{k_1}, \D_{k_1,r^{*}}(m,a))|^{-1} \leq C_3e^{-C_4 \sigma_1^2 + \Psi(u/\D_1)}.$$
Next we choose $\gamma_{\tau_2}$ to be $w_0(\sigma_2, d(\lambda))$ with $K \coloneqq (\D_2)^3 T$ such that we have
$$|G(\zeta_2(\sigma_2) \vt n_{k_2} - n_{k_3-1}, \D_{k_3,k_2}(m,a))|^{-1} \leq C_3e^{-C_4 \sigma_2^2 + \Psi(\lambda/\D_2)}.$$
Finally, $\gamma_{\tau_3}$ is chosen to be $w_0(\sigma_3, d(-v-\lambda))$ with $K \coloneqq (\D_3)^3 T$ such that
$$|G(\zeta_3(\sigma_3) \vt n_{k_3-1}-j+1, \D_{s,k_3}(m,a)|^{-1} \leq C_3 e^{-C_4 \sigma_3^2 + \Psi \big(-(v+\lambda)/\D_3 \big)}.$$

We have to maintain the ordering $\tau_2 < \tau_3$ due to the term $(\zeta_2-\zeta_3)^{-1}$ in the integrand.
So we should have $d(-v-\lambda)/\D_3 < (d(\lambda)-1)/\D_2$, say. We know that $d(\lambda)/\D_2$
may belong to the interval $[C_1/\D_2, C_2 T^{1/3}]$ if $T$ is sufficiently large in terms of $q$ and $L$.
If $v+\lambda \leq 0$ then $d(-(v+\lambda))/\D_3$ may belong to $[C_1/\D_3, C_2 T^{1/3}]$, and we may
order the contours as we wish.

On the other hand, if $v + \lambda > 0$ then $d(-v-\lambda)/\D_3$ may belong
to the interval
$$[C_1/\D_3 + \delta (v+\lambda)^{1/2}/\D_3^{3/2}, C_2 T^{1/3}].$$
Since $d(\lambda)/\D_2$ belongs to $[C_1/\D_2, C_2 T^{1/3}]$, we can ensure that
$d(-v-\lambda)/\D_3 < (d(\lambda)-1)/\D_2$ for all sufficiently large $T$ so long as
$\delta (v+\lambda)^{1/2} < C_2 \Delta_3^{3/2} T^{1/3} - C_1 \D_3^{1/2}$. Now observe that $v \leq 0$
because index $j$ belongs to column block $s$ with $s < p$ due to $s < k_3 < k_2$.
Therefore, $(v+\lambda)^{1/2} \leq \lambda^{1/2} = (\D_{k_3} t/c_0)^{1/2} \, T^{1/3} + C_{q,L}T^{1/6}$.
So we are fine if $\delta < C_2 c_0^{1/2} \D_3^{3/2} (\D_{k_3}t)^{-1/2}$.
We note that $\D_3^{3/2} \geq \min_{k} \, \{(\D_k t)^{1/2}\}$ and $\D_{k_3}t \leq t_p$.
So $\delta$ satisfies the required bound as it is chosen according to \eqref{mudelta}.

Let $F_T(r,u;s,v)$ be our matrix re-scaled according to \eqref{Fredholmembed}.
Recall $|f(\zeta_1,\zeta_2)| \leq C_{q,L}T^{1/3}$ according to \eqref{fbound}. Then,
using the above bounds for the $G$-functions and arguing as in the proof of goodness
of $L^{\veps}[{\scriptstyle k_1,k_2 \vt (k_1,k_2]}]$, we find that
\begin{align*}
|F_T(r,u;s,v)| &\leq C_{q,L} \, e^{\mu(v-u)} e^{\Psi(u/\D_1) + \Psi(\lambda/\D_2) + \Psi (-(v+\lambda)/\D_3)} \\
&= \underbrace{C_{q,L} \, e^{-\mu \lambda + \Psi(\lambda/\D_2) }}_{\eta_T} \cdot
e^{-\mu u + \Psi(u/\D_1)} \cdot e^{\mu (v+\lambda) + \Psi (-(v+\lambda)/\D_3) }.
\end{align*}
Note that every $\D_{\ell} \geq \min_k \, \{(\D_k t)^{1/3}\}$ and $\mu$ satisfies \eqref{mudelta}.
Therefore, from \eqref{psiproperty}, we have that $\eta_T \to 0$ as $T \to \infty$ due to $\lambda \to \infty$.
We also see that the functions $e^{-\mu u + \Psi(u/\D_1)}$ and $e^{\mu (v+\lambda) + \Psi (-(v+\lambda)/\D_3) }$
are bounded and integrable over the reals. This certifies smallness of $\sm L^{\veps}[{\scriptstyle k_1,k_2,k_3 \vt (k_1,k_2]}]$.
\qed

\subsubsection{Proof of claims regarding $L^{\veps}[{\scriptstyle k_1 \vt (k_1, k_2]}]$}
We will first prove $L^{\veps}[{\scriptstyle k_1 \vt (k_1, k_2]}]$ is good and convergent as stated by Proposition \ref{prop:Alimit}.
Then we will prove Lemma \ref{lem:JcdB}.

\paragraph{\textbf{Proof that it is good and convergent}}
Fix $\veps$ and $k_1 < k_2$. Note $L^{\veps}[{\scriptstyle k_1 \vt (k_1, k_2]}]$ has non-zero blocks
only of column block $s = k_2 < p$. Consider the block $(r,s)$ such that $k_1 < r^{*}$ and $s= k_2 < p$.
On this block the matrix has form
\begin{align} \label{Jlimit}
&L^{\veps}[{\scriptstyle k_1 \vt (k_1, k_2]}](r,i;s,j) = c(r,i;s,j) \frac{1}{w_c} \intz{1} \oint\limits_{\gamma{R_{k_2}}(1)} dz_{k_2}\,
f(\zeta_1, z_{k_2})\, \times \\ \nonumber
& \qquad \times \, \frac{G\big (z_{k_2} \vt j-n_{k_2-1} -1, \D_{k_2}(m,a)\big)}{G \big(\zeta_1 \vt i-n_{k_1}, \D_{k_1, r^{*}}(m,a) \big)}.\\ \nonumber
&f(\zeta_1, z_{k_2}) = \oint\limits_{\gamma_{R_{k_1+1}}(1)} dz_{k_1+1}\, \cdots \oint\limits_{\gamma_{R_{k_2-1}}(1)} d z_{k_2-1}\,
\frac{\prod_{k_1 < k < k_2} G(z_k \vt \D_k(n,m,a)) \, \left ( \frac{1-\zeta_1}{1-z_1}\right)^{\ind{k_1=0}}}
{\prod_{k_1 < k < k_2} (z_k - z_{k+1})\, (z_{k_1+1} - \zeta_1)}.
\end{align}
The contours around 1 are ordered according to $\veps$.

Under KPZ scaling the indices $i$ and $j$ are re-scaled as $i = n_{r^{*}} + \nu_T u$ and $j = n_{k_2} + \nu_T v$, where
we ignore rounding. Note that $v \leq 0$ since $s = k_2 < p$. We have that
\begin{align*}
(i-n_{k_1}, \D_{k_1, r^{*}}(m,a)) &= \D_{k_1, r^{*}} (n,m,a) + c_0 \big (u/(\D_{k_1,r^{*}} t)^{1/3},0,0\big ) \cdot(\D_{k_1,r^{*}} t\, T)^{1/3},\\
(j-n_{k_2-1}, \D_{k_2}(m,a)) & = \D_{k_2}(n,m,a) + c_0 \big ( v / (\D_{k_2} t)^{1/3},0,0 \big) \cdot (\D_{k_2} t\, T)^{1/3}.
\end{align*}
The triple $(i-n_{k_1}, \D_{k_1, r^{*}}(m,a))$ has the form \eqref{inmarescale} and $(j-n_{k_2-1}, \D_{k_2}(m,a))$
has the form \eqref{jnmarescale}.

Now we choose contours for the variables. We choose $\gamma_{\tau_1}$ to be $w_0(\hat{\sigma}_1, d(u))$ with $K \coloneqq \D_{k_1,r^{*}} t \, T$.
Then with an appropriate choice of $d(u)$ from Lemma \ref{lem:Gest}, we have the estimate \eqref{Gest0}:
$$ |G(\zeta_1(\hat{\sigma}_1) \vt (i-n_{k_1}, \D_{k_1, r^{*}}(m,a)) )|^{-1} \leq C_3 e^{-C_4 \hat{\sigma}_1^2 + \Psi \big(u/(\D_{k_1,r^{*}} t )^{1/3}\big)}.$$
Next we choose $\gamma_{R_{k_2}}(1)$, the contour of $z_{k_2}$, to be $w_1(\sigma_{k_2}, D(v))$ with $K \coloneqq \D_{k_2}t\, T$
so that we get the estimate \eqref{Gest1}:
$$ |G(z_{k_2}(\sigma_{k_2}) \vt (j-n_{k_2-1}-1, \D_{k_2}(m,a)) )| \leq C_3 e^{-C_4 \sigma_{k_2}^2 + \Psi \big(v/(\D_{k_2} t)^{1/3}\big)}.$$

For $k_1 < k < k_2$, we choose the contour $\gamma_{R_k}(1)$ to be $w_1(\sigma_k, D_k)$ with $K \coloneqq \D_kt\, T$
such that we have the estimate \eqref{Gest1} from Lemma \ref{lem:Gest}:
$$|G(z_k(\sigma_k) \vt \D_k(n,m,a))| \leq C_3 e^{-C_3 \sigma_k^2}.$$
The parameter $D_k$ may be chosen from the range $[C_1, C_2 (\D_k t\, T)^{1/3}]$.
We have seen that we can choose these $D_k$s such that they are ordered according to $\eps$.
The parameter $D_{k_2-1}$ has to be ordered with respect to $D(v)$. We can first chose these two and then
choose the remaining $D_k$s accordingly.

To see that $D_{k_2-1}$ and $D(v)$ can be ordered, set $\D_1 = (\D_{k_2-1})^{1/3}$ and $\D_2 = (\D_{k_2}t)^{1/3}$.
Since $v \leq 0$, $D(v)$ may be chosen such that $D(v)/\D_2$ belongs to the range
$[C_1/\D_2 + \delta (v)_{-}^{1/2}/\D_2^{3/2}, C_2 T^{1/3}]$.
The number $D_{k_2-1}/\D_1$ may belong to $[C_1/\D_1, C_2T^{1/3}]$. If $\eps_{k_2-1} = 2$
then we require $D_{k_2-1}/\D_1 < (D(v)-1)/\D_2$, say, and this is possible
within the aforementioned ranges. Suppose $\eps_{k_2-1} = 1$. Then we are fine so long
as $\delta (v)_{-}^{1/2} < C_2 \D_2^{3/2}\, T^{1/3} - C_1 \D_2^{1/2}$. Now since $j \in (n_{k_2-1}, n_{k_2}]$,
we have $(v)_{-} \leq \D_{k_2}n /\nu_T \leq (\D_{k_2} t / c_0)\, T^{2/3} + C_{q,L} T^{1/3}$.
So $(v)_{-}^{1/2} \leq (\D_{k_2}t / c_0)^{1/2} \, T^{1/3} + C_{q,L} T^{1/6}$.
Therefore, it suffices to have $\delta < C_2 \D_2^{3/2} (\D_{k_2}t / c_0)^{-1/2}$, which is the case
since $\delta$ satisfies \eqref{mudelta}.

Let $F_T(r,u;s,v)$ be the re-scaling of our matrix by \eqref{Fredholmembed}. Having chosen the
contours, the estimates above imply the following, if we argue as in the proof of goodness of $L^{\veps}[{\scriptstyle k_1,k_2 \vt (k_1,k_2]}]$.
\begin{align*}
& |F_T(r,u;s,v)| \leq C_{q,L} \, \nu_T \, (T^{-\frac{1}{3}})^{k_2-k_1+1} \int_{\R^{k_2-k_1+1}} d \hat{\sigma}_1 \cdots \sigma_{k_2}\,
e^{-C_4 (\hat{\sigma}_1^2 + \cdots + \sigma_{k_2}^2)} (T^{\frac{1}{3}})^{k_2-k_1} \times \\
& \qquad \qquad \ind{v \leq 0} e^{-\mu u + \Psi \big(u / (\D_{k_1,r^{*}} t)^{1/3} \big)} \, e^{\mu v + \Psi \big (v / (\D_{k_2}t)^{1/3} \big)} \\
& \qquad \qquad \leq C_{q,L}\,  e^{-\mu u + \Psi \big(u / (\D_{k_1,r^{*}} t)^{1/3} \big)} \cdot \ind{v \leq 0} e^{\mu v + \Psi \big (v / (\D_{k_2}t)^{1/3} \big)}.
\end{align*}
Both $\D_{k_1,r^{*}} t$ and $\D_{k_2}t $ are at least $\min_k \, \{\D_k t \}$ and $\mu$ satisfies $\eqref{mudelta}$.
So the functions of $u$ and $v$ above are bounded and integreble by \eqref{psiproperty}, and the matrix is good.

For the proof of convergence of $L^{\veps}[{\scriptstyle k_1 \vt (k_1,k_2]}]$ to $(-1)^{k_2-k_1} F^{\veps}[{\scriptstyle k_1 \vt (k_1,k_2]}]$
we can repeat the argument for convergence of $L^{\veps}[{\scriptstyle k_1,k_2 \vt (k_1,k_2]}]$. \qed

\paragraph{\textbf{Proof of Lemma \ref{lem:JcdB}}}
Since $L^{\veps}[{\scriptstyle k_1 \vt (k_1,k_2]}]$ has non-zero blocks only on column block $k_2$,
$$L^{\veps}[{\scriptstyle k_1 \vt (k_1,k_2]}]  \cdot B (r,i;s,j) = (1+ \Theta(k_2 \vt s))\, \sum_{\ell \in (n_{k_2-1}, n_{k_2}]}
L^{\veps}[{\scriptstyle k_1 \vt (k_1,k_2]}](r,i;k_2,\ell) \, B(k_2,\ell, s,j).$$
We can compute this using Lemma \ref{lem:multiply} as follows.
\begin{align*}
&L^{\veps}[{\scriptstyle k_1 \vt (k_1,k_2]}] \cdot B (r,i;s,j) =(1 + \Theta(k_2 \vt s)) \, c(r,i;s,j) \, \times \\
&\ind{k_1 < r^{*}, s < k_2 < p} \frac{1}{w_c^2} \intz{1} \oint\limits_{\gamma_{R_{k_2}}(1)} \intz{2}\,
\frac{f(\zeta_1, z_{k_2})}{G\big(\zeta_1 \vt i-n_{k_1}, \D_{k_1,r^{*}}(m,a)\big)} \, \times\\
& \sum_{\ell \in (n_{k_2-1}, n_{k_2}]} \frac{1}{G\big(z_{k_2} \vt n_{k_2-1}-\ell+1, - \D_{k_2}(m,a) \big) \, G \big(\zeta_2 \vt \ell-j+1, \D_{s,k_2}(m,a) \big )}\\
& = (1 + \Theta(k_2 \vt s)) \, c(r,i;s,j)\, \ind{k_1 < r^{*}, s < k_2 < p}\frac{1}{w_c} \intz{1} \intz{2} \oint\limits_{\gamma_{R_{k_2}}(1)} \\
&\Big [ \, \frac{f(\zeta_1, z_{k_2})(z_{k_2}-\zeta_1)^{-1} G\big(z_{k_2} \vt \D_{k_2}(n,m,a)\big)}
{G\big(\zeta_1 \vt i-n_{k_1}, \D_{k_1,r^{*}}(m,a)\big)\, G \big( \zeta_2 \vt n_{k_2}-j+1, \D_{s,k_2} (m,a) \big)} \, - \\
&\frac{f(\zeta_1, z_{k_2})(z_{k_2}-\zeta_1)^{-1} G\big(z_{k_2} \vt 0, \D_{k_2}(m,a)\big)}
{G \big(\zeta_1 \vt i-n_{k_1}, \D_{k_1,r^{*}}(m,a)\big) \, G \big( \zeta_2 \vt n_{k_2-1}-j+1, \D_{s,k_2} (m,a) \big)} \, \Big ] \\
& = (1 + \Theta(k_2 \vt s)) \Big [ L^{\veps}[{\scriptstyle k_1,k_2, \vt (k_1,k_2]}] - \sm L^{\veps}[{\scriptstyle k_1,k_2, \vt (k_1,k_2]}] \Big].
\end{align*}
The function $f$ is from \eqref{Jlimit}. We observed above that
$f(\zeta_1, z_d) (z_{k_2}-\zeta_2)^{-1} G(z_{k_2} \vt \D_{k_2}(n,m,a))$ divided by
$G(\zeta_1 \vt \cdots) \cdot G(\zeta_2 \vt \cdots)$ makes the integrand of
$L^{\veps}[{\scriptstyle k_1,k_2, \vt (k_1,k_2]}]$, as is required.

To complete the proof we show that the matrix $\sm L^{\veps}[{\scriptstyle k_1,k_2, \vt (k_1,k_2]}]$ is small.
The argument is analogous to the prior proof of smallness of $\sm L^{\veps}[{\scriptstyle k_1,k_2, k_3 \vt (k_1,k_2]}]$.
The role of variables $\zeta_1, \zeta_2, \zeta_3$ from there is now given to $\zeta_1, z_{k_2}, \zeta_2$, respectively.
The parameter $\lambda = \D_{k_2} n / \nu_T = (\D_{k_2}t / c_0) T^{2/3} +C_{q,L}T^{1/3}$. Since the
$\zeta_2$-contour lies around 0 and the $z_{k_2}$-contour around 1, there is no ordering between them.
We need the $z_{k_2}$-contour to be ordered with respect to the $z_{k_2-1}$-contour according to $\eps_{k_2-1}$,
and for this we may repeat the prior argument for the goodness of $L^{\veps}[{\scriptstyle k_1, \vt (k_1,k_2]}]$.

After choosing contours as before we get the following estimates for the $G$-functions.
\begin{align*}
& |G \big (\zeta_1(\sigma_1) \vt \D_{k_1,r^{*}}n + \nu_T u, \D_{k_1,r^{*}}(m,a)\big )|^{-1} \leq C_3 e^{-C_4 \sigma_1^2 + \Psi(u/\D_1)} \\
& |G \big(\zeta_2 \vt \D_{s, k_2}n - \nu_T (v+\lambda), \D_{s, k_2}(m,a) \big)|^{-1} \leq C_3 e^{-C_4 \sigma_2^2 + \Psi \big( -(v+\lambda)/\D_2 \big)}\\
& |G \big (z_{k_2}(\sigma_3) \vt \D_{k_2}n -\nu_T \lambda, \D_{k_2}(m,a) \big)| \leq C_3 e^{-C_4 \sigma_3^2 + \Psi(-\lambda/\D_3)}.
\end{align*}
Here, $\D_1 = (\D_{k_1,r^{*}} t)^{1/3},\D_2 = (\D_{s,k_2} t)^{1/3}$ and $\D_3 = (\D_{k_2} t)^{1/3}$.
	
Using these estimates, and arguing as before, we find the following estimate for the re-scaled kernel $F_T$
of $\sm L^{\veps}[{\scriptstyle k_1,k_2, \vt (k_1,k_2]}]$.
$$ |F_T(r,u;s,v)| \leq \underbrace{C_{q,L} e^{-\mu \lambda + \Psi(-\lambda/\D_3)}}_{\eta_T} \cdot
e^{-\mu u + \Psi(u/\D_1)} \cdot e^{\mu(v+\lambda) + \Psi(-(v+\lambda)/\D_2)}.$$
We observe that $\eta_T = C_{q,L} e^{-\mu \lambda - \mu_1 (\lambda/\D_3)^{3/2}} \to 0$, and the two functions
of $u$ and $v$ are bounded and integrable over $\R$ due to \eqref{psiproperty}.
So the matrix is small.\qed

\subsubsection{Proof of claims regarding $L[{\scriptstyle k,k \vt \emptyset}]$}
First we will prove that $L[{\scriptstyle k,k \vt \emptyset}]$ is good and convergent to $F[{\scriptstyle k,k \vt \emptyset}]$.
Then we will prove Lemma \ref{lem:LkB} by first showing that $L[{\scriptstyle k_1, k_1, k_2 \vt \emptyset}]$ is good
and convergent, and then that $\sm L[{\scriptstyle k_1, k_1, k_2 \vt \emptyset}]$ is small.

\paragraph{\textbf{Proof that $L[{\scriptstyle k,k \vt \emptyset}]$ is good and convergent}}
The matrix $L[{\scriptstyle k,k \vt \emptyset}]$ has non-zero blocks $(r,s)$ only if $s < k < r^{*}$.
Let us fix such $k,r$ and $s$, so then $L[{\scriptstyle k,k \vt \emptyset}](r,i;s,j)$ equals
\begin{equation*}
c(r,i;s,j) \, \frac{1}{w_c} \intz{1} \intz{2}\,
\frac{(\zeta_1 - \zeta_2)^{-1}}{G \big ( \zeta_1 \vt i-n_k, \D_{k,r^{*}}(m,a) \big)\, G \big ( \zeta_2 \vt n_k-j+1, \D_{s,k}(m,a) \big)}.
\end{equation*}

Ignoring rounding, the indices are re-scaled according to $i = n_{r^{*}} + \nu_T u$ and $j = n_s + \nu_T v$.
Note that $v \leq 0$ since $s < p$. In this case the KPZ re-scaling of $(i-n_k, \D_{k,r^{*}}(m,a))$ looks like
\eqref{inmarescale}, and that of $(n_k-j, \D_{s,k}(m,a))$ like \eqref{jnmarescale}. Set $\D_1 = (\D_{k,r^{*}} t)^{1/3}$
and $\D_2 = (\D_{s,k}t)^{1/3}$.

For establishing goodness, contours are chosen so that the $\zeta_1$-contour is $w_0(\sigma_1, d(u))$
with $K \coloneqq \D_{k,r^{*}} t\, T$. The $\zeta_2$-contour is $w_0(\sigma_2, d(-v))$ with $K \coloneqq \D_{s,k}t\, T$.
With appropriate choices for $d(u)$ and $d(-v)$, Lemma \ref{lem:Gest} provides the estimates
\begin{align*}
& |G(\zeta_1(\sigma_1) \vt \D_{k,r^{*}} n + \nu_T u, \D_{k,r^{*}} (m,a))|^{-1} \leq C_3 \, e^{-C_4 \sigma_1^2 + \Psi(u/\D_1)},\\
& |G(\zeta_2(\sigma_2) \vt \D_{s,k}n - \nu_T v, \D_{s,k}(m,a))|^{-1} \leq C_3\,  e^{-C_4 \sigma_2^2 + \Psi(-v/\D_2)}.
\end{align*}
We need to have $\tau_2 < \tau_1$, which translates to $d(u)/\D_1 < (d(-v)-1)/\D_2$, say. Since $v \leq 0$,
the number $d(-v)/\D_2$ may be chosen from $[C_1/\D_2, C_2 T^{1/3}]$ once $T$ is large enough it terms
of $q$ and $L$. When $u \geq 0$, $d(u)/\D_1$ can be chosen from $[C_1/\D_1, C_2 T^{1/3}]$, and we can
order the contours accordingly. If $u \leq 0$ then $d(u)/\D_1$ may belong to
$[C_1/\D_1 + \delta (u)_{-}^{1/2}/\D_1^{3/2}, C_2 T^{1/3}]$. We can order the contours
so long as $\delta (u)_{-}^{1/2} < C_2 \D_1^{3/2}\, T^{1/3} - C_1 \D_1^{1/2}$. We have that
$(u)_{-} \leq (\D_{r^{*}}) t/c_0) T^{2/3} + C_{q,L}T^{1/3}$. Therefore, as before,
we are fine since $\delta$ satisfies $\eqref{mudelta}$.

Let $F_T$ be the re-scaled kernel of $L[{\scriptstyle k,k \vt \emptyset}]$ by \eqref{Fredholmembed}.
The estimates above for the $G$-functions and the same argument used to show goodness of
$L^{\veps}[{\scriptstyle k_1,k_2 \vt (k_1,k_2]}]$ implies the following bound.
\begin{equation} \label{FLk}
|F_T(r,u;s,v)| \leq C_{q,L}\, e^{-\mu u + \Psi(u/\D_1)} \cdot e^{\mu v + \Psi(-v/\D_2)}.\end{equation}
This certifies goodness of $L[{\scriptstyle k,k \vt \emptyset}]$ by \eqref{psiproperty}.

The proof of convergence to $F[{\scriptstyle k,k \vt \emptyset}]$ is same as that of
$L^{\veps}[{\scriptstyle k_1,k_2 \vt (k_1,k_2]}]$ converging to the kernel
$(-1)^{k_2-k_1} F^{\veps}[{\scriptstyle k_1,k_2 \vt (k_1,k_2]}]$. So we omit the details. \qed

\paragraph{\textbf{Proof of Lemma \ref{lem:LkB}}}
We multiply $L[{\scriptstyle k,k \vt \emptyset}]$ by $B$ using Lemma \ref{lem:multiply}.
\begin{align*}
& L[{\scriptstyle k,k \vt \emptyset}]\cdot B \, (r,i;s,j) = \sum_{k_2} (1+ \Theta(k_2 \vt s)) \, \ind{k < r^{*},\, s < k_2 < k} \, c(r,i;s,j) \, \times \\
&\frac{1}{w_c^2} \intz{1} \intz{2} \intz{3} \frac{(\zeta_1-\zeta_2)^{-1}}{G(\zeta_1 \vt i-n_k, \D_k(m,a))} \, \times \\
& \Big [ \sum_{\ell \in (n_{k_2-1}, n_{k_2}]} \frac{1}{G\big (\zeta_2 \vt n_k-\ell+1, \D_k (m,a) \big) G\big(\zeta_3 \vt \ell-j+1, \D_{s,k_2}(m,a) \big)} \Big]\\
& = \sum_{k_2} (1+ \Theta(k_2 \vt s)) \cdot \big [L[{\scriptstyle k,k,k_2 \vt \emptyset}](r,i;s,j) - (\sm L)[{\scriptstyle k,k,k_2 \vt \emptyset}](r,i;s,j) \big].
\end{align*}

Now consider $L[{\scriptstyle k_1,k_1,k_2 \vt \emptyset}]$ to see that it is good, and converges to $F[{\scriptstyle k_1,k_1,k_2 \vt \emptyset}]$.
Recall
\begin{align*}
&L[{\scriptstyle k_1,k_1,k_2 \vt \emptyset}](r,i;s,j) = \ind{k_1 < r^{*},\, s < k_2 < k_1}\ c(r,i;s,j) \frac{1}{w_c} \intz{1} \intz{2} \intz{3} \\
& \frac{(\zeta_1 - \zeta_2)^{-1} (\zeta_2-\zeta_3)^{-1}}
{G \big(\zeta_1 \vt i-n_{k_1}, \D_{k_1,r^{*}}(m,a) \big)\, G \big(\zeta_2 \vt \D_{k_2,k_1}(n,m,a) \big)\, G \big(\zeta_1 \vt n_{k_2}-j+1, \D_{s,k_2}(m,a) \big)}.
\end{align*}
This matrix has the same structure as $L[{\scriptstyle k,k \vt \emptyset}]$, and the proof of goodness and convergence is analogous.
The new terms in the integrand are $(\zeta_2-\zeta_3)^{-1}$ and $G(\zeta_2 \vt \D_{k_2,k_1}(n,m,a))$. The latter converges
to $\G(\zeta_2 \vt  \D_{k_2,k_1}(t,x,\xi))$ under KPZ re-scaling by Lemma \ref{lem:Glimit}, which leads to the limit kernel
$F[{\scriptstyle k_1,k_1,k_2 \vt \emptyset}]$. In the proof of goodness, one uses estimate \eqref{Gest0} from Lemma \ref{lem:Gest}
to derive the same bound \eqref{FLk} on the re-scaled kernel of $L[{\scriptstyle k_1,k_1,k_2 \vt \emptyset}]$.

During the estimates leading to goodness, one has to ensure that the contours are ordered appropriately.
We require that $\tau_2 < \tau_1, \tau_3$ due to the term $(\zeta_1-\zeta_2)^{-1} (\zeta_2-\zeta_3)^{-1}$.
We choose the $\zeta_2$-contour to be $w_0(\sigma_2, d_2)$ with $K \coloneqq \D_{k_2,k_2} t \,T$.
The parameter $d_2$ may be chosen from an interval with length of order $T^{1/3}$. Then, the same
argument used for ordering contours in showing goodness of $L[{\scriptstyle k,k \vt \emptyset}]$ shows
that contours can be ordered accordingly.

We are left to prove that $\sm L [{\scriptstyle k_1,k_1,k_2 \vt \emptyset}]$ is small. It is similar to proofs of smallness so far.
Let us fix $k_1, k_2$ and consider a non-zero $(r,s)$-block, so then $k_1 < r^{*}$ and $s < k_2 < k_1$. We have
\begin{align*}
& \sm L[{\scriptstyle k_1,k_1,k_2 \vt \emptyset}] (r,i;s,j) = c(r,i;s,j) \, \frac{1}{w_c} \intz{1} \intz{2} \intz{3} \\
& \frac{(\zeta_1-\zeta_2)^{-1} (\zeta_2-\zeta_3)^{-1}} {G \big ( \zeta_1 \vt i-n_{k_1}, \D_{k_1,r^{*}} (m,a)\big)\,
G \big ( \zeta_2 \vt n_{k_1} - n_{k_2-1}, \D_{k_2,k_1} (m,a)\big)\, G \big ( \zeta_3 \vt n_{k_2-1}-j+1, \D_{s,k_2} (m,a)\big)\,}.
\end{align*}
The radii satisfy $\tau_2 < \tau_1, \tau_3 < 1- \sqrt{q}$.

We have $i = n_{r^{*}} + \nu_T u$ and $j = n_{s} + \nu_T v$. Set $\lambda = \D_{k_2}n /\nu_T = (\D_{k_2}t/c_0) T^{2/3}$.
Also set $\D_1 = (\D_{k_1,r^{*}} t)^{1/3}$, $\D_2 = (\D_{k_2,k_1} t)^{1/3}$ and $\D_1 = (\D_{s,k_2} t)^{1/3}$. Then,
\begin{align*}
& (i-n_{k_1}, \D_{k_1,r^{*}} (m,a)) = (\D_{k_1,r^{*}} n + \nu_T u, \D_{k_1, r^{*}}(m,a)), \\
& (n_{k_1} - n_{k_2-1}, \D_{k_2,k_1} (m,a)) = (\D_{k_2,k_1} n + \nu_T \lambda, \D_{k_2,k_1} (m,a)), \\
& (n_{k_2-1}-j+1, \D_{s,k_2} (m,a)) = (\D_{s,k_2} n - \nu_T(v + \lambda), \D_{s,k_2}(m,a)).
\end{align*}

We choose the $\zeta_1$-contour to be $w_0(\sigma_1, d(u))$, the $\zeta_2$-contour as $w_0(\sigma_2, d(\lambda))$
and the $\zeta_3$-contour as $w_0(\sigma_3, d(-v-\lambda))$. The corresponding values of $K$ are
$\D_{k_1,r^{*}} t\, T$, $\D_{k_2,k_1} t\, T$ and $\D_{s,k_2} t\, T$, respectively. By Lemma \ref{lem:Gest}, we have the following estimates.
\begin{align*}
& |G(\zeta_1(\sigma_1) \vt \D_{k_1,r^{*}} n + \nu_T u, \D_{k_1, r^{*}}(m,a))|^{-1} \leq C_3 \, e^{-C_4 \sigma_1^2 + \Psi(u/\D_1)}, \\
& |G(\zeta_2(\sigma_2) \vt \D_{k_2,k_1} n + \nu_T \lambda, \D_{k_2,k_1} (m,a))|^{-1} \leq C_3 \, e^{-C_4 \sigma_2^2 + \Psi(\lambda/\D_2)}, \\
& |G(\zeta_3 \vt \D_{s,k_2} n - \nu_T(v + \lambda), \D_{s,k_2}(m,a))|^{-1} \leq C_3\, e^{-C_4 \sigma_3^2 + \Psi(- (v+\lambda)/\D_3)}.
\end{align*}
To ensure constraints on the radii of contours, we need $(d(\lambda)-1)/\D_2 > \max \, \{ d(u)/\D_1,\, d(-v-\lambda)/\D_3 \}$, say.
We can choose $d(\lambda)/\D_2$ from the interval $[C_1/\D_2, C_2 T^{1/3}]$. We also have $(u)_{-} \leq \D_{r^{*}}n / \nu_T$,
and the square root of the latter is of order $T^{1/3}$. Since $v \leq 0$ (due to $s < p$), $v+\lambda \leq \lambda$,
and $\lambda^{1/2}$ is of order $T^{1/3}$. Then, since $\delta$ satisfies \eqref{mudelta}, arguing as before we see that 
the $d$s can be chosen to satisfy the constraints.

Let $F_T$ be the re-scaled kernel of $\sm L[{\scriptstyle k_1,k_1,k_2 \vt \emptyset}]$ by \eqref{Fredholmembed}.
Using the estimates above and arguing as before we find the following.
\begin{align*}
|F_T(r,u;s,v)| & \leq C_{q,L}\, e^{\mu(v-u)}\, e^{\Psi(u/\D_1) + \Psi(\lambda/\D_2) + \Psi((-v-\lambda)/\D_3))} \\
& = \underbrace{C_{q,L} e^{-\mu \lambda + \Psi(\lambda/\D_2)}}_{\eta_T}
\cdot e^{-\mu u + \Psi(u/\D_1)} \cdot e^{\mu(v+\lambda) + \Psi((-v-\lambda)/\D_3)}.
\end{align*}
We observe that $\eta_T = C_{q,L} e^{(\frac{\mu_2}{\D_2}-\mu) \lambda}$ tends to zero since $\mu$ satisfies \eqref{mudelta}.
The functions of $u$ and $v$ are bounded and integrable over $\R$. So the matrix is small. \qed

\subsubsection{Proof of claims regarding $L[{\scriptstyle p \vt p}]$}
First we will prove that $L[{\scriptstyle p \vt p}]$ is good with limit $- F[{\scriptstyle p \vt p}]$, which will complete
the proof of Proposition \ref{prop:Alimit}. Then we will prove Lemma \ref{lem:LpB}.

\paragraph{\textbf{Proof that it is good and convergent}}
The argument is similar to the goodness and convergence of $L[{\scriptstyle k_1 \vt (k_1, k_2]}]$ as these matrices are alike.
The only non-zero row block of $L[{\scriptstyle p \vt p}]$ is for $r = p$ (see $L_p$ from Lemma \ref{lem:Lcd}).
On the $(p,s)$-block the indices $i,j$ are re-scaled as $i = n_{p-1} + \nu_T u$ for $0 \leq u \leq \D_p n / \nu_T$,
and $j = n_{s*} + \nu_T v$. We ignore rounding. So we find that
\begin{align*}
(n_p - i, \D_p(m,a)) &= \D_p (n,m,a) + c_0 \big (-u / (\D_p t)^{1/3}, 0,0 \big) \cdot (\D_p t \, T)^{1/3} \\
(n_p-j, \D_{s^{*}, p}(m,a)) & = \D_{s^{*},p}(n,m,a) + c_0 \big (-v/ (\D_{s^{*},p}t)^{1/3} , 0,0 \big) \cdot (\D_{s^{*},p}t \, T)^{1/3}.
\end{align*}

We choose $\gamma_{\tau_2}$ to be the contour $w_0(\sigma_1, d(-v))$ with $K \coloneqq \D_{s^{*},p}t T$ and $\gamma_{R_p}(1)$
to be the contour $w_1(\sigma_2, d(-u))$ with $K \coloneqq \D_p t T$. Since the $\zeta_2$-contour is around 0 and the $z_p$-contour
is around 1, we can ensure that $|z_p - \zeta_2| \geq C_{q,L}T^{-1/3}$ along these contours. According to Lemma \ref{lem:Gest}
we then have the following estimates.
\begin{align} \label{LpGest}
&|G(\zeta_2(\sigma_1) \vt n_p-j+1, \D_{s^{*},p} (m,a))|^{-1} \leq C_3 e^{- C_4 \sigma_1^2 + \Psi(-v/(\D_{s^{*},p}t)^{1/3})}, \\ \nonumber
&|G(z_p(\sigma_2) \vt n_p-i, \D_p (m,a))| \leq C_3 e^{-C_4 \sigma_2^2 + \Psi(-u / (\D_p t)^{1/3})}.
\end{align}

The re-scaled kernel of $L[{\scriptstyle p \vt p}]$ according to \eqref{Fredholmembed} then satisfies the following,
arguing as before.
\begin{equation*}
|F_T(r,u;s,v)| \leq \ind{r=p} C_{q,L} \, \ind{u \geq 0} e^{- \mu u + \Psi(-u / (\D_p t)^{1/3})} \cdot e^{\mu v + \Psi(-v/(\D_{s^{*},p}t)^{1/3})}.
\end{equation*}
The functions of $u$ and $v$ above are bounded and integrable by \eqref{psiproperty}. So $L[{\scriptstyle p \vt p}]$ is good.
The argument for convergence of $L[{\scriptstyle p \vt p}]$ to $- F[{\scriptstyle p \vt p}]$ is the same as before.

\paragraph{\textbf{Proof of Lemma \ref{lem:LpB}}}
We multiply $L[{\scriptstyle p \vt p}]$ by $B$ using Lemma \ref{lem:multiply}:
$$ L[{\scriptstyle p \vt p}] \cdot B (r,i;s,j) = \sum_{k=1}^p (1 + \Theta(k \vt s)) \big ( \hat{L}_k - (\sm \hat{L})_k \big) (r,i;s,j),$$
where
\begin{align*}
\hat{L}_k(r,i;s,j) & = \ind{r=p, \, s < k^{*}}\, c(r,i;s,j)\, \frac{1}{w_c} \intz{2} \intz{3} \oint\limits_{\gamma_{R_p}(1)} d z_p \\
&\frac{G \big (z_p \vt n_p-i, \D_p(m,a)\big) (z_p-\zeta_2)^{-1} (\zeta_2 - \zeta_3)^{-1}}
{G \big (\zeta_2 \vt n_p-n_k, \D_{k^{*},p}(m,a)\big) G \big( \zeta_3 \vt n_k-j+1, \D_{s,k^{*}}(m,a) \big)},
\end{align*}
and $(\sm \hat{L})_k$ looks the same as $\hat{L}_k$ except for $n_k$ being changed to $n_{k-1}$
in both of $G(\zeta_2 \vt n_p-n_k, \cdots)$ and $G(\zeta_3 \vt n_k-j+1, \cdots)$ above.
The contours are arranged to satisfy $\tau_2 < \tau_3 < w_c$.

Now if $k < p$ then we see in the above that $\hat{L}_k$ equals $L[{\scriptstyle p, k \vt p}]$ as $k^{*} = k$.
However, when $k=p$, $\hat{L}_p = 0$ because there is no pole at $\zeta_2 = 0$ in its integrand due to
$n_p = n_k$ and the $\zeta_2$-contour being the innermost one. So in this way we get the matrices
$L[{\scriptsize p,k \vt p}]$.
Now consider the matrix $(\sm \hat{L})_k$. If $k < p$ then it equals $(\sm L)[{\scriptstyle p,k \vt p}]$
by definition. When $k = p$ it is actually $(\sm L)[{\scriptstyle p,p-1 \vt p}]$ by definition since $k^{*}$
then equals $p-1$. This implies the expression for $L[{\scriptstyle p \vt p}] \cdot B$ given in the lemma.

The goodness and convergence of $L[{\scriptstyle p,k \vt p}]$ is analogous to that for $L[{\scriptstyle p \vt p}]$ above.
We explain the difference. We use the estimates from \eqref{LpGest} to estimate the $G$-functions associated to
the $\zeta_3$ and $z_p$ contours. They involve the variables $u$ and $v$ from the kernel. There is an additional
function $G(\zeta_2 \vt \D_p(n,m,a))$ in the denominator of the integrand. For it we choose the $\zeta_2$-contour
to be $w_0(\sigma, d)$ with $K = \D_{k,p} t T$, and use the estimate \eqref{Gest0} from Lemma \ref{lem:Gest}.
We have to keep the $\zeta_2$ and $\zeta_3$ contours ordered ($\tau_2 < \tau_3$), for which we require
$d/ (\D_{k,p}t)^{1/3} > (d(-v)+1)/(\D_{s^{*},p}t)^{1/3}$. This is ensured as before since the parameter $d$
may be chosen from an interval whose length is of order $T^{1/3}$.

The proof of smallness of $(\sm L)[{\scriptstyle p,k \vt p}]$ is also similar to the smallness of
$(\sm L)[{\scriptstyle k_1, k_2 \vt (k_1, k_2])}]$ from before. Arguing as there, we will get the
following estimate for the re-scaled kernel $F_T(r,u;s,v)$ of $(\sm L)[{\scriptstyle p,k \vt p}]$.
Set $\lambda = \D_k n / \nu_T$ and $\eta_T = e^{-\mu \lambda + \Psi(\lambda / (\D_{k,p}t)^{1/3})}$.
Recall $1 \leq k < p$, so $\lambda \to \infty$ and $\D_{k,p}t > 0$. If $\mu$ satisfies \eqref{mudelta},
then $\eta_T \to 0$ and
\begin{equation*}
|F_T(r,u;s,v)| \leq \ind{r=p} C_{q,L} \, \eta_T\, \ind{u \geq 0} e^{- \mu u + \Psi(-u / (\D_p t)^{1/3})} \cdot
e^{\mu (v+\lambda) + \Psi \big(-(v+\lambda)/(\D_{s^{*},k}t)^{1/3} \big)},
\end{equation*}
which guarantees smallness. \qed

\subsection{Tying up loose ends} \label{sec:asy4}
Here we will prove Proposition \ref{prop:small} and that the limit from Theorem \ref{thm:1} is a probability distribution.
\smallskip

\paragraph{\textbf{Proof of Proposition \ref{prop:small}}}
It is enough to show $L \cdot B_2$ is small where $L$ is any one of the matrices
$L[{\scriptstyle k,k \vt \emptyset}] $, $L[{\scriptstyle k_1,k_1,k_2 \vt \emptyset}]$,
$L[{\scriptstyle p \vt p}]$, $L[{\scriptstyle p,k \vt p}]$, $L^{\veps}[{\scriptstyle k_1,k_2 \vt (k_1,k_2]}]$,
$L^{\veps}[{\scriptstyle k_1 \vt (k_1,k_2]}]$ or $L^{\veps}[{\scriptstyle k_1,k_2, k_3 \vt (k_1,k_2]}]$.
Recall from Lemma \ref{lem:Bsquare} that $B_2$ is a weighted sum of the matrices $(\sm L)[{\scriptstyle k,k \vt \emptyset}]$.
So it suffices to prove that each of the aforementioned matrices are small when the multiplication by $B_2$
is replaced by $(\sm L)[{\scriptstyle k,k \vt \emptyset}]$.

\begin{lem} \label{lem:sLk}
Consider the matrix $\sm L[{\scriptstyle k,k \vt \emptyset}]$ and denote $F_{T,k}$ its re-scaled kernel according to \eqref{Fredholmembed}.
Set $\lambda_k = \D_k n/\nu_T = (\D_k t/c_0) \, T^{2/3} + C_{q,L} T^{1/3}$, $\D_1 = (\D_{k,r^{*}} t)^{1/3}$ and $\D_2 = (\D_{s,k}t)^{1/3}$.
The following bound holds for $F_{T,k}$.
\begin{equation*}
|F_{T,k}(r,u;s,v)| \leq \ind{s < k < r^{*}}\, C_{q,L}\, 
e^{-\mu(u+\lambda_k) + \Psi \big((u+\lambda_k)/\D_1\big)} \cdot e^{\mu(v+\lambda_k) + \Psi \big(- (v+\lambda_k)/\D_2\big)}.
\end{equation*}
\end{lem}

\begin{proof}
Let us recall $\sm L[{\scriptstyle k,k \vt \emptyset}]$ from Lemma \ref{lem:Bsquare}.
The entry $\sm L [{\scriptstyle k,k \vt \emptyset}](r,i;s,j)$ equals
\begin{equation*}
\ind{s < k < r^{*}} \frac{c(r,i;s,j)}{w_c} \intz{1} \intz{2}
\frac{(\zeta_1-\zeta_2)^{-1}}{G \big (\zeta_1 \vt i-n_{k-1}, \D_{k,r^{*}}(m,a) \big)\, G \big (\zeta_2 \vt n_{k-1}-j+1, \D_{s,k}(m,a) \big)}.
\end{equation*}

Indices $i,j$ are re-scaled according to \eqref{ijrescale}. Ignoring the rounding, this means that
\begin{align*}
& (i-n_{k-1}, \D_{k,r^{*}}(m,a))= \D_{k,r^{*}}(n,m,a) \,+\, c_0 \big((u+\lambda_k)/\D_1,0,0 \big)\cdot (\D_{k,r^{*}}t\, T)^{1/3}, \\
& (n_{k-1}-j, \D_{s,k}(m,a)) = \D_{s,k}(n,m,a) \,-\, c_0\big ((v+\lambda_k)/\D_2,0,0 \big )\cdot (\D_{s,k}t\, T)^{1/3}.
\end{align*}

We choose the $\zeta_1$-contour to be $w_0(\sigma_1, d(u+\lambda_k))$ with $K \coloneqq \D_{k,r^{*}} t\, T$.
Similarly, the $\zeta_2$-contour is $w_0(\sigma_2, d(-v-\lambda_k))$ with $K \coloneqq \D_{s,k}t\, T$.
Due to the constraint $\tau_2 < \tau_1$ we should have $d(u+\lambda_k)/\D_1 < (d(-v-\lambda_k)-1)/\D_2$.
In this case, $|\zeta_1(\sigma_1) - \zeta_2(\sigma_2)|^{-1} \leq C_{q,L} T^{1/3}$. Furthermore, with $d(\cdot)$s
chosen according to Lemma \ref{lem:Gest} we have the following estimates.
\begin{align*}
& |G(\zeta_1(\sigma_1) \vt i-n_{k-1}, \, \D_{k,r^{*}}(m,a))|^{-1} \leq C_3 e^{-C_4\sigma_1^2 + \Psi \big( (u+\lambda_k) /\D_1\big)}, \\
& |G(\zeta_2(\sigma_2) \vt n_{k-1}-j+1, \, \D_{s,k}(m,a))|^{-1} \leq C_3 e^{-C_4 \sigma_2^2 + \Psi \big(-(v+\lambda_k)/\D_2 \big)}.
\end{align*}
With these estimates, changing variables $\zeta_{\ell} \mapsto \sigma_{\ell}$ and arguing as before, we see that
\begin{align*}
|F_{T,k}(r,u;s,v)| & \leq \ind{s<k<r^{*}}\, C_{q,L}\, \int_{\R^2} d \sigma_1 d\sigma_2\, e^{-C_4 (\sigma_1^2 + \sigma_2^2)} \, \times \\
& \; e^{\mu(v-u)}\, e^{\Psi \big( (u+\lambda_k) /\D_1\big)} e^{\Psi \big(-(v+\lambda_k)/\D_2 \big)} \\
& = \ind{s<k<r^{*}}\, C_{q,L}\, e^{-\mu(u+\lambda_k) + \Psi \big((u+\lambda_k)/\D_1\big)} \cdot e^{\mu(v+\lambda_k) + \Psi \big(- (v+\lambda_k)/\D_2\big)}.
\end{align*}

It remains to order the contours. We know that if $T$ is sufficiently large in terms of $q$ and $L$, then
\begin{align*}
& d(u+\lambda_k)/ \D_1 \in \begin{cases}
[C_1\D_1^{-1}, \, C_2T^{1/3}] & \text{if}\;\; u+\lambda_k \geq 0 \\
[C_1 \D_1^{-1} + \D_1^{-3/2}\delta (u+\lambda_k)_{-}^{1/2}, \, C_2 T^{1/3}] & \text{if}\;\; u + \lambda_k < 0;
\end{cases} \\
& d(-v-\lambda_k)/\D_2 \in \begin{cases}
[C_1 \D_2^{-1}, \, C_2T^{1/3}] & \text{if}\;\; v+\lambda_k \leq 0 \\
[C_1 \D_2^{-1} + \D_2^{-3/2} \delta (v+\lambda_k)^{1/2}, \, C_2 T^{1/3}] & \text{if}\;\; v + \lambda_k > 0.
\end{cases}
\end{align*}
If $u + \lambda_k \geq 0$ then we can order the contours by first choosing $d(-v-\lambda_k)$ and then choosing $d(u+\lambda)$
accordingly from an interval with length of order $T^{1/3}$. Suppose $u +\lambda_k < 0$. Then we will first choose $d(u+\lambda_k)$
and then $d(-v-\lambda_k)$ accordingly. We are able to do so if $C_1 \D_1^{-1} + \D_1^{-3/2}\delta (u+\lambda_k)_{-}^{1/2} < C_2 T^{1/3}$.
In this regard, since $\lambda_k > 0$, $(u+\lambda_k)_{-} \leq (u)_{-}$.
Now $(u)_{-} \leq \D_{r^{*}}n/\nu_T = (\D_{r^{*}}t/c_0)\, T^{2/3} + C_{q,L}T^{1/3}$. Therefore, we are fine so long as
$\delta < C_2c_0^{1/2} \D_1^{3/2} (\D_{r^{*}}t)^{-1/2}$, which holds because $\delta$ satisfies \eqref{mudelta}.
\end{proof}

\begin{lem} \label{lem:goodsmall}
Let $M_1, M_2, \ldots$ be a sequence of good matrices where $M_n$ is $n \times n$ and
$n = n_p$ is according to \eqref{KPZscaling}. Then the sequence of matrices
$M_n \cdot \sm L[{\scriptstyle k,k \vt \emptyset}]_{n \times n}$ is small.
\end{lem}

\begin{proof}
Let $F_T$ and $F_{T,k}$ be the re-scaled kernels of $M_n$ and $\sm L[{\scriptstyle k,k\vt \emptyset}]_{n \times n}$, respectively,
via \eqref{Fredholmembed}. Let $F'_{T}$ be the one for their product. We have that
$$F'_T(r,u;s,v) = \sum_{\ell=1}^p \, \int dz\, F_T(r,u;\ell,z) F_{T,k}(\ell,z; s,v) \cdot \ind{s < k < \ell^{*}}.$$
The $z$-integral is over $\R_{<0}$ for $\ell < p$ and over $\R_{>0}$ if $\ell =p$.
Note that $\sm L[{\scriptsize k,k\vt \emptyset}]$ is non-zero only for $k < p-1$, and so we may replace $\ell^{*}$ by $\ell$ above.
It suffices to show that for every $\ell$ such that $s < k < \ell$, the corresponding $z$-integral is a small kernel in terms
of $u$ and $v$.

Fix $s,\ell$ and $k$ such that $s < k < \ell$. Let $g_1, \ldots, g_p$ be the bounded and integrable functions
over $\R$ that certify goodness of $F_T$. Recalling Lemma \ref{lem:sLk}, let $\lambda$ denote the parameter
$\lambda_k$ there. Also set $\D_1 = (\D_{k,\ell} t)^{1/3}$, $\D_2 = (\D_{s,k} t)^{1/3}$ and the function $f(z) = e^{-\mu z + \Psi(z/\D_1)}$.

First, suppose $\ell < p$. Due to goodness of $F_T$ and Lemma \ref{lem:sLk}, we infer that
$$ |\, \int_{-\infty}^{0}dz\, F_T(r,u;\ell,z) F_{T,k}(\ell,z; s,v) \, | \leq C_{q,L} \int_{-\infty}^0dz\, g_{\ell}(z) f(z+\lambda) \, \cdot \,
g_r(u) \cdot e^{\mu(v+\lambda) + \Psi \big(-(v+\lambda)/\D_2 \big)}.$$
By \eqref{psiproperty} we see that the function $e^{\mu(v+\lambda) + \Psi \big(-(v+\lambda)/\D_2 \big)}$ is bounded
and integrable over $\R$ in variable $v$. Smallness thus follows if the $z$-integral tends to zero as $T \to \infty$.
In this regard observe that for $x \geq 0$, $f(x) = e^{(\frac{\mu_2}{\D_1} - \mu)x}$, and $\frac{\mu_2}{\D_1} - \mu < 0$
since $\mu$ satisfies \eqref{mudelta}. Therefore, $\max_{x \geq B} f(x) = f(B) \to 0$ as $B \to \infty$. Also, $f$ is bounded.
Therefore,
\begin{align*}
\int_{-\infty}^0 dz\, g_{\ell}(z) f(z+\lambda) &=
\int_{-\infty}^{-\lambda/2} dz\, g_{\ell}(z) f(z+\lambda) + \int_{\lambda/2}^{\lambda}dz\, g_{\ell}(z-\lambda)f(z) \\
& \leq ||f||_{\infty}\, \int_{-\infty}^{-\lambda/2}dz\, g_{\ell}(z) + ||g_{\ell}||_{1}\, \max_{z \geq \lambda/2}\, \{f(z) \}.
\end{align*}
As $T$ goes to $\infty$ so does $\lambda$, and both the integral and maximum above tend to zero.

Now consider $\ell = p$. In this case,
\begin{align*}
&|\, \int_{0}^{\infty}dz\, F_T(r,u;\ell,z) F_{T,k}(\ell,z; s,v) \, | \leq C_{q,L} \int_{0}^{\infty} dz\, g_{\ell}(z) f(z+\lambda) \, \cdot \,
g_r(u) \cdot e^{\mu(v+\lambda) + \Psi \big(-(v+\lambda)/\D_2 \big)} \\
&\leq \underbrace{C_{q,L} \, ||g_{\ell}||_1 \cdot \max_{z \geq \lambda} \, \{f(z) \}}_{\eta_T} \cdot \,
g_r(u) \cdot e^{\mu(v+\lambda) + \Psi \big(-(v+\lambda)/\D_2 \big)}.
\end{align*}
We see that this is small as required.
\end{proof}

Lemma \ref{lem:goodsmall} implies that the matrices $L \cdot B_2$ are small where $L$ is any one of the good matrices
mentioned in the opening of this section. So this concludes the proof of Proposition \ref{prop:small}.
\smallskip

\paragraph{\textbf{Proof that the KPZ-scaling limit is a consistent family of probability distributions}}
Let $P(\xi_1, \ldots, \xi_p)$ denote the limiting expression from Theorem \ref{thm:1} as a function of the parameters $\xi_k$.
Namely, recalling $\mathbf{H}_T$ from \eqref{hxt},
$$ P(\xi_1, \ldots, \xi_p) = \lim_{T \to \infty} \pr{\mathbold{H}_T(x_1, t_1) < \xi_1, \ldots, \mathbold{H}_{T}(x_p,t_p) < \xi_p}.$$

From the discussion for the single time law we know that $P(\xi_1) = F_{GUE}(\xi_1 + x_1^2)$,
which is a probability distribution in $\xi_1$ (see \cite{JoSh, TW}). Assume that $p \geq 2$.
We need to establish that $P$ has appropriate limit values as any $\xi_k \to \pm \infty$ since the other
necessary properties are retained in the limit. Consider the parameter $\xi_1$ for concreteness.
Since $P$ is the limit of probability distribution functions,
$$P(\xi_1, \ldots, \xi_p) \leq P(\xi_1) = F_{GUE}(\xi_1 + x_1^2).$$
So as $\xi_1 \to - \infty$, $P(\xi_1, \ldots, \xi_p)$ tends to 0 as required.

Now consider the limit as $\xi_1 \to \infty$. We have
\begin{align*}
& \pr{\mathbold{H}_T(x_1, t_1) < \xi_1, \mathbold{H}_T(x_2, t_2) < \xi_2, \ldots, \mathbold{H}_{T}(x_p,t_p) < \xi_p} =
\pr{\mathbold{H}_T(x_2, t_2) < \xi_2, \ldots, \mathbold{H}_{T}(x_p,t_p) < \xi_p} \\
& - \pr{\mathbold{H}_T(x_1, t_1) \geq \xi_1, \mathbold{H}_T(x_2, t_2) < \xi_2, \ldots, \mathbold{H}_{T}(x_p,t_p) < \xi_p}.
\end{align*}
Since the first two terms above have limits, so does the third, and we find that
$$P(\xi_1, \xi_2, \ldots, \xi_p) = P(\xi_2,\ldots, \xi_p) - \bar{P}(\xi_1, \xi_2, \ldots, \xi_p),$$
where $\bar{P}$ is the limit of the third term. Moreover, $\bar{P}(\xi_1, \ldots, \xi_p) \leq 1 - F_{GUE}(\xi_1 + x_1^2)$
since the corresponding pre-limit inequality holds. It follows that $P(\xi_1, \ldots, \xi_p)$ tends to
$P(\xi_2, \ldots, \xi_p)$ as $\xi_1 \to \infty$. This shows that the KPZ-scaling limit provides a consistent
family of probability distribution functions. It also concludes the proof of Theorem \ref{thm:1}.


\begin{thebibliography}{plain}

   \bibitem{BaLi}
   J.~Baik and Z.~Liu.
   \newblock \emph{Multi-point distribution of periodic TASEP}.
   \newblock J.~Amer.~Math.~Soc. 32:609--674, 2019.
   \newblock \arxiv{1710.03284}
   
   \bibitem{BaGa}
   R.~Basu and S.~Ganguly.
   \newblock \emph{Time correlation exponents in last passage percolation}.
   \newblock preprint, 2018.
   \newblock \arxiv{1807.09260}
	
	\bibitem{BFPS}
	A.~Borodin, P.~L.~Ferrari, M.~Prah\"{o}fer and T.~Sasamoto.
	\newblock \emph{Fluctuation properties of the TASEP with periodic initial configurations}.
	\newblock J.~Stat.~Phys. 129:1055--1080, 2007.
	\newblock \arxivmaph{0608056}
	
	\bibitem{BFS}
	A.~Borodin, P.~L.~Ferrari and T.~Sasamoto.
	\newblock \emph{Large time asymptotics of growth models on space-like paths II: PNG and parallel TASEP}.
	\newblock Comm.~Math.~Phys. 283:417--449, 2008.
	\newblock \arxiv{0707.4207}
	
	\bibitem{BG}
	A.~Borodin and V.~Gorin.
	\newblock \emph{Lectures on integrable probability}.
	\newblock In Probability and Statistical Physics in St.~Petersburg,
	Proceedings of Symposia in Pure Mathematics, volume 91, pp.~155--214, 2016.
	\newblock \arxiv{1212.3351}
	
	\bibitem{BGW}
	A.~Borodin, V.~Gorin and M.~Wheeler.
	\newblock \emph{Shift--invariance for vertex models and polymers}
	\newblock preprint, 2019.
	\newblock \arxiv{1912.02957}
	
	\bibitem{CoKPZ}
	I.~Corwin.
	\newblock \emph{The Kardar-Parisi-Zhang equation and universality class}.
	\newblock Random Matrices Theory Appl. 1(1):1130001, 2012.
	\newblock \arxiv{1106.1596}
	
	\bibitem{CFP}
	I.~Corwin, P.~L.~Ferrari and S.~P\'ech\'e.
	\newblock \emph{Universality of slow de-correlation in KPZ growth}.
	\newblock Ann.~Inst.~Henri Poincar\'{e} Probab.~Stat. 48:134--150, 2012.
	\newblock \arxiv{1001.5345}
	
	\bibitem{CH}
	I.~Corwin and A.~Hammond.
	\newblock \emph{Brownian Gibbs property for Airy line ensembles}.
	\newblock Invent.~Math. 195:441--508, 2014.
	\newblock \arxiv{1108.2291}
	
	\bibitem{DOV}
	D.~Dauvergne, J.~Ortmann and B.~Vir\'{a}g.
	\newblock \emph{The directed landscape}.
	\newblock preprint, 2018.
	\newblock \arxiv{1812.00309}
	
	\bibitem{NarDou} 
	J.~De Nardis and P.~Le Doussal. 
	\newblock \emph{Tail of the two-time height distribution for KPZ growth in one dimension}.
	\newblock J. Stat. Mech. Theory Exp. 053212, 2017.
	\newblock \arxiv{1612.08695}

    \bibitem{NaDoTa} J.~De Nardis, P.~Le Doussal and K. A.~Takeuchi.
    \newblock \emph{Memory and universality in interface growth}.
    \newblock Phys. Rev. Lett. 118: 125701, 2017.
    \newblock \arxiv{1611.04756}

    \bibitem{dNLD} J.~De Nardis and P.~Le Doussal. 
    \newblock \emph{Two-time height distribution for 1D KPZ growth: the recent exact result and its tail via replica}. 
    \newblock J. Stat. Mech. 093203, 2018.
    \newblock \arxiv{1804.01948}
	
	\bibitem{DNJLD}
	J.~De Nardis, K.~Johansson and P.~Le Doussal.
	\newblock In preparation.
	
	\bibitem{DiWa}
	A.~B.~Dieker and J.~Warren.
	\newblock \emph{Determinantal transition kernels for some interacting particles on the line}.
	\newblock Ann.~Inst.~Henri Poincar\'{e} Probab.~Stat. 44(6):1162--1172, 2008.
	\newblock \arxiv{0707.1843}
	
	\bibitem{Fe}
	P.~L.~Ferrari.
	\newblock \emph{Slow decorrelations in KPZ growth}.
	\newblock J.~Stat.~Mech.~Theory Exp. P07022, 2008.
	\newblock \arxiv{0806.1350}
	
	\bibitem{FO}
	P.~L.~Ferrari and A.~Occelli.
	\newblock \emph{Time-time covariance for last passage percolation with generic initial profile}.
	\newblock Math.~Phys.~Anal.~Geom. 22:1, 2019.
	\newblock \arxiv{1807.02982}
	
	\bibitem{FSW}
	P.~L.~Ferrari, H.~Spohn and T.~Weiss.
	\newblock \emph{Reflected Brownian motions in the KPZ universality class}.
	\newblock In SpringerBriefs in Mathematical Physics, volume 18, 2017.
	\newblock \arxiv{1702.03910}
	
	\bibitem{Hm}
	A.~Hammond.
	\newblock \emph{Brownian regularity for the Airy line ensemble, and multi-polymer watermelons in Brownian last passage percolation}.
	\newblock Mem.~Amer.~Math.~Soc. (to appear).
	\newblock \arxiv{1609.02971}
	
	\bibitem{IS}
	T.~Imamura and T.~Sasamoto
	\newblock \emph{Dynamics of a tagged particle in the asymmetric exclusion process with the step initial condition}.
	\newblock J.~Stat.~Phys. 128: 799–846, 2007.
	\newblock \arxivmaph{0702009}
	
	\bibitem{JoMar}
	K.~Johansson.
	\newblock \emph{A multi-dimensional Markov chain and the Meixner ensemble}.
	\newblock Ark.~Mat. 48:437--476, 2010.
	\newblock \arxiv{0707.0098}
	
	\bibitem{JoTwo}
	K.~Johansson.
	\newblock \emph{The two-time distribution in geometric last-passage percolation}.
	\newblock Probab.~Theory Relat.~Fields 175:849--895, 2019.
	\newblock \arxiv{1802.00729}
	
	\bibitem{JoLim}
	K.~Johansson.
	\newblock \emph{The long and short time asymptotics of the two-time distribution in local random growth}.
	\newblock preprint, 2019.
	\newblock \arxiv{1904.08195}
	
	\bibitem{JoSh}
	K.~Johansson.
	\newblock \emph{Shape fluctuations and random matrices}.
	\newblock Comm.~Math.~Phys. 209:437--476, 2000.
	\newblock \arxivmath{9903134}
	
	\bibitem{JoDPG}
	K.~Johansson.
	\newblock \emph{Discrete polynuclear growth and determinantal processes}.
	\newblock Comm.~Math.~Phys. 242:277--295, 2003.
	\newblock \arxivmath{0206208}
	
	\bibitem{KPZ}
	M.~Kardar, G.~Parisi and Y.-C.~Zhang.
	\newblock \emph{Dynamic scaling of growing interfaces}.
	\newblock Phys.~Rev.~Letts. 56:889--892, 1986.
	
	\bibitem{Kr}
	M.~Kardar.
	\newblock \emph{An underlying link for scaling of fluctuations in growth fronts, fracture lines, strong localization, $\cdots$}.
	\newblock Seminar slides from \url{https://www.mit.edu/~kardar/research/seminars/Growth/talks/Kyoto/introduction.html}
	
	\bibitem{KS}
	J.~Krug and H.~Spohn.
	\newblock \emph{Kinetic Roughening of Growing Interfaces}.
	\newblock In Solids far from Equilibrium: Growth, Morphology and Defects, ed. by C.~Godr\`{e}che.
	\newblock Cambridge University Press, pp. 479--582, 1992.
	
	\bibitem{Liu}
	Z.~Liu
	\newblock \emph{Multi-time distribution of TASEP}.
	\newblock preprint, 2019.
	\newblock \arxiv{1907.09876}.
	
	\bibitem{MQR}
	K.~Matetski, J.~Quastel and D.~Remenik.
	\newblock \emph{The {KPZ} fixed point}.
	\newblock preprint, 2017.
	\newblock \arxiv{1701.00018}
	
	\bibitem{OC}
	N.~O'Connell.
	\newblock \emph{Conditioned random walks and the RSK correspondence}.
	\newblock J.~Phys.~A: Math.~Gen. 36:3049–3066, 2003.
	
	\bibitem{PS}
	M.~Pr\"{a}hofer and H.~Spohn.
	\newblock \emph{Scale invariance of the PNG droplet and the Airy process}.
	\newblock J.~Stat.~Phys. 108:1071--1106, 2002.
	\newblock \arxivmath{0105240}
	
	\bibitem{QuKPZ}
	J.~Quastel.
	\newblock \emph{Introduction to KPZ}.
	\newblock In Current Developments in Mathematics. International Press of Boston, Inc., 2011.
	
	\bibitem{RS}
	A.~R\'{a}kos and G.~M.~Sch\"{u}tz.
	\newblock \emph{Current distribution and random matrix ensembles for an integrable asymmetric fragmentation process}.
	\newblock J.~Stat.~Phys. 118:511--530, 2005.
	\newblock \arxivcnma{0405464}
	
	\bibitem{ST}
	 M.~Sano and K.~A.~Takeuchi.
	 \newblock \emph{Evidence for geometry-dependent universal fluctuations of the Kardar-Parisi-Zhang interfaces in liquid-crystal turbulence}.
	 \newblock J.~Stat.~Phys. 147:853--890, 2012.
	 \newblock \arxiv{1203.2530}
	 
	\bibitem{Sas}
	T.~Sasamoto.
	\newblock \emph{Spatial correlations of the 1D KPZ surface on a flat substrate}.
	\newblock J.~Phys.~A 38(33): L549--L556, 2005.
	\newblock \arxivcnma{0504417}
	
	\bibitem{Sch}
	G.~M.~Sch\"{u}tz.
	\newblock \emph{Exact solution of the master equation for the asymmetric exclusion process}.
	\newblock J.~Stat.~Phys. 88:427--445, 1997.
	\newblock \arxivcnma{9701019}
	
	\bibitem{SeCG}
	T.~Sepp\"{a}l\"{a}inen.
	\newblock \emph{Lecture notes on the corner growth model}, 2009.
	\newblock Available from \url{https://www.math.wisc.edu/~seppalai/cornergrowth-book/ajo.pdf}.
	
	\bibitem{TW}
	C.~A.~Tracy and H.~Widom.
	\newblock \emph{Level-spacing distributions and the Airy kernel}.
	\newblock Comm.~Math.~Phys. 159:151--174,1994.
	\newblock \href{https://arxiv.org/abs/hep-th/9211141}{\texttt{ArXiv:9211141}}
	
	\bibitem{Warr}
	J.~Warren.
	\newblock \emph{Dyson's Brownian motions, intertwining and interlacing}.
	\newblock Electron.~J.~Probab. 12(19):573--590, 2007.
	\newblock \arxivmath{0509720}
	
\end{thebibliography}
\end{document}